\documentclass[11pt]{amsart}
\usepackage[margin=1in]{geometry}
\usepackage{amssymb,amsmath,amsthm,thmtools}

\declaretheorem[name=Theorem,numberwithin=section]{theorem}
\declaretheorem[name=Proposition,sibling=theorem]{proposition}
\declaretheorem[name=Lemma,sibling=theorem]{lemma}
\declaretheorem[name=Corollary,sibling=theorem]{corollary}

\PassOptionsToPackage{numbers}{natbib}
\usepackage[numbers]{natbib}

\usepackage[most]{tcolorbox}

\usepackage[utf8x]{inputenc}

\usepackage{epsfig}
\usepackage{color}
\graphicspath{{fig/}} 
\usepackage{pgfplotstable}

\usepackage{graphicx}
\usepackage{constants}
\usepackage{wrapfig}
\usepackage{lscape}
\usepackage{rotating}
\usepackage{epstopdf}

\usepackage{amsmath}
\usepackage{amsthm}

\usepackage{hyperref}
\usepackage{cleveref}
\hypersetup{breaklinks=true,
    linkcolor=blue,
    citecolor=blue,
    colorlinks=true,
pdfborder={0 0 0}}

\usepackage{thmtools}
\usepackage{thm-restate}

\usepackage{latexsym}
\usepackage{graphicx}
\usepackage{amsfonts}
\usepackage{amssymb}
\usepackage{mathrsfs}
\usepackage[lofdepth,lotdepth]{subfig}
\usepackage{multirow}
\usepackage{booktabs}
\usepackage{lscape}

\usepackage{thmtools}
\usepackage{thm-restate}

\numberwithin{equation}{section}
\numberwithin{equation}{section}

\def\argmin{\mathop{\rm arg\, min}}

\newcommand\onArxivOrInJournal[2]{#2}

\newcommand{\bel}{\begin{eqnarray}\label}
\newcommand{\eel}{\end{eqnarray}}
\newcommand{\bes}{\begin{eqnarray*}}
\newcommand{\ees}{\end{eqnarray*}}
\newcommand{\bei}{\begin{itemize}}
\newcommand{\eei}{\end{itemize}}
\newcommand{\beiftnt}{\begin{itemize}\footnotesize}

\def\constants{\gamma,\mu,\mu_g,\varphi,a_*}

\def\argmin{\mathop{\rm arg\, min}}
\def\real{{\mathbb{R}}}
\def\R{{\real}}

\def\E{{\mathbb{E}}}

\def\P{{\mathbb{P}}}

\def\complex{\mathop{{\rm I}\kern-.58em\hbox{\rm C}}\nolimits}

\def\diag{\hbox{\rm diag}}

\def\Rem{\hbox{\rm Rem}}

\DeclareMathOperator{\trace}{Tr}

\def\mathbold{\boldsymbol} 

\def\ba{\mathbold{a}}

\def\bA{\mathbold{A}}
\def\hbA{{\widehat{\bA}}}\def\tbA{{\widetilde{\bA}}}

\def\bb{\mathbold{b}}
\def\hbb{{\widehat{\bb}}}\def\tbb{{\widetilde{\bb}}}
\def\bB{\mathbold{B}}

\def\bc{\mathbold{c}}

\def\bD{\mathbold{D}}

\def\bff{\mathbold{f}}
\def\tbf{{\widetilde{\bff}}}
\def\bF{\mathbold{F}}

\def\bG{\mathbold{G}}

\def\bh{\mathbold{h}}
\def\tbh{{\widetilde{\bh}}}
\def\bH{{\mathbold{H}}}

\def\bI{\mathbold{I}}

\def\bM{\mathbold{M}}

\def\bP{\mathbold{P}}

\def\bQ{\mathbold{Q}}\def\hbQ{{\widehat{\bQ}}}

\def\br{\mathbold{r}}
\def\tbr{{\widetilde{\br}}}
\def\bR{\mathbold{R}}

\def\be{\mathbold{e}}

\def\bu{\mathbold{u}}

\def\bU{\mathbold{U}}

\def\bv{\mathbold{v}}

\def\bV{\mathbold{V}}

\def\bw{\mathbold{w}}

\def\bW{{\mathbold{W}}}

\def\bx{\mathbold{x}}

\def\bX{\mathbold{X}}\def\tbX{{\widetilde{\bX}}}\def\bXbar{{\overline \bX}}

\def\by{\mathbold{y}}
\def\tby{{\widetilde{\by}}}\def\bybar{{\overline \by}}

\def\bz{\mathbold{z}}
\def\tbz{{\widetilde{\bz}}}
\def\bZ{\mathbold{Z}}

\def\balpha{\mathbold{\alpha}}

\def\hbalpha{{\widehat{\balpha}}}

\def\bbeta{\mathbold{\beta}}\def\hbeta{\widehat{\beta}}

\def\hbbeta{{{\widehat{\bbeta}}}}
\def\tbbeta{\widetilde{\bbeta}}
\def\bbetabar{{\overline\bbeta}}

\def\ep{\varepsilon}\def\eps{\epsilon}
\def\bep{ {\mathbold{\ep} }}

\def\tbep{{\widetilde{\bep}}}

\def\bfeta{\mathbold{\eta}}

\def\btheta{\mathbold{\theta}}\def\htheta{\widehat{\theta}}

\def\hbtheta{{\widehat{\btheta}}}\def\tbtheta{{\widetilde{\btheta}}}

\def\bmu{\mathbold{\mu}}

\def\brho{\mathbold{\rho}}

\def\brhobar{{\overline \brho}}

\def\bSigma{\mathbold{\Sigma}}

\def\bSigmabar{{\overline\bSigma}}

\def\bpsi{\mathbold{\psi}}\def\hpsi{\widehat{\psi}}

\def\hbpsi{{\widehat{\bpsi}}}\def\tbpsi{{\widetilde{\bpsi}}}

\declaretheorem[name=Assumption,numberwithin=section]{assumption}

\crefname{proposition}{proposition}{propositions}
\crefname{corollary}{corollary}{corollaries}
\crefname{example}{exapmle}{examples}
\crefname{remark}{remark}{remarks}
\crefname{lemma}{lemma}{lemmas}
\crefname{assumption}{assumption}{assumptions}
\crefname{table}{table}{tables}

\def\argmin{\mathop{\rm arg\, min}}

\def\SURE{\widehat{\rm{\scriptstyle SURE}}}

\def\defas{\stackrel{\text{\tiny def}}{=}}

\usepackage{enumitem}
\DeclareMathOperator{\dv}{div}

\def\df{{\hat{\mathsf{df}}}{}}
\def\tdf{{\tilde{\mathsf{df}}}{}}

\def\RHS{{{\mathrm{RHS}}}{}}

\usepackage{soul}

\usepackage[textsize=scriptsize,textwidth=1.7in]{todonotes}
\usepackage{marginnote}

\newcommand\pb[1]{{ #1 }}

\usepackage[color]{changebar} 
\cbcolor{magenta}
\title{Out-of-sample error estimation for M-estimators with convex penalty}
\author{Pierre C Bellec}
\begin{document}
\begin{abstract}
{
A generic out-of-sample error estimate is proposed for $M$-estimators
regularized with a convex penalty in high-dimensional linear regression where
$(\bX,\by)$ is observed and the dimension $p$ and sample size $n$ are of the same
order.  The out-of-sample error estimate enjoys a relative error of order
$n^{-1/2}$ in a linear model with Gaussian covariates and independent noise,
either non-asymptotically when $p/n\le \gamma$ or asymptotically in the
high-dimensional asymptotic regime $p/n\to\gamma'\in(0,\infty)$.  General
differentiable loss functions $\rho$ are allowed provided that the derivative
of the loss is 1-Lipschitz; this includes the least-squares loss as well as
robust losses such as the Huber loss and its smoothed versions.  The validity
of the out-of-sample error estimate holds either under a strong convexity
assumption, or for the L1-penalized Huber M-estimator and the Lasso under a
sparsity assumption and a bound on the number of contaminated observations.

For the square loss and in the absence of corruption in the response, the
results additionally yield $n^{-1/2}$-consistent estimates of the noise
variance and of the generalization error. This generalizes, to arbitrary
convex penalty and arbitrary covariance, estimates that were previously known
for the Lasso.
}
{
    \onArxivOrInJournal{}{\\\emph{Keywords: }}
    M-estimators,
    out-of-sample error estimation,
    parameter tuning,
    regularization,
    Huber loss,
    robustness.
}
\end{abstract}
\onArxivOrInJournal{}{\maketitle}


\section{Introduction}

Consider a linear model 
\begin{equation}
    \label{LM}
    \by=\bX\bbeta+\bep
\end{equation}
where $\bX\in\R^{n\times p}$ has iid $N({\mathbf 0},\bSigma)$ rows and $\bep\in\R^n$
is a noise vector independent of $\bX$.
The entries of $\bep$ may be heavy-tailed, for instance
with infinite second moment,
or follow Huber's contamination model
with $\eps_i$ iid with cumulative distribution function (cdf)
$F(u) = (1-q)\P(N(0,\sigma^2)\le u) + q G(u)$
where $q\in[0,1]$ is the proportion of corrupted entries
and $G$ is an arbitrary cdf chosen by an adversary.
Since the seminal work of Huber in \cite{huber1964robust},
a popular means to handle heavily-tails
or corruption of certain entries of $\bep$
is based
on robust loss functions $\rho:\R \to [0,+\infty)$ to construct $M$-estimators $\hbbeta$
by minimization of optimization problems of the form
$
    \hbbeta \in \argmin_{\bb\in\R^p}
    \frac 1 n 
    \sum_{i=1}^n \rho(y_i - \bx_i^\top\bb)
$
where $(\bx_i)_{i=1,...,n}$ are the rows of $\bX$.
Robustness against corruption of the above estimator typically
requires the convex loss $\rho$ to grow linearly at $\pm \infty$
and a well-studied example is the Huber loss $\rho_H(u)=\int_0^{|u|}\min(t,1)dt$.

As we are interested in the high-dimensional regime where $p$
is potentially larger than $n$, we also allow for convex penalty functions
to leverage structure
in the signal $\bbeta$ and fight the curse of dimensionality.
The central object of the present paper is thus
a penalized robust $M$-estimator of the form
\begin{equation}
    \label{M-estimator-rho}
    \hbbeta(\by,\bX) \in \argmin_{\bb\in\R^p}
    \Big(
        \frac 1 n 
        \sum_{i=1}^n \rho(y_i - \bx_i^\top\bb)
        + g(\bb)
    \Big)
\end{equation}
where $\rho:\R \to\R$ is a convex differentiable loss function,
and $g:\R^p\to\R$ is a convex penalty.
We may write simply $\hbbeta$ for $\hbbeta(\by,\bX)$
if the context is clear.

The main contribution of the present paper is the introduction
of a generic out-of-sample error estimate for penalized $M$-estimators
of the form \eqref{M-estimator-rho}.
Here, the out-of-sample error refers to the random quantity
\begin{equation}
\|\bSigma^{\frac12}(\hbbeta-\bbeta)\|^2
= \E[((\hbbeta-\bbeta)^\top\bx_{new} )^2 ~ | ~ (\bX,\by)]
\label{MSE}
\end{equation}
where $\bx_{new}$ is independent of the data $(\bX,\by)$ with the
same distribution as any row of $\bX$.
Our goal is to develop such out-of-sample error estimate
for $\hbbeta$ in \eqref{M-estimator-rho}
with little or no assumption on the robust loss $\rho$
and the convex penalty $g$, in order to allow broad choices by practitioners
for $(\rho,g)$.

We consider the high-dimensional regime where $p$ and
$n$ are of the same order. The results of the present paper
are non-asymptotic and assume that $p/n \le \gamma\in(0,\infty)$
for some fixed constant $\gamma$ independent of $n,p$.
Although non-asymptotic, these results are applicable
in the regime where $n$ and $p$ diverge such that 
\begin{equation}
    \label{eq:regime-gamma-prime}
    p/n\to \gamma'
    ,
\end{equation}
simply by considering a constant $\gamma>\gamma'$.
The analysis of the performance of convex estimators
in the asymptotic regime \eqref{eq:regime-gamma-prime}
has received considerable attention in the last few years
in the statistics, machine learning, electrical engineering
and statistical physics communities.
Most results available in the $p/n\to\gamma'$ literature regarding
$M$-estimators are
either based on
Approximate Message Passing (AMP) \cite{bayati2012lasso,donoho2016high,bradic2015robustness,wang2017bridge,celentano2019fundamental,gerbelot2020asymptotic}
following the pioneering work \cite{donoho2009message}
in compressed sensing problems,
on leave-one-out methods \cite{el_karoui2013robust,bean2013optimal,karoui2013asymptotic,el_karoui2018impact},
or on the Gordon's Gaussian min-max theorem (GMT) 
\cite{stojnic2013framework,
thrampoulidis2015lasso,
thrampoulidis2018precise,miolane2018distribution}.
The goal of these techniques is to
summarize the performance and behavior of the $M$-estimator $\hbbeta$
by a system of nonlinear equations with up to six unknown scalars
(e.g., the system of \cite{el_karoui2013robust} with unknowns
$(r,c)$ for unregularized robust $M$-estimators,
the system with unknowns $(\tau,\beta)$ of \cite[Proposition 3.1]{miolane2018distribution}
for the Lasso which dates back to \cite{bayati2012lasso},
the system with unknowns $(\tau,\lambda)$ of \cite[Section 4]{celentano2019fundamental} for permutation-invariant penalties,
or recently the system with six unknowns
$(\alpha,\sigma,\gamma,\theta,\tau,r)$
of
\cite{salehi2019impact} in regularized logistic regression).
Solving these nonlinear equations provide information about the risk
$\|\hbbeta-\bbeta\|$,
and in certain cases asymptotic normality results for the coefficients of $\hbbeta$
after a bias correction, see
\cite[e.g., Proposition B.3(iii)]{celentano2019fundamental}.
These systems of nonlinear equations depend the true coefficient vector
$\bbeta$ and the knowledge of $\bbeta$ or its limiting emprical
distribution is required to compute
the solutions. 
For Ridge regression, results can be obtained
using random matrix theory tools such as the Stieljes transform
and limiting spectral distributions of certain
random matrices
\cite{dicker2016ridge,dobriban2018high}.
Our results are of a different nature, as they do not involve solving systems
of nonlinear equations or their solutions.
Instead, our results relate fully data-driven quantities
to the out-of-sample error \eqref{MSE}. 

Additionally, most of the aforementioned works 
require isotropic design ($\bSigma=\bI_p$),
although there are notable exceptions for specific examples:
isotropy can be relaxed for Ridge regularization \cite{dobriban2018high},
in unregularized logistic regression \cite{zhao2020asymptotic}
and for the Lasso in sparse linear regression \cite{celentano2020lasso}.
The techniques developed in the present paper
do not rely on isotropy: general $\bSigma\ne\bI_p$ is allowed
without additional complexity.

We assume throughout the paper that $\rho$ is differentiable
and denote by $\psi:\R\to\R$ the derivative of $\rho$.
We also assume that $\psi$ is absolutely continuous and denote by
$\psi'$ its derivative. The functions $\psi,\psi'$
act componentwise when applied to vectors,
for instance $\psi(\by-\bX\hbbeta)=(\psi(y_i-\bx_i^\top\hbbeta))_{i=1,...,n}$.
Throughout, $\|\cdot\|$ is the Euclidean norm.

\paragraph{Contributions.}

The paper introduces a novel data-driven
        estimate of the out-of-sample 
        error \eqref{MSE}.
        The estimate depends on
        the data only through
        $\hbpsi \defas \psi(\by-\bX\hbbeta)$, the vector
        $\bSigma^{-\frac12}\bX^\top\hbpsi$
        and the derivatives of
        $\by\mapsto \hbbeta$
        and $\by \mapsto \psi(\by-\bX\hbbeta)$
        for fixed $\bX$.
        For certain choices of $(\rho,g)$ these derivatives
        have closed forms, for instance in the case of
        the $\ell_1$-penalized Huber $M$-estimator
        when $\rho$ is the Huber loss,
        the estimator $\hat R$ of the out-of-sample error
        \eqref{MSE} is
        $$\hat R
        = \big(|\hat I| - |\hat S|)^{-2}
        \big\{
            \|\psi(\by-\bX\hbbeta)\|^2
            \big(2|\hat S| -p\big)
            + \|\bSigma^{-\frac12}\bX^\top
               \psi(\by-\bX\hbbeta)
            \|^2
        \big\}
        $$
        where $\hat S =\{j\in[p]:\hbeta_j\ne 0\}$ is the active set
        and $\hat I = \{i\in [n]: \psi'(y_i-\bx_i^\top\hbbeta) > 0\}$
        is the set of inliers.
        Explicit formulae are also available for
        the Elastic-Net penalty $g(\bb)=\lambda\|\bb\|_1 + \mu\|\bb\|_2$
        for any loss $\rho$.
        For general choices of $(\rho,g)$, the derivatives
        can be approximated by a Monte Carlo scheme
        (cf. \Cref{sec:divergence}).

        The estimate is valid
        under mild assumptions, namely:
        $\psi$ is $1$-Lipschitz,
        $p/n\le \gamma$ for some constant $\gamma$ independent of $n,p$
        and that either (i) the penalty function
        $g$ is $\mu$-strongly convex,
        (ii) the loss $\rho$ is strongly convex and $\gamma<1$,
        (iii.a) $\smash{\hbbeta}$ is the Lasso with square loss
        with a sparse $\bbeta$,
        or (iii.b) $\smash{\hbbeta}$ is the $\ell_1$ penalized Huber $M$-estimator
        together with an additional assumption on the fraction of corrupted
        observations and sparsity of $\bbeta$.

        The proof arguments for the main theorem in \Cref{sec:main} below
        provide new avenues to study
        $M$-estimators in the regime $p/n\to\gamma'$.
        The results rely on novel moment
        inequalities (cf. \Cref{big_cor} below) that let us directly bound
        the difference between 
        quantities of interest (e.g., the out of sample error)
        and their estimates.
        These new techniques do not overlap with
        arguments typically used to analyse $M$-estimators when
        $p/n\to\gamma'$ such as
        Approximate Message Passing (AMP)
\cite{bayati2012lasso,donoho2016high,bradic2015robustness,wang2017bridge,celentano2019fundamental,gerbelot2020asymptotic},
        or the Gordon's Gaussian Min-Max Theorem (GMT)
\cite{stojnic2013framework,
thrampoulidis2015lasso,
thrampoulidis2018precise,miolane2018distribution}.

    In the special case of the square loss, our estimate of the out-of-sample error
        coincides with previous estimates
        for the Ordinary Least-Squares \cite{leeb2008evaluation},
        for the Lasso
        \cite{bayati2012lasso,bayati2013estimating,miolane2018distribution}
        and for $\hbbeta=\mathbf{0}$ \cite{dicker2014variance}.
        Our results can be seen
        as a broad generalization
        of these estimates to
        (a) arbitrary covariance,
        (b) general loss function, including robust losses, and
        (c) general convex penalty.
        For the square loss, our results also yield
        generic estimates for the noise level
        and the generalization error
        $\E[(\bx_{new}^\top\hbbeta - Y_{new})^2 | (\bX,\by)]$.
        Most comparable to our results are the estimates of 
        out-of-sample errors and other out-of-sample metrics studied
        in \cite{rad2018scalable,rad2020error,wang2018approximate}. However,
        the accuracy of the estimates in these works is only 
        guaranteed for smooth penalty functions \cite[Theorem 3, Assumption 6]{rad2018scalable},
        which excludes the $\ell_1$-penalty, the Elastic-Net
        and the nuclear norm as regularizers.

\paragraph{Organization.}
\Cref{sec:main} is devoted to the out-of-sample estimate $\hat R$,
the proof of its consistency, and
explicit formulae 
for specific loss and penalty function commonly used in high-dimensional and
robust statistics.
\Cref{sec:square_loss} is devoted to the square loss for which
additional results are available regarding estimation of the noise level
and the generalization error.
\Cref{sec:derivatives-6,sec:gradient-identities} derive several Lipschitz properties
to ensure existence of the derivatives as well as useful
gradient identities for M-estimators.
\Cref{sec:6-identitities-gaussian,sec:proof_chi2} provide the main
probabilistic results used in the paper.

\paragraph{Notation}

The abbreviation a.s. means almost surely.
Let $\bI_d$ be the identity matrix of size $d\times d$.
For any $p\ge 1$, let $[p]$ be the set $\{1,...,p\}$.
Let $\|\cdot\|$ be the Euclidean norm
and $\|\cdot\|_q$ the $\ell_q$ norm of vectors.
Let $\|\cdot\|_{op}$ be the operator norm (largest singular value),
$\|\cdot\|_F$ the Frobenius norm.
If $\bM$ is positive semi-definite we also use the notation
$\phi_{\max}(\bM)$ and $\phi_{\min}(\bM)$ for the largest and smallest
eigenvalue.
For any event $\Omega$, denote by $I\{\Omega\}$ its indicator function. 
For $a\in\R$, $a_+=\max(0,a)$.
Throughout the paper, we use $\C, \C,...$ to denote positive absolute constants,
$\C(\gamma),\C(\gamma),...$ to denote constants that depend on $\gamma$ only
and for instance $\C(\constants)$ to denote a constant that depend on
$\{\constants\}$ only. 

Canonical basis vectors are denoted by $(\be_i)_{i=1,...,n}$
or $(\be_l)_{l=1,...,n}$
for the canonical basis in $\R^n$, and
by $(\be_j)_{j=1,...,p}$ or $(\be_k)_{k=1,...,p}$
for the canonical basis vectors in $\R^p$.
Indices $i$ and $l$ are used to loop or sum over $[n]=\{1,...,n\}$ only,
while 
indices $j$ and $k$ are used to loop or sum over $[p]=\{1,...,p\}$ only.
This lets us use the notation $\be_i,\be_l, \be_j, \be_k$ for canonical
basis vectors in $\R^n$ or $\R^p$ as the index reveals without
ambiguity whether the canonical basis vector lies in $\R^n$ or $\R^p$.

We will refer to Frechet differentiability for the usual notion of differentiability, i.e., $f(\bx + \bh) = f(\bx) + \nabla f(\bx)^\top \bh + o(\|\bh\|)$.
This is stronger
than, e.g., Gateaux differentiability for which linearity is not required.

\section{Main result}
\label{sec:main}

Throughout, $\hbbeta$ is the estimator \eqref{M-estimator-rho}
with loss $\rho:\R\to\R$ and penalty $g:\R^p\to \R$.
The goal of this section is to develop a generic
estimator $\hat R$ for the Out-of-sample error $\|\bSigma^{1/2}(\hbbeta-\bbeta)\|^2$.

\subsection{Assumptions}

\begin{assumption}[Loss function]
    \label{assum:rho}
    The loss $\rho$ is convex and differentiable, and
    $\psi=\rho'$ is $1$-Lipschitz with derivative denoted
    by $\psi'$ where the derivative exists.
\end{assumption}
This allows for a large class of non-robust and robust loss functions,
including the square loss $\rho(u)=u^2/2$,
the Huber loss
$\rho_H(u) = \int_0^{|u|}\min(t,1)dt$ as well as
smoothed versions of $\rho_H$.
Since $\psi$ is $1$-Lipschitz, $\psi'$ exist almost everywhere
thanks to Rademacher's theorem.
Loss functions typically require a scaling parameter
that depends on the noise level to obtain
satisfactory risk bounds, see
\cite{dalalyan2019outlier} and the references therein.
For instance we consider in the main result below the loss 
\begin{equation}
    \label{eq:loss-scaling}
\rho(u)=\Lambda_*^2\rho_H\big(\Lambda_*^{-1}u\big)
\end{equation}
where $\rho_H$ is the Huber loss
and $\Lambda_*>0$ is a scaling parameter.
Since for the Huber loss $\psi_H =\rho_H'$ is 1-Lipschitz,
$\psi(u)=\rho'(u) = \Lambda_*\psi_H(\Lambda_*^{-1} u )$ is
also 1-Lipschitz.
In short, scaling a given loss with a tuning parameter
$\Lambda_*$ as in \eqref{eq:loss-scaling} does not change
the Lipschitz constant of the first derivative of $\rho$,
and the above assumption does not prevent using
a scaling parameter $\Lambda_*$.
Additionally, if the desired loss is such that $\psi$ is $L$-Lipschitz for some constant
$L\ne 1$, one may replace $(\rho,g)$ by $(L^{-1}\rho,L^{-1} g)$
to obtain a 1-Lipschitz loss without changing the value of $\hbbeta$
in \eqref{M-estimator-rho}.

\begin{assumption}[Probability distribution]
    \label{assum:Sigma}
    The rows of $\bX$ are iid $N(\mathbf{0}, \bSigma)$
    with $\bSigma$ invertible,
    $\bep$ is independent of $\bX$,
    and $(\bX,\by)$ has continuous distribution.
\end{assumption}

The Gaussian assumption is admittedly the strongest assumption
required in this work. However arbitrary covariance $\bSigma$ is allowed,
while a large body of related literature requires $\bSigma$ proportional
to identity, see for instance \cite{bayati2013estimating,bradic2015robustness,celentano2019fundamental,salehi2019impact}. Allowing arbitrary $\bSigma$ 
together with general penalty functions is made possible
by developing new techniques that are of a different nature
than this previous literature;
see the proof in \Cref{sec:proof-sketch,sec:proof-main-results}.
We require that $(\bX,\by)$ has continuous distribution
in order to ensure that derivatives of certain Lipschitz functions
of $(\by,\bX)$
exist with probability one, again by Rademacher's theorem.
If $(\bX,\by)$ does not have continuous distribution, one can
always replace $\by$ with $\tby = \by + a \tbz$
where $a$ is very small and $\tbz\sim N(\mathbf{0},\bI_n)$ is sampled
independently of $(\bep,\bX)$. Hence the continuous distribution
assumption is a mild technicality.

\begin{assumption}[Penalty]
    \label{assum:g}
    Assume either one of the following:
    \begin{enumerate}
        \item $p/n\le \gamma\in (0,+\infty)$
            and the penalty $g$ is $\mu>0$ strongly convex with respect to $\bSigma$,
            in the sense
            that any $\bb,\bb'\in\R^p$,
            $\mathsf{d}\in\partial g(\bb)$ and
            $\mathsf{d}'\in\partial g(\bb')$
            satisfy
            $(\mathsf d - \mathsf d')^\top(\bb - \bb') \ge \mu \|\bSigma^{\frac12}(\bb-\bb')\|^2$.
        \item The penalty $g$ is only assumed convex, $p/n\le\gamma < 1$
            and $\rho$ is $\mu_\rho>0$ strongly convex
            in the sense that $(u-s)(\psi(u)-\psi(s))\ge\mu_\rho(u-s)^2$ for all $u,s\in\R$.
        \item 
            For any constants $\varphi\ge 1, \gamma>1,a_*>0$
            independent of $n,p$, assume
    $\diag(\bSigma)=\bI_p$,
    $p/n\le\gamma\in (0,\infty)$ and
    $\phi_{\max}(\bSigma)/\phi_{\min}(\bSigma)\le\varphi$.
    The penalty is $g(\bb)=\pb{n^{-1/2}}\lambda\|\bb\|_1$
    and either
    \begin{enumerate}
\item The loss is the squared loss $\rho(u)=u^2/2$,
    the noise is normal $\bep\sim N(\mathbf{0},\sigma^2\bI_n)$
    and $\|\bbeta\|_0 \le s_* n$
    where $s_*>0$ is a small enough constant
    depending on $(\varphi,\gamma)$ only,
    and the tuning parameter $\lambda$ satisfies
    $\lambda\ge \sigma \lambda_*$
    for some large enough constant $\lambda_*>0$ depending only
    on $(\varphi,\gamma)$.
\item
    The loss is
    $\rho(u) = \lambda_H^2  \rho_H(\lambda_H^{-1}u)$ for $\rho_H$ the Huber loss $\rho_H(u)=\int_0^{|u|}\min(t,1)dt$
    and some tuning parameter $\lambda_H$.
    Furthermore 
    $s_*>0$ is a small enough constant
    depending on $\{\varphi,\gamma,a_*\}$ only
    such that there exists a set $O\subset [n]$
    with $|O|+\|\bbeta\|_0\le s_*n$ such that the $n-|O|$ noise components
    $(\eps_i)_{i\in[n]\setminus O}$ are iid $N(0,\sigma^2)$.
    The tuning parameters are assumed to satisfy
    $\lambda/\lambda_H=a_*$ and 
    $\lambda\ge \sigma \lambda_*$
    for some large enough constant $\lambda_*>0$ depending only
    on $(\varphi,\gamma,a_*)$.
    \end{enumerate}
    \end{enumerate}
    Here and throughout the paper $\constants\ge0$
    are constants independent of $n,p$.
\end{assumption}
Strong convexity on the penalty (i.e., (i) above)
or strong convexity of the loss (i.e., (ii) above) can be found in numerous
other works on regularized $M$-estimators 
\cite[among others]{donoho2016high,celentano2019fundamental,xu2019consistent}.
In our setting, strong convexity simplifies the analysis
as it grants existence of the derivatives
of $\hbbeta$ with respect to $(\by,\bX)$ ``for free''
\pb{thanks to the Lipschitz conditions obtained in
\Cref{sec:lipschitz-properties}}.
Assumption (iii.a) above relaxes strong convexity on the penalty
and (iii.b) relaxes strong convexity on both the loss and penalty,
by instead assuming a specific choice for $(\rho,g)$.
Assumption~(iii.a) focuses on the Lasso under a sparsity assumption,
while Assumption~(iii.b) focuses
the Huber loss and $\ell_1$ penalty
together with an upper bound on the sparsity
of $\bbeta$ and the number of corrupted components
of $\bep$. In Assumption~(iii.b),
the uncorrupted observations are indexed
in $[n]\setminus O$ and
the corrupted ones are those indexed in $O$.
\Cref{assum:g}(iii.a) and (iii.b) provides non-trivial examples
for which our result holds
without strong convexity on either the loss or the penalty.
Under \Cref{assum:g}(iii.b), the result holds
provided that the corruption is not too strong
and the penalty $g$ (here the $\ell_1$ norm)
is well suited to the structure of $\bbeta$ (here, the sparsity).

The generality in \Cref{assum:g}(i) and (ii) is obtained
by leveraging the strong convexity of either the loss or the penalty.
\Cref{assum:g}(iii.a) and (iii.b) are more specific and
show that without strong convexity,
our results still hold in these specific cases.
The proof under \Cref{assum:g}(iii.a) and (iii.b) leverages
the special form of the loss and penalty and requires a case-by-case
analysis for these choices of $(\rho,g)$.
Although we expect our main results to hold without strong convexity
for other penalty functions than the $\ell_1$ norm (e.g., the group Lasso norm
or indicator functions of convex sets by developing again
case-by-case analysis), a global strategy
to characterize the pairs $(\rho,g)$ for which the results hold
is currently out of reach.

\subsection{Jacobians of \texorpdfstring{$\hbpsi,\hbbeta$}{psi,beta-hat} at the observed data}
Throughout the paper, we view the functions
\begin{equation}
    \label{hbpsi-hbbeta-y-X}
    \begin{aligned}
    \hbbeta:
        \R^{n}\times \R^{n\times p} &\to \R^p,
        \qquad
                                    &(\by,\bX)&\mapsto \hbbeta(\by,\bX) \text{ in \eqref{M-estimator-rho}},
\\
\hbpsi:
        \R^{n}\times \R^{n\times p} &\to \R^n,
        &(\by,\bX)&\mapsto \hbpsi(\by,\bX) = \psi(\by-\bX\hbbeta(\by,\bX))
    \end{aligned}
\end{equation}
as functions of $(\by,\bX)$, though we may drop the dependence in $(\by,\bX)$
and write simply $\hbbeta$ or $\hbpsi$ if the context is clear. 
Here, recall that $\psi$ acts componentwise  on the residuals $\by-\bX\hbbeta$,
so that $\psi(\by-\bX\hbbeta)\in\R^n$ has components
$\psi(y_i - \bx_i^\top\hbbeta)_{i=1,...,n}$.
The hat in the functions $\smash{\hbbeta}$ and $\smash{\hbpsi}$ above emphasize
that they are data-driven quantities, and since they are functions
of $(\by,\bX)$, the directional derivatives of $\hbbeta$ and $\hbpsi$
at the observed data $(\by,\bX)$ are also observable quantities,
for instance
$$\frac{\partial}{\partial y_i} \hbbeta(\by,\bX)
=
\frac{d}{dt} \hbbeta(\by + t\be_i,\bX) \Big|_{t=0}.
$$
Provided that they exist, the derivatives can be computed approximately
by finite-difference or other numerical methods;
a Monte Carlo scheme to compute the required derivatives
is given in \Cref{sec:divergence}.
We thus assume that the Jacobians
\begin{equation}
    \label{jacobians-hbpsi-hbbeta}
    \pb{\bV\defas}
\frac{\partial \hbpsi}{\partial \by}(\by,\bX) = 
\Big(
    \frac{\partial \hpsi_i}{\partial y_l} (\by,\bX)
\Big)_{(i,l)\in[n]\times [n]},
\quad
\frac{\partial \hbbeta}{\partial \by}(\by,\bX) = 
\Big(
    \frac{\partial \hbeta_j}{\partial y_l} (\by,\bX)
\Big)_{(j,l)\in[p]\times [n]}
\end{equation}
are available. Above, $\bV$ is a matrix in $\R^{n\times n}$ with columns
$\frac{\partial \hbpsi}{\partial y_l}(\by,\bX), l=1,...,n$.
\Cref{sec:lipschitz-properties} will make clear
that the existence of such partial derivatives is granted,
under our assumptions, for almost every $(\by,\bX)\in\R^n\times \R^{n\times p}$
by Rademacher's theorem
(cf. Proposition~\ref{prop:constant-bv-e_j} below).
For brevity and if it is clear from context,
we will drop the dependence in $(\by,\bX)$ from the notation,
so that the above Jacobians
$\bV = (\partial/\partial \by) \hbpsi \in\R^{n\times n}$,
$({\partial }/{\partial \by}) \hbbeta \in\R^{p\times n}$
as well as their entries
$(\partial  / \partial y_l ) \hpsi_i$ and
$(\partial / \partial y_l) \hbeta_j$ are implicitly taken
at the currently observed data $(\by,\bX)$.
Next, define
\begin{equation}
    \label{eq:def-df}
    \df = \trace[(\partial/\partial\by) \bX\hbbeta(\by,\bX)],
\end{equation}
let $\hat I = \{i\in[n]: \psi'(y_i-\bx_i^\top\hbbeta)>0\}$
be the set of detected inliers and $\hat O = [n]\setminus \hat I$
the set of detected outliers.
Finally, throughout the paper we denote by
$\pb{\bpsi'}\in\R^n$ the vector with $i$-th component $\psi'(y_i-\bx_i^\top\hbbeta)$,
$\diag(\bpsi')\in\R^{n\times n}$ the diagonal matrix with the entries
of $\bpsi'$ as diagonal entries,
and
$\bh = \hbbeta - \bbeta \in\R^p$
the error vector so that the out-of-sample error
that we wish to estimate is $\|\bSigma^{\frac12}\bh\|^2$.

\subsection{Main result}
Equipped with the above notation \pb{for $\df$} and
the Jacobian
$\bV = (\partial/\partial\by)\hbpsi$
at the observed data $(\by,\bX)$,
we are ready to state the main result of the paper.

\begin{theorem}
    \label{thm:main_robust}
    \pb{
    Let $\hbbeta$ be the $M$-estimator \eqref{M-estimator-rho}.
    Define the estimate $\hat R$ and the remainder $\Rem$ by
    \begin{equation}
\label{eq:hat-R}
\hat R
\defas
\trace[\bV]^{-2}\bigl\{ \|\hbpsi\|^2
    \bigl(2\pb{\df}-p\bigr)
+ \|\bSigma^{-\frac12}\bX^\top\hbpsi\|^2\bigr\},
\qquad
    \Rem
    \defas
\frac{ 
\|\bSigma^{\frac12}\bh\|^2 -\hat R
}{
 \|\hbpsi\|^2/n + \|\bSigma^{\frac12}\bh\|^2
}
.
\end{equation}

\begin{enumerate}
\item
    If Assumptions~\ref{assum:Sigma}, \ref{assum:rho}
    and \ref{assum:g}(i) hold then
$\E|
(\frac 1 n \trace \bV)^{2}
\Rem
|
\le
\C(\mu,\gamma) n^{-1/2}$.

\item
    If Assumptions~\ref{assum:Sigma}, \ref{assum:rho}
    and \ref{assum:g}(ii) hold,
$\E|
\Rem|
\le
\C(\mu_\rho,\gamma) n^{-1/2}$
and
$\frac 1 n
    \trace[\bV]\ge \mu_\rho( 1-\gamma )$
    a.s.
\item
    If Assumptions~\ref{assum:Sigma}, \ref{assum:rho}
    and \ref{assum:g}(iii.a) or (iii.b) hold then
$\E\big|
I \{\Omega\}
\Rem
\big|
\le
\C(\varphi,\gamma,a_*) n^{-1/2}$
where $I \{\Omega\}$ is the indicator of an event $\Omega$ defined in the proof such that $\mathbb P(\Omega)\to 1$  as $n,p\to+\infty$ while $\{\varphi,\gamma,a_*\}$ remain fixed.
Furthermore, $\frac 1 n \trace[\bV]\ge 1-d_* >0$ in $\Omega$ for some constant $d_*$ depending on $\{\varphi,\gamma,a_*\}$ only.
\end{enumerate}
}
%
\end{theorem}
The proof is given in \Cref{sec:proof-main-results}.
Recall that the target of estimation is the 
out-of-sample error $\|\bSigma^{\frac12}\bh\|^2=\|\bSigma^{1/2}(\hbbeta-\bbeta)\|^2$.
In the regime of interest here with $p/n\to\gamma'$,
the risk $\|\bSigma^{\frac12}\bh\|^2$ is typically of the order of a constant, see
\cite{el_karoui2013robust,el_karoui2018impact,donoho2016high,bayati2012lasso,thrampoulidis2018precise,celentano2019fundamental} among others.
When $\|\hbpsi\|^2/n=\|\psi(\by-\bX\hbbeta)\|^2/n$ is also of order of a constant (e.g., with Huber loss $\rho_H$ for which $\sup_{t\in\R} |\psi(t)|=1$), 
\Cref{thm:main_robust} provides
$|\hat R - \|\bSigma^{1/2}\bh\|^2|= O_\P(n^{-\frac12})$
if the multiplicative factor
$(\frac 1 n \trace \bV)^{2}
=(\frac 1 n {\trace[({\partial / }{\partial \by})\hbpsi ]})^{2}$
is bounded away from 0 in the sense
that $n/\trace[\bV] = O_\P(1)$.
In particular, $n/\trace[\bV] = O_\P(1)$ holds by \Cref{thm:main_robust}
under \Cref{assum:g}(ii) and (iii).

The inequality in \Cref{thm:main_robust} is sharp for the 
Ordinary Least-Squares (OLS) with normal noise when $p/n\le \gamma <1$,
i.e.,  for $\rho(u)=u^2/2$ and $g(\bb)=0$.
Assuming $\bep\sim N(\mathbf{0},\sigma^2\bI_n)$,
in this case we have 
\begin{equation}
    \label{eq:ls-discussion}
    \df=p,
    \qquad
    \trace[(\partial/\partial\by)\hbpsi]=n-p,
    \qquad
    \bX^\top\hbpsi = 0,
    \qquad
    \hat R/\sigma^2 = (n-p)^{-2} p \chi^2_{n-p}
\end{equation}
where $\chi^2_{n-p} = \|\hbpsi\|^2/\sigma^2 = \|\by-\bX\smash{\hbbeta}\|^2/\sigma^2$
has chi-square distribution with $n-p$ degrees of freedom.
Furthermore $\by-\bX\smash{\hbbeta}$ is independent of $\smash{\hbbeta}$ and
$
(\hat R - \|\bSigma^{1/2}\bh\|^2)/\sigma^2
$
is the sum of two independent random variables, so that its standard deviation
is at least the standard deviation of $\hat R/\sigma^2$.
Since the standard deviation
of $\hat R/\sigma^2$ equals $(n-p)^{-2}p \sqrt{2 (n-p)}$
and is equivalent to $(1-\gamma')^{-1}\gamma'\sqrt{2(1-\gamma')} n^{-1/2}$
if $p/n\to\gamma'<1$,
this proves that $(\hat R - \|\bSigma^{1/2}\bh\|^2)$
incurs an unavoidable standard deviation of order $\sigma^2 n^{-1/2}$.
This argument valid for the OLS
shows that an error term of order $n^{-1/2}$ in the right hand side
of \Cref{thm:main_robust} is unavoidable at least for this specific 
example.
\Cref{thm:main_robust} is, in this sense, unimprovable.

The OLS is a simple example for which the inequality
$\hat R\approx \|\bSigma^{1/2}\bh\|^2$
also follows, for instance, using
the convergence of the \pb{spectral distribution} of
$\bSigma^{-1/2}\bX^\top\bX\bSigma^{-1/2}/n$
to the Marchenko-Pastur law. Similarly, the approximation
$\hat R\approx \|\bSigma^{1/2}\bh\|^2$ can be obtained
for Ridge regression, that is, $\rho(u)=u^2/2$ and $g(\bb)=\mu\|\bb\|^2$,
using again the limiting spectral distribution of
$\bSigma^{-1/2}\bX^\top\bX\bSigma^{-1/2}/n$.
Outside of these cases, the approximation
$\hat R\approx \|\bSigma^{1/2}\bh\|^2$
\pb{does not directly follow
from the spectral distribution of 
$\bSigma^{-1/2}\bX^\top\bX\bSigma^{-1/2}/n$.}
The present paper develops the theory to explain
the approximation $\hat R \approx \|\bSigma^{1/2}\bh\|^2$
using simple first and second moment identities 
described in \Cref{sec:proof-sketch} which contains
a proof sketch of \Cref{thm:main_robust}.

\begin{table}[]
    \begin{tabular}{@{}|l|l|p{5.7cm}|l|@{}}
\toprule
Loss & Penalty  & $\df$ & $\trace[\frac{\partial\hbpsi}{\partial\by)}]$  \\ \toprule
 $u^2/2$ (Square) & $\lambda\|\bb\|_1$ & $|\hat S|$  & $n-|\hat S|$ \\ \midrule
 $u^2/2$ (Square) & $\lambda\|\bb\|_1 + \mu\|\bb\|_2^2$ &
 $\trace[\bX_{\hat S}(\bX_{\hat S}^\top\bX_{\hat S}+n\mu\bI)^{-1}\bX^\top_{\hat S}]$
                  & $n-\df$  \\ \midrule
 $\rho_H$ (Huber loss) & $\lambda\|\bb\|_1$ & $|\hat S|$ & $|\hat I| - |\hat S|$  \\ \midrule
 $\rho_H$ (Huber loss) & $\lambda\|\bb\|_1 + \mu\|\bb\|_2^2$ &
 $\trace[\bX_{\hat S}(\bX_{\hat S}^\top\bD\bX_{\hat S}+n\mu\bI)^{-1}\bX^\top_{\hat S}\bD]$ \newline where $\bD=\diag(\bpsi')$
 & $|\hat I| - \df$  \\ \midrule
 Any & $\lambda\|\bb\|_1 + \mu\|\bb\|_2^2$ &
 $\trace[\eqref{eq:df-Elastic-Net}]$
     &
 $\trace[\eqref{eq:trace-V-Elastic-Net}]$
 \\
 \bottomrule
\end{tabular}
\caption[Explicit formulae for $\df$ and $\trace((\partial/\partial\mathbf{y})\mathbf{\hat\psi})$]{Explicit formulae for the factors $\df$ and
    $\trace[(\partial/\partial\by)\hbpsi]$ used
    in the out-of-sample estimate $\hat R$
    for commonly used penalty functions.
    Here $\hat I =\{i\in[n]: \psi'(y_i-\bx_i^\top\hbbeta)\}$ is the
    set of inliers, $\hat S=\{j\in [p]:\hat\beta_j \ne 0\}$ is the set
    of active variables, $\bX_{\hat S}$ the submatrix of $\bX$
    made of columns indexed in $\hat S$.
See Propositions~\ref{prop:specific-ENet} and~\ref{prop:Huber-Lasso} for more details.
}
\label{tab:my-table}
\end{table}

\subsection{On the range of the multiplicative factors in $\hat R$}
\Cref{thm:main_robust} involves the multiplicative
factors $\trace[(\partial/\partial\by)\hbpsi]$
and \pb{$\df$}.
The following result provides
the possible range for these quantities.

\begin{restatable}{proposition}{PropPsiPsd}
    \label{prop:jacobian-psd-htheta}
    Assume that $\rho$ is convex differentiable and that $\psi=\rho'$ is 1-Lipschitz.
    For every fixed $\bX\in\R^{n\times p}$ the following holds.
    \begin{itemize}
    \item 
    For almost every $\by$,
     the map $\by\mapsto \hbpsi = \psi(\by-\bX\hbbeta)$ is Frechet differentiable at $\by$, 
     and the Jacobian $\bV = (\partial/\partial\by)\hbpsi\in\R^{n\times n}$ is symmetric
    psd with operator norm at most one so that
    $\trace[\bV]=\trace[({\partial}/{\partial\by})\hbpsi] ~\in~ [0,n]$.
    \item
    If additionally \Cref{assum:g}(i) or (iii.b) holds
    then almost surely
    $
    \pb{\df}
    \le |\hat I|
    $
    where $\hat I = \{i\in[n]: \psi'(y_i -\bx_i^\top\hbbeta)>0\}$
    is the set of inliers.
    \end{itemize}
\end{restatable}

The proof of Proposition~\ref{prop:jacobian-psd-htheta}
is given in \Cref{sec:proof-jacobian-psd}.
For the square loss, $\df$ is no more than the sample
size $n$ since for any penalty, $\df=\trace[(\partial/\partial\by)\bX\hbbeta]$
is the divergence of a 1-Lipschitz function \cite[e.g.]{bellec2016bounds}.
The second point above states that this inequality is
replaced by $\df\le |\hat I|$ for general loss functions,
i.e., $n$ is replaced by the number of inliers.

\subsection{$\hat R$ for certain examples of loss functions}

\subsubsection{Square loss} 
As a first illustration of the above result,
consider the square loss $\rho(u)=u^2/2$. 
As we will detail in 
\Cref{sec:square_loss} devoted to the square loss,
$\hbpsi = \by-\bX\hbbeta$ is the residuals and
$\bV = (\partial/\partial\by)\hbpsi = \bI_n - \bX(\partial/\partial\by)\hbbeta$
by the chain rule so that $\hat R$ reduces to
\begin{equation*}
\hat R =
(n-\df)^{-2} \bigl\{ \|\hbpsi\|^2
    \bigl(2\df -p\bigr)
+ \|\bSigma^{-\frac12}\bX^\top\hbpsi\|^2\bigr\}.
\end{equation*}
Above, $\df=\trace[(\partial/\partial\by)\bX\hbbeta]$
is the usual effective number of parameters or effective
degrees-of-freedom of $\hbbeta$ that
dates back to \citet{stein1981estimation}.
This estimator of the out-of-sample error
for the square loss was known only for two specific penalty functions $g$.
The first is $g=0$ \cite{leeb2008evaluation} in which case $\hbbeta$ is
the Ordinary Least-Squares 
and $\df=p$.
The second is $g(\bb)=\lambda\|\bb\|_1$
\cite{bayati2013estimating,miolane2018distribution},
in which case $\hbbeta$ is the Lasso and $\df=|\{j\in[p]:\hbeta_j\ne 0\}|$.
For $g$ not proportional to the $\ell_1$-norm,
the above result is to our knowledge
novel, even restricted to the square loss.
As we detail in \Cref{sec:square_loss},
the algebraic nature of the square loss leads to additional results
for noise level estimation and adaptive estimation of the generalization error.
Here, adaptive means without knowledge of $\bSigma$.
To our knowledge, the estimate $\hat R$ for general loss functions
($\rho$ different than the square loss) is new.

\subsubsection{Huber loss}
As a second illustration, consider the Huber loss
$\rho_H(u)=\int_0^{|u|}\min(t,1)dt$, i.e.,
\begin{equation}
    \label{eq:Huber-loss}
\rho_H(u) = u^2/2 \text{ for } |u|\le 1
\quad\text{ and }\quad \rho_H(u) = (|u| - 1/2) \text{ for }|u|>1.
\end{equation}
By the chain rule in \eqref{eq:chain-rule-hard} below,
using \eqref{eq:derivatives-y} and noting that $\diag(\bpsi')=\diag(\bpsi')^2$
for the Huber loss,
$\trace[\bV]=\trace[(\partial/\partial \by)\hbpsi]
= \trace[\diag(\bpsi')(\bI_n - (\partial/\partial\by)\bX\hbbeta)]
=|\hat I| - \df$
where $\hat I = \{i\in[n]: \psi'(y_i-\bx_i^\top\hbbeta)>0\}$.
The out-of-sample estimate $\hat R$ becomes
$$\hat R
= (|\hat I| - \df)^{-2}
\bigl\{
    \|\hbpsi\|^2
    (2\df -p)
    + \|\bSigma^{-\frac12}\bX^\top\hbpsi\|^2
\bigr\}.
$$
This conveniently mimics the estimate available for the square loss, with 
the sample size $n$ replaced by the number of inliers $|\hat I|$.
If a scaled version of the Huber loss is used, i.e.,
with loss $\rho(y)=\Lambda_*^2\rho_H\big(\Lambda_*^{-1}u\big)$
for some scaling parameter $\Lambda_*>0$,
then the previous display still holds.

\subsection{When is \texorpdfstring{$\trace[\bV]= \trace[(\partial/\partial\by)\hbpsi]/n$}{trace dpsi/dy} too small or 0? }

\label{subsec:too_small}

We emphasize that the above result does not provide guarantees
against all forms of corruption in the data, and $\hat R$
may produce incorrect inferences (or be undefined due to division by 0)
if the multiplicative factor 
$(\frac 1 n \trace \bV)^2 
= 
(\frac 1 n \trace[(\partial/\partial\by)\hbpsi])^2$ is too small
or equal to 0.
This issue does not arise under \Cref{assum:g}(ii) or (iii), as in this case
\Cref{thm:main_robust} grants $\frac 1 n \trace[\bV]$ larger than
some positive constant with high probability.

Recall that $\frac 1 n \trace\bV = \frac 1 n \trace[(\partial/\partial\by)\hbpsi] \in [0,1]$
by Proposition~\ref{prop:jacobian-psd-htheta}.
To exhibit situations for which 
$\frac 1 n \trace[(\partial/\partial\by)\hbpsi]$ is close to 0
under \Cref{assum:g}(i),
by the chain rule \eqref{eq:chain-rule-hard}
below we have
\begin{equation*}
    \tfrac 1 n \trace[\bV] = 
\tfrac 1 n \trace[(\partial/\partial\by)\hbpsi]
=
\tfrac 1 n \trace[\diag(\bpsi')(\bI_n - \bX (\partial/\partial\by)\hbbeta)].
\end{equation*}
Hence the above multiplicative factor is equal to 0
when $\diag(\bpsi') = \mathbf{0}$, i.e., $\psi'(y_i -\bx_i^\top\hbbeta) = 0$
for all observations $i=1,...,n$: If all observations are classified as outliers
by the minimization problem \eqref{M-estimator-rho} then
$\hat R$ is undefined and cannot be used. On the other hand, by \Cref{thm:main_robust} under \Cref{assum:g}(i)
the relationship
\begin{equation}
    \label{eq:previous-display-upper-bound-hat-R}
\bigl(\tfrac 1 n \trace[(\partial/\partial\by)\hbpsi]\bigr)^2
\big|\|\bSigma^{\frac12}\bh\|^2  - \hat R\big|
\le\Rem (\|\bSigma^{\frac12}\bh\|^2 + \|\hbpsi\|^2/n)
\end{equation}
always holds with $\Rem=O_\P(n^{-1/2})$, which suggests that
$(\frac 1 n \trace[\bV])^2 = (\frac 1 n \trace[(\partial/\partial\by)\hbpsi])^2$
must be bounded away from 0 in order to obtain meaningful upper bounds
on $\big|\|\bSigma^{\frac12}\bh\|^2  - \hat R\big|$.
If the loss is strongly convex and $\gamma<1$
as in \Cref{assum:g}(ii),
or under \Cref{assum:g}(iii.a) or (iii.b)
for $\ell_1$ penalty with square or Huber loss, the factor
$\big(\frac 1 n \trace[(\partial/\partial\by)\hbpsi]\big)^2$
is bounded away from 0 as noted in the second claim
of \Cref{thm:main_robust}.
However, $\big(\frac 1 n \trace[(\partial/\partial\by)\hbpsi]\big)^2$
is not necessarily bounded away from 0 under
\Cref{assum:g}(i): Indeed it is easy to construct an example
where $\psi'(y_i - \bx_i^\top\hbbeta)=0$ for all $i\in[n]$
with high probability, for instance for the Huber loss
$\rho=\rho_H$ defined in \eqref{eq:Huber-loss}
with penalty
$g(\bb)=K\|\bb-\ba\|^2$ for some large $K$ and some vector $\ba\in\R^p$
with large distance $\|\ba-\bbeta\|$
(this is a purposely poor choice of penalty function that will induce
a large error $\|\bSigma^{\frac12}\bh\|^2$).
This example highlights that the above result does not provide
estimation guarantees against all forms of corruption
under \Cref{assum:g}(i) without further assumption:
If the corruption is so strong that all observations are outliers and
$\trace[(\partial/\partial\by) \hbpsi] = 0$ then
$\hat R$ is undefined and
the inequality of \Cref{thm:main_robust} is unusable to estimate
or bound from above the out-of-sample error.

\subsection{Closed form expression for specific choices of $(\rho,g)$}
The multiplicative factors
$\trace[ (\partial/\partial \by)\hbpsi]$
and
$\df = \trace[(\partial/\partial \by)\bX\hbbeta]$
have explicit closed form expressions for particular 
choices of $(\rho,g)$. We now provide such examples;
a summary is provided in \Cref{tab:my-table}.
The next section provides a general method to approximate
$\trace[ (\partial/\partial \by)\hbpsi]$ and $\df$
for arbitrary $(\rho,g)$ when no closed form
expressions are available.

\begin{proposition}
    \label{prop:specific-ENet}
Assume that $\psi$ is 1-Lipschitz and 
consider an Elastic-Net penalty of the form $g(\bb) = \mu \|\bb\|^2/2 + \lambda \|\bb\|_1$
for $\mu>0$, $\lambda\ge 0$.
For almost every $(\by,\bX)$,
the map $\by\mapsto\bX\hbbeta$ is differentiable at $\by$ and
\begin{equation}
(\partial/{\partial \by}) \bX\hbbeta
= \bX_{\hat S} (\bX_{\hat S}^\top\diag(\bpsi')\bX_{\hat S} + \pb{n}\mu \bI_{|\hat S|} )^{-1}
\bX_{\hat S}^\top\diag(\bpsi') 
\label{eq:df-Elastic-Net}
\end{equation}
where $\hat S = \{j\in[p]: \hbeta_j\ne 0\}$
and $\bX_{\hat S}$ is the submatrix of $\bX$ obtained made of columns indexed
in $\hat S$, and
\begin{equation}
\label{eq:trace-V-Elastic-Net}
\tfrac{\partial\hbpsi}{\partial\by}
= \diag(\bpsi')^{\frac12}\bigl[\bI_n - \diag(\bpsi')^{\frac12}\bX_{\hat S} (\bX_{\hat S}^\top\diag(\bpsi')\bX_{\hat S} + \pb{n}\mu \bI_{|\hat S|} )^{-1}
\bX_{\hat S}^\top\diag(\bpsi')^{\frac12}\bigr]\diag(\bpsi')^{\frac12}.
\end{equation}
\end{proposition}
The proof is given in \Cref{proof:ENet-Huber-Lasso}.
For the Elastic-Net penalty,
the factors $
\df = \trace [(\partial/\partial \by) \bX\hbbeta]$ and $
\trace [ (\partial / \partial \by )\hbpsi ]$
appearing in the out-of-sample estimate
$\hat R$ have thus reasonably tractable forms and can be computed efficiently
by inverting a matrix of size $\hat S$ once the elastic-net
estimate $\hbbeta$ has been computed.
The above estimates for general loss functions
are closely related to the formula
for
$(\partial/\partial \by)\bX\hbbeta$ known for the Elastic-Net
with square loss
\cite[Equation (28)]{tibshirani2012}
\cite[Section 3.5.3]{bellec_zhang2018second_stein},
the only difference being several multiplications by the diagonal matrix $\diag(\bpsi')$.
Closed form expressions can also be obtained for different penalty
functions, such as the Group-Lasso penalty, by differentiating the KKT
conditions as explained in
\cite{bellec_zhang2019second_poincare} for the square loss.

For the Huber loss with $\ell_1$-penalty,
these multiplicative factors are even simpler, as shown
in the following proposition. We keep using the notation
$\hat I=\{i\in[n]: \psi'(y_i-\bx_i^\top\hbbeta)>0\}$
for the set of inliers (the set of outliers being $[n]\setminus \hat I$),
and $\hat S=\{j\in[p]:\hbeta_j\ne 0\}$ for the set of active covariates.

\begin{proposition}
    \label{prop:Huber-Lasso}
    Let $\rho(u)=n\lambda_*^2\rho_H((\sqrt n\lambda_*)^{-1} u)$ where $\rho_H$ is the Huber loss and let $g(\bb)=\lambda\|\bb\|_1$ be the penalty for $\lambda_*,\lambda>0$.
    For almost every $(\by,\bX)$,
    the functions $\by\mapsto \hat I, \by\mapsto \hat S$ and $\by\mapsto\diag(\bpsi')$ are constant in a neighborhood of $\by$
    and 
    $\hbQ \defas \pb{\diag(\bpsi')} (\partial  /\partial \by) \bX\hbbeta$
is the orthogonal projection onto the column span of $\diag(\bpsi')\bX_{\hat S}$.
Furthermore
$
(\partial / \partial \by )\hbpsi
= \diag(\bpsi') - \hbQ
$ and the multiplicative factors appearing in $\hat R$ satisfy for almost
every $(\by,\bX)$
$$
\df = 
\trace [\diag(\bpsi')(\partial/\partial \by) \bX\hbbeta] = |\hat S|,
\qquad
\trace [ (\partial / \partial \by )\hbpsi ] = |\hat I| - |\hat S|\ge0.
$$
\end{proposition}
The proof is given in \Cref{proof:ENet-Huber-Lasso}.
Proposition~\ref{prop:Huber-Lasso} implies that for the Huber loss with $\ell_1$-penalty,
the out-of-sample error estimate $\hat R$ becomes simply
\begin{equation}
\hat R =(|\hat I| - |\hat S|)^{-2}
\big\{
    \|\hbpsi\|^2(2|\hat S| -p)
    +
    \|\bSigma^{-\frac12}\bX^\top\hbpsi\|^2
\big\}.
\label{eq:hat-R-Huber-Lasso}
\end{equation}
For the square-loss and identity covariance, the above estimate was known
\cite{bayati2013estimating,miolane2018distribution}
with $|\hat I|$ replaced by $n$.
In hindsight the extension of this estimate to the Huber loss
is natural: the sample size should be replaced by
the number of observed inliers $|\hat I|$.

\subsection{Proof ingredients and a new probabilistic inequality}
\label{sec:proof-sketch}
Preliminaries for the proofs are twofold.
First several Lipschitz
properties are derived, to make sure that the derivatives used in the proofs
exist almost surely.
This is done in \Cref{sec:lipschitz-properties}.
Second, without loss of generality we may assume that $\bSigma=\bI_p$, replacing
if necessary $(\bX,\bbeta,\hbbeta,g)$ by $(\bX^*,\bbeta^*,\hbbeta{}^*,g^*)$ as follows,
\begin{equation}
    \label{eq:variable-change}
    \bX\rightsquigarrow \bX^* =  \bX\bSigma^{-\frac12},
    \quad
    g(\cdot) \rightsquigarrow g^*(\cdot) = g(\bSigma^{-\frac12}(\cdot)),
    \quad
    \hbbeta \rightsquigarrow \hbbeta{}^* = 
    \bSigma^{\frac12}\hbbeta,
    \quad
    \bbeta \rightsquigarrow \bbeta^* = 
    \bSigma^{\frac12}\bbeta
    .
\end{equation}
This change of variable leaves the quantities $\{\by, \df, \hbpsi, \bX\hbbeta,
\|\bSigma^{-\frac12}\bX^\top\hbpsi\|^2,\|\bSigma^{\frac12}\bh\|^2,\trace[\bV]\}$ unchanged,
so that \Cref{thm:main_robust} holds for general $\bSigma$
if it holds for $\bSigma=\bI_p$ after the change of variable in
\eqref{eq:variable-change}.
Next, throughout the proof we 
consider the scaled version of $\hbpsi$ and the error vector $\bh$ given by
\begin{equation}
    \label{eq:br-bh-proof-overview}
\br = n^{-\frac12} \hbpsi = n^{-\frac12} \psi(\by-\bX\hbbeta),
\qquad
\bh = \hbbeta-\bbeta 
\end{equation}
so that $\|\br\|^2$ and $\|\bh\|^2$ are of the same order.

At this point the main ingredients of the proof are threefold.
The first ingredient is the following.
\begin{restatable*}{proposition}{propSOSmatrix}
    \label{prop:SOS-X-D-bh-br-main-text}
    Let $\bX=(x_{ij})\in\R^{n\times p}$ with  iid $N(0,1)$ entries and $\bfeta:\R^{n\times p}\to \R^p$,
    $\brho:\R^{n\times p}\to \R^n$ two vector valued functions,
    with weakly differentiable components
    $\eta_1,...,\eta_p$ and $\rho_1,...,\rho_n$.
    Then
\begin{align}
    \E\Big[\Big(
            {\brho^\top \bX \bfeta}
    -
    \sum_{i=1}^n
    \sum_{j=1}^p
        \frac{\partial(\rho_i\eta_j)}{\partial x_{ij}}
    \Big)^2
\Big]
    &=
    \E\Big[{\| \brho\|^2 \|\bfeta\|^2}
    +
    \sum_{j=1}^p\sum_{k=1}^p
    \sum_{i=1}^n\sum_{l=1}^n
        \frac{\partial(\rho_i \eta_j)}{\partial x_{lk}}
        \frac{\partial(\rho_l \eta_k)}{\partial x_{ij}}
    \Big]
    \nonumber
    \\ &\le
    \E\Big[{\| \brho\|^2 \|\bfeta\|^2}
    +
    \sum_{j=1}^p
    \sum_{i=1}^n
    \Big\|
        \frac{\partial(\brho \bfeta^\top)}{\partial x_{ij}}
    \Big\|_F^2
    \Big]
\label{eq:to-evaluate-SOS}
\end{align}
    provided that the second line is finite,
    where for brevity we write $\brho=\brho(\bX)$, $\bfeta=\bfeta(\bX)$,
    and similarly for the partial derivatives (i.e.,
    omitting the dependence in $\bX$).
\end{restatable*}
The proof is given in \Cref{sec:6-identitities-gaussian}.
In practice for the proofs of the main theorems, we take
\begin{equation}
\bfeta(\bX) =(\|\br\|^2 + \|\bh\|^2)^{-1/2} \bX^\top\br
\quad\text{ and }\quad
\brho(\bX) = (\|\br\|^2 + \|\bh\|^2)^{-1/2} \br 
\label{eq:bfeta-brho}
\end{equation}
with $\br,\bh$ defined in \eqref{eq:br-bh-proof-overview}.
The equality in \eqref{eq:to-evaluate-SOS} is a matrix generalization
of
\cite[Eq. (8.6)]{stein1981estimation}, \cite{bellec_zhang2018second_stein}.
Its proof
relies on Gaussian integrations by parts and presents no difficulty,
although it requires 
some bookkeeping for the different summation signs and indices.
The result of \cite{bellec_zhang2018second_stein}, that covers 
the case $p=1,\bfeta=1$, is recalled in Proposition~\ref{prop:BZ18}.
Although the above matrix formulation is new and particularly
useful for our purpose, it is essentially
equivalent to the $p=1,\bfeta=1$ case after vectorization
as explained in \Cref{sec:6-identitities-gaussian}.

The second ingredient is the following novel probabilistic inequality,
which is the main probabilistic contribution of the present paper.
\begin{restatable*}{theorem}{thmChiSquare}
    \label{thm:chi-square-type-main-text-rho}
    Assume that $\bX$ has iid $N(0,1)$ entries,
    that $\brho:\R^{n\times p}\to\R^n$ is weakly differentiable
    and that $\|\brho\|\le 1$ almost everywhere.
    Then
    \begin{equation}
\label{novel-inequality-chi2}
\E\Big|
p \|\brho\|^2
-
\sum_{j=1}^p\Big(
    \brho^\top\bX\be_j -\sum_{i=1}^n \frac{\partial \rho_i}{\partial x_{ij}} \Big)^2
\Big|
\le \Cl{novel}
\E
\Bigl[
1+
\sum_{i=1}^n\sum_{j=1}^p \|
    \frac{\partial \brho}{\partial x_{ij}} 
    \|^2
\Bigr]^{1/2}
\sqrt p
+
\Cr{novel} \E\sum_{i=1}^n\sum_{j=1}^p \|\frac{\partial \brho }{\partial x_{ij}} \|^2
    \end{equation}
where $\Cr{novel}>0$ is an absolute constant.
\end{restatable*}
The proof of \Cref{thm:chi-square-type-main-text-rho} is given in
\Cref{sec:proof_chi2}.
To our knowledge, inequality \eqref{novel-inequality-chi2} is novel.
In the simplest case, if $\brho$ is constant with $\|\brho\|=1$ then
\eqref{novel-inequality-chi2} reduces to
$\E|\chi^2_p - p| \le \C \sqrt p$ and the dependence in $\sqrt p$ is optimal,
so that \eqref{novel-inequality-chi2} is in a sense unimprovable.
The flexibility of inequality \eqref{novel-inequality-chi2} is that
the left hand side of \eqref{novel-inequality-chi2} is provably of order
$\sqrt p$, as in the case of $\E|\chi^2_p-p|$, as long as
the derivatives of $\brho$ do not vary too much in the sense
that $\E\sum_{i=1}^n\sum_{j=1}^p \|(\partial/\partial x_{ij})\brho\|^2 \le \Cl{lip}$ for some constant independent of $n,p$.
This inequality holds for instance for all $(\Cr{lip}/n)^{1/2}$-Lipschitz
functions $\brho:\R^{n\times p}\to \R^n$ since the squared Frobenius
norm of the Jacobian of $\brho$ is bounded by $n$ times the square
of the Lipschitz constant. 
A right-hand side of order $\sqrt p$ in \eqref{novel-inequality-chi2}
would be expected if the $p$ terms
$$A_j = \|\brho\|^2 - \bigl(
    \brho^\top\bX\be_j -\sum_{i=1}^n \frac{\partial \rho_i}{\partial x_{ij}} \bigr)^2
$$
were mean-zero and independent, thanks to
$\E[(\sum_{j=1}^p A_j)^2]=\sum_{j=1}^p \E[A_j^2]$
by independence.
The surprising feature of \eqref{novel-inequality-chi2}
is that such bound of order $\sqrt p$ holds despite
the intricate, nonlinear dependence between $A_j$
and the $p-1$ other terms $(A_k)_{k\ne j}$
through the $(\Cr{lip}/n)^{1/2}$-Lipschitz
function $\brho$ and its partial derivatives.
\pb{We now state two useful inequalities that follow directly by combining
    \Cref{thm:chi-square-type-main-text-rho} and \Cref{prop:SOS-X-D-bh-br-main-text}
for $\bfeta=\bX^\top\brho$.
\begin{corollary}
    \label{big_cor}
    Assume that $\bX$ has iid $N(0,1)$ entries,
    that $\brho:\R^{n\times p}\to\R^n$ is weakly differentiable
    and that $\|\brho\|\le 1$ almost everywhere.
    Then
    \begin{align}
\E\Big|
\|\bX^\top\brho\|^2
- p \|\brho\|^2
-
\sum_{j=1}^p\Big(
\sum_{i=1}^n \frac{\partial \rho_i}{\partial x_{ij}} \Big)^2
- 2\sum_{i=1}^n\sum_{j=1}^p
\rho_i\be_j^\top\bX^\top\frac{\partial \brho}{\partial x_{ij}}
\Big|
&\le 
\C RHS,
\label{eq:new_cor_first}
\\
\E\Big|
\sum_{j=1}^p
\brho^\top\bX\be_j\sum_{i=1}^n
\frac{\partial \rho_i}{\partial x_{ij}}
-
\sum_{j=1}^p\Big(
\sum_{i=1}^n \frac{\partial \rho_i}{\partial x_{ij}} \Big)^2
- \sum_{i=1}^n\sum_{j=1}^p
\rho_i\be_j^\top\bX^\top\frac{\partial \brho}{\partial x_{ij}}
\Big|
&\le \C RHS
\label{eq:new_cor_second}
    \end{align}
    where
    $RHS=
\E
[
    \sum_{i=1}^n\sum_{j=1}^p \|\frac{\partial \brho }{\partial x_{ij}} \|^2
]
+
\sqrt{p+n} 
+
    \E[(p +\|\bX\|_{op}^2)\sum_{i=1}^n\sum_{j=1}^p \|\frac{\partial \brho}{\partial x_{ij}} \|^2]^{\frac12}
$.
\end{corollary}
The proof is given in \Cref{sec:proof_chi2}.
Inequalities \eqref{eq:to-evaluate-SOS},
\eqref{novel-inequality-chi2},
\eqref{eq:new_cor_first} and \eqref{eq:new_cor_second}
involve derivatives with respect to the
entries of $\bX$.
It might thus be surprising at this point
that \Cref{thm:main_robust}
and the estimate $\hat R$ in \eqref{eq:hat-R}
involve the derivatives of $(\hbpsi,\hbbeta)$
with respect to $\by$ only,
and no derivatives
with respect to the entries of $\bX$.
The third major ingredient
of the proof is to provide gradient identities
between the derivatives of $\hbbeta,\hbpsi$ with
respect to $\by$ and those with respect 
to $\bX$, by identifying certain perturbations of the data $(\by,\bX)$ that leave $\hbbeta$ or $\hbpsi$ unchanged.
For instance,
\Cref{corollary:constant-bv-e_j}
shows that $\hbpsi$ stays the same and $\hbbeta$ is still solution
of the optimization problem \eqref{M-estimator-rho}
if the observed data $(\by,\bX)$ is replaced by
$(\by+\smash{\hbeta_j}\bv, \bX + \bv\be_j^\top)$
for any canonical basis vector $\be_j\in\R^p$
and any direction $\bv\in\R^n$ with $\bv^\top\smash{\hbpsi} = 0$.
If $\smash{\hbpsi}(\by,\bX)$ and $\smash{\hbbeta}(\by,\bX)$ are Frechet differentiable
with respect to $(\by,\bX)$, such perturbations that leaves
$\hbbeta,\hbpsi$ unchanged provide relationship
between the partial derivatives with respect to $\by$
and the partial derivatives with respect to to entries of $\bX$.
If $\hbbeta{}^0=\hbbeta(\by^0,\bX^0)$,
$\hbpsi{}^0=\hbpsi(\by^0,\bX^0)$ and similary for $\bpsi^0{}'$,
another perturbation that leaves
$\hbbeta$ unchanged at $(\by^0,\bX^0)$ is
\begin{equation}
\by(t)=
    \by + t(\bU\diag(\bpsi^0{}')\bX\hbbeta-\bU^\top\bpsi^0),
    \qquad
    \qquad
\bX(t))= \bX + t\bU \diag(\bpsi^0{}')
\label{perturationU}
\end{equation}
as $t\to 0$ for any fixed $\bU\in\R^{n\times n}$, in the sense that
$\frac{d}{dt}\hbbeta(\by(t),\bX(t))|_{t=0} = \mathbf 0$ when $\mu>0$
(i.e, the penalty is strongly convex). A more convenient form
of such results was developed in \cite{bellec2021derivatives} after the first version of the present manuscript appeared; we 
include it in the next lemma for convenience as it makes the proofs easier to
read.

\begin{restatable}[Variant of Theorem 1 in \cite{bellec2021derivatives}]{lemma}{lemmaNinja}
    \label{lemma:ninja-variant}
    Let $\mu\ge 0$ (allowing $\mu=0$).
    Let $\hbbeta$ 
    be the $M$-estimator \eqref{M-estimator-rho}
    with convex loss $\rho$
    and $\mu$-strongly convex penalty $g$
    with respect to a positive definite $\bSigma$
    in the sense of \Cref{assum:g}(i).
    Assume that $\hbbeta(\by,\bX)$
    and $\hbpsi(\by,\bX)$
    are Frechet
    differentiable at $(\by^0,\bX^0)$.
    Let $\hbbeta{}^0 = \hbbeta(\by^0,\bX^0)$
    and $\bpsi^0 = \psi(\by^0-\bX^0\hbbeta{}^0)$,
    as well as
    the $n\times n$ diagonal matrix
    $\bD^0=
    \diag(\bpsi{}^0{}') = \diag(\psi_1^0{}', ..., \psi_n^0{}')$
    where $\psi_i^0{}' = \psi'(y_i^0-\bx_i^0{}^\top\hbbeta{}^0)$ for each $i=1,...,n$.
    If $\psi$ is continuously differentiable
    and $\mu n \bSigma + \bX^0{}^\top\diag(\bpsi^0{}')\bX^0$ is positive
    definite then there exists a $p\times p$ matrix
    $\hbA{}(\by^0,\bX^0)$
    depending on $(\by^0,\bX^0)$ such that
    \begin{align}
    \tfrac{\partial \hbbeta}{\partial x_{ij}}
    (\by^0,\bX^0)
    &= \hbA(\by^0,\bX^0)\bigl[\be_j\psi_i^0 - \bX^0{}^\top\bD^0\be_i \hbeta_j^0\bigr]
    \qquad \text{ for all } i\in[n],j\in[p],
    \label{eq:derivatives-X}
    \\
    \tfrac{\partial \hbbeta}{\partial y_{l}}
    (\by^0,\bX^0)
    &= \hbA(\by^0,\bX^0)\bigl[\bX^0{}^\top\bD^0\be_l\bigr]
    \qquad\qquad\qquad\qquad \text{ for all } l\in[n],
    \label{eq:derivatives-y}
    \\
        \|\bSigma^{1/2}\hbA{}(\by^0,\bX^0)\bSigma^{1/2}\|_{op}
    &\le \phi_{\min}\bigl(\mu n \bI_p + \bSigma^{-1/2}\bX^\top \bD^0\bX\bSigma^{-1/2} \bigr)^{-1}.
    \label{eq:hbA-op}
    \end{align}
    If $\psi$ is only 1-Lipschitz (but not necessarily continuously
    or everywhere differentiable)
    and the function
    $(\by,\bX)\mapsto \hbbeta(\by,\bX)$ is Lipschitz
    in some open set $U$, then
    for almost every $(\by^0,\bX^0)$ in $U$,
    the chain rule
    \begin{equation}
        \label{eq:chain-rule-hard}
    \tfrac{\partial\hbpsi}{\partial x_{ij}}(\by^0,\bX^0)
    =\bD^0[-\bX^0\tfrac{\partial\hbbeta}{\partial x_{ij}}(\by^0,\bX^0) - \be_i \hbeta_j^0],
    \qquad
    \tfrac{\partial\hbpsi}{\partial y_l}(\by^0,\bX^0)
    =\bD^0[\be_l-\bX^0\tfrac{\partial\hbbeta}{\partial y_l}(\by^0,\bX^0)]
    \end{equation}
    and \eqref{eq:derivatives-X}-\eqref{eq:derivatives-y}-\eqref{eq:hbA-op} still hold
    for some $\hbA(\by^0,\bX^0)\in\R^{p\times p}$
    when the right-hand side of \eqref{eq:hbA-op} is finite.
\end{restatable}
We provide a short proof in \Cref{sec:gradient-identities}.
We will omit the $^0$ superscript and the explicit dependence on $(\by,\bX)$
for brevity, and write simply using the chain rule
\begin{align}
    \tfrac{\partial\hbbeta }{\partial x_{ij}}
    = \hbA[\be_j\hpsi_i - \bX^\top\diag(\bpsi')\be_i \hbeta_j],
    \qquad
    &
    \tfrac{\partial\hbbeta }{\partial y_l}
    = \hbA\bX^\top\diag(\bpsi')\be_l,
    \label{def-V}
    \\
    \tfrac{\partial\hbpsi}{\partial x_{ij}}
    = -\diag(\bpsi')\bX\hbA\be_j\hpsi_i - \bV\be_i \hbeta_j,
    \qquad
    &
    \tfrac{\partial\hbpsi}{\partial y_l}
    = \bV\be_l
    \quad\text{ where }\bV = \diag(\bpsi')-\diag(\bpsi')\bX\hbA\bX^\top\diag(\bpsi').
    \nonumber
\end{align}
Since the matrix $\hbA$ is the same in the derivatives with respect to $\bX$
and with respect to $\by$, \eqref{def-V} provides 
relationships between
the partial derivatives with respect to $\bX$ and to $\by$.
As we see in the next section, this
lets us evaluate the left-hand side of \eqref{eq:new_cor_first}
to obtain \Cref{thm:main_robust}.

\subsection{Proof of the main result}
\label{sec:proof-main-results}

As defined in \eqref{hbpsi-hbbeta-y-X}, we consider the functions
$\hbpsi=\hbpsi(\by,\bX)=\psi(\by-\bX\hbbeta)$ and $\hbbeta=\hbbeta(\by,\bX)$ as functions
of $(\by,\bX)\in\R^{n}\times \R^{n\times p}$.
At a point $(\by,\bX)$ where these functions are Frechet differentiable,
the statistician has access to the Jacobians and partial derivatives in \eqref{jacobians-hbpsi-hbbeta}.
These two functions, $\hbpsi$ and $\hbbeta$ are the only functions
of $(\by,\bX)$ that we will consider; the hat in $\hbpsi,\hbbeta$
emphasize that these are functions of $(\by,\bX)$.

In the proof, we argue conditionally on $\bep$ and consider functions of $\bX$ only, such as
\begin{equation}
    \label{bpsi-br-bh}
    \bpsi = \psi(\by-\bX\hbbeta), \qquad \br = n^{-\frac12}\bpsi = n^{-\frac12}\psi(\by-\bX\hbbeta),
\qquad
\bh = \hbbeta-\bbeta
\end{equation}
valued in $\R^n$, $\R^n$ and $\R^p$ respectively.
Formally, 
$\bpsi:\R^{n\times p}\to\R^n$,
$\br:\R^{n\times p}\to\R^n$,
$\bh:\R^{n\times p}\to \R^p$
and we view the functions in \eqref{bpsi-br-bh}
as functions of $\bX$ only while the noise $\bep$ is fixed.
We may write $\br=\br(\bX)$ to recall that convention.
We will denote their partial derivatives by $(\partial/\partial x_{ij})$.
With the above definitions,
the function $\bpsi=\bpsi(\bX)$ is related to $\hbpsi=\hbpsi(\by,\bX)$ by
$\bpsi = \hbpsi(\bX\bbeta + \bep,\bX)$ so that 
if $\hbpsi$ is Frechet differentiable at $(\by,\bX)$ we have
\bel{eq:chain-rule-r}
  \sqrt n  
  (\partial/\partial x_{ij})\br (\bX)
  &=&
  (\partial/\partial x_{ij})\bpsi (\bX)
\cr&=&
[(\partial/\partial x_{ij})\hbpsi(\by,\bX) + \beta_j(\partial/\partial y_i)\hbpsi(\by,\bX) ]
\cr&=&
- \diag(\bpsi')\bX\hbA\be_j \hpsi_i
- \bV \be_i (\hbeta_j-\beta_j)
\eel
with $\bV\in\R^{n\times n}$ given by \eqref{def-V}.

\begin{proof}[Proof of \Cref{thm:main_robust} under \Cref{assum:g}(i)]
    Let us start with the proof under the strongly convex assumption
    \Cref{assum:g}(i). By the change of variable \eqref{eq:variable-change},
    we may assume that $\bSigma=\bI_p$ and $\bX$ has iid $N(0,1)$ entries.

    By \Cref{prop:lipschtiz} and \eqref{eq:lipschitz-br-D-inverse}
    we have that the function $\brho(\bX)$ in \eqref{eq:bfeta-brho}
    is $K$-Lipschitz where $K=(n^{-1/2} 2 L_*)$ and $L_*=\max(1,\mu^{-1})$.
    The Frobenius norm of the Jacobian of a $K$-Lipschitz function
    $\R^{n\times p}\to \R^n$ is bounded above by the rank of the Jacobian
    times its squared operator norm, and the operator norm of the Jacobian
    is bounded from above by $K$. Thus the Frobenius norm of the Jacobian
    of $\brho$ satisfies 
    $\sum_{i=1}^n\sum_{j=1}^p \|\frac{\partial \brho}{\partial x_{ij}}
    \|^2
    \le n K^2$ because the rank is at most $n$.
    We obtain that the quantity
    RHS in the right-hand side of \eqref{eq:new_cor_first} is bounded
    from above by 
    \begin{equation}
        4\E[L_*^2] + \sqrt{p+n} + \E[(p+\|\bX\|_{op}^2)4L_*^2]^{1/2}
        \label{eq:RHS_2_18}
    \end{equation}
    which is smaller than $\sqrt n \C(\gamma,\mu)$ thanks to
    \Cref{lemma:Davidson} to bound from above
    the expectation of $\|\bX\|_{op}^2$ for a random matrix with iid $N(0,1)$
    entries.

    Writing $D=(\|\bh\|^2+\|\br\|^2)^{1/2}$ for the denominator,
    we have $\brho = \br D^{-1}$.
    Using \eqref{eq:chain-rule-r} and the product rule
    $\frac{\partial}{\partial x_{ij}}\brho
    =D^{-1}\frac{\partial}{\partial x_{ij}}\br
    +
    \br \frac{\partial}{\partial x_{ij}} D^{-1}
    $,
    the last term in the left-hand side
    of \eqref{eq:new_cor_first} equals
    \begin{align}
    - 2\sum_{i=1}^n\sum_{j=1}^p
    \rho_i\be_j^\top\bX^\top\frac{\partial \brho}{\partial x_{ij}}
    &=
    2\sum_{i=1}^n\sum_{j=1}^p
    \Bigl[
    \frac{
        \rho_i \be_j^\top\bX^\top\diag(\bpsi')\bX\hbA\be_j \hpsi_i
        +
        \rho_i \be_j^\top\bX^\top\bV \be_i h_j
        }{D \sqrt n}
    -
    \rho_i\be_j^\top\bX^\top \br \frac{\partial(D^{-1})}{\partial x_{ij}}
    \Bigr]
    \nonumber
    \\&=
    2\df \|\brho\|^2
    +
    \frac{2 \bh^\top\bX^\top\bV\brho}{D\sqrt n}   
    -
    2\sum_{i=1}^n\sum_{j=1}^p
    \rho_i\be_j^\top\bX^\top \br \frac{\partial(D^{-1})}{\partial x_{ij}}
    \label{eq:key_algebra-1}
    \end{align}
    thanks to $\hpsi_i /(D\sqrt n) = \rho_i$ and
    $\df=\sum_{j=1}^p \be_j^\top\bX^\top\diag(\bpsi')\bX\hbA\be_j
    = \trace[\bX^\top\diag(\bpsi')\bX\hbA]$ for the first term.
    For the second term,
    since $\bV$ is the Jacobian of $\hbpsi$ with respect to $\by$,
    by \Cref{prop:jacobian-psd-htheta} we have $\|\bV\|_{op} \le 1$.
    By the Cauchy-Schwarz inequality and using $\|\brho\|\le 1$,
    $\|\bh\|\le D$, the 
    absolute value of the second term is smaller
    than
    $2\|\bX n^{-1/2}\|_{op}$.
    By the Cauchy-Schwarz inequality, 
    the third term is smaller than
    $\|\bX^\top\br\|\|\brho\|
    [\sum_{i=1}^n\sum_{j=1}^p
    (\frac{\partial(D^{-1})}{\partial x_{ij}})^2
    ]^{1/2}$.
    Inequality \eqref{eq:lipschitz-D-inverse} shows that the gradient
    of the map $D^{-1}:\R^{n\times p}\to \R$ has Euclidean norm at most 
    $n^{-1/2}L_* D^{-1}$, that is,
    $[\sum_{i=1}^n\sum_{j=1}^p(\frac{\partial(D^{-1})}{\partial x_{ij}})^2
    ]^{1/2}
    \le
    n^{-1/2} L_* D^{-1}.$
    Hence using $\|\brho\|\le 1$ and $\|\br\|D^{-1}\le 1$, the third term
    in \eqref{eq:key_algebra-1} is bounded from above by
    $2\|\bX n^{-1/2}\|_{op} L_*$.
    In summary,
    for the last term on the left-hand side of \eqref{eq:new_cor_first},
    \begin{equation}
    \Big|
    2\sum_{i=1}^n\sum_{j=1}^p
    \rho_i\be_j^\top\bX^\top\frac{\partial \brho}{\partial x_{ij}}
    +
    2\df \|\brho\|^2
    \Big|
    \le
    2 \|\bX n^{-1/2}\|_{op}
    +
    2 L_* \|\bX n^{-1/2}\|_{op}
    \label{eq:negligible-term-cross}
    \end{equation}
    which satisfies
    $\E[\eqref{eq:negligible-term-cross}]\le \C(\gamma,\mu)$
    by \Cref{lemma:Davidson}.
    For the term
    $\sum_{j=1}^p(
\sum_{i=1}^n \frac{\partial \rho_i}{\partial x_{ij}} )^2$ in
the left-hand side of \eqref{eq:new_cor_first},
\begin{equation}
\frac{\trace[\bV] h_j}{D \sqrt n} +
\sum_{i=1}^n \frac{\partial \rho_i}{\partial x_{ij}}
= - \frac{ \hbpsi{}^\top\diag(\bpsi')\bX\hbA\be_j}{D \sqrt n}
+ \sum_{i=1}^n r_i \frac{\partial D^{-1}}{\partial x_{ij}} 
\label{eq:key_algebra-2}
\end{equation}
by \eqref{eq:chain-rule-r}, for any fixed $j\in[p]$.
For the first term on the right-hand side we have
$\sum_{j=1}^p
(\frac{1}{D\sqrt n}\hbpsi{}^\top\diag(\bpsi')\bX\hbA\be_j)^2=
\|\hbA{}^\top \bX^\top \diag(\bpsi')\brho\|^2
\le
\|\hbA\|_{op}^2 \|\bX\|_{op}^2$.
For the second term on the right-hand side of \eqref{eq:key_algebra-2}, by the Cauchy-Schwarz inequality
$\sum_{j=1}^p|\sum_{i=1}^n r_i \frac{\partial D^{-1}}{\partial x_{ij}}|^2
\le \|\br\|^2 n^{-1} L_*^2 D^{-2} \le n^{-1} L_*^2$
since the norm of the gradient of $D^{-1}$ satisfies
$[\sum_{i=1}^n\sum_{j=1}^p(\frac{\partial(D^{-1})}{\partial x_{ij}})^2
]^{1/2}
\le
n^{-1/2} L_* D^{-1}$ again thanks to \eqref{eq:lipschitz-D-inverse}.
Consequently, $\ba,\bb\in\R^p$ defined componentwise as
\begin{equation}
\label{eq:a_j-b_j}
\textstyle
a_j = - \tfrac{\trace[\bV]h_j}{D\sqrt n},
~
b_j=
\sum_{i=1}^n \frac{\partial \rho_i}{\partial x_{ij}}
\quad
\text{ satisfy }
\quad
\|\eqref{eq:key_algebra-2}\|
=
\|\bb-\ba\|\le
\|\hbA\|_{op}\|\bX\|_{op} +  n^{-1/2}L_*.
\end{equation}
For the difference of squares,
$|\|\bb\|^2 - \|\ba\|^2|
=
|\|\bb-\ba\|^2 + 2(\bb-\ba)^\top\ba|
\le 
\|\bb-\ba\|^2 + \|\bb-\ba\| 2 \|\ba\|.$
Next, $\|\ba\|^2 \le n^{-1} \trace[\bV]^2 \|\bh\|^2/D^2 \le n$
since $\|\bh\|\le D$ and $0\le\trace[\bV]\le n$
by \Cref{prop:jacobian-psd-htheta}. Thus
\begin{equation}
\Big|
\sum_{j=1}^p(
\sum_{i=1}^n \frac{\partial \rho_i}{\partial x_{ij}} )^2
-
\trace[\bV]^2 \frac{\|\bh\|^2}{nD^2}
\Big|
\le 
    \bigl(\|\hbA\|_{op}\|\bX\|_{op} + \frac{L_*}{\sqrt n}\bigr)^2
+
\bigl(\|\hbA\|_{op}\|\bX\|_{op} + \frac{L_*}{\sqrt n}\bigr)
2\sqrt n
\label{divergcence-square-to-bound}
\end{equation}
Using \Cref{lemma:Davidson}
and \eqref{eq:hbA-op}
again,
$\E[\eqref{divergcence-square-to-bound}]
\le \C(\gamma,\mu)$. An application of \eqref{eq:new_cor_first}
combined with the bounds in expectation obtained for
\eqref{eq:negligible-term-cross} and \eqref{divergcence-square-to-bound}
thus provides
$
\E\Big|
\|\bX^\top\brho\|^2
+(2\df - p)
\|\brho\|^2
-
\trace[\bV]^2 \|\bh\|^2 \frac{1}{D^2 n}
\Big|
\le 
\C(\gamma,\mu) \sqrt n
$
which is exactly \Cref{thm:main_robust} under \Cref{assum:g}(i).
We mention in passing that using the notation and the bound in
\eqref{eq:a_j-b_j},
the first term in \eqref{eq:new_cor_second} satisfies
\begin{equation}
\Big|
\sum_{j=1}^p
\brho^\top\bX\be_j\sum_{i=1}^n
\frac{\partial \rho_i}{\partial x_{ij}}
+ \frac{\trace[\bV]\brho^\top\bX\bh}{D\sqrt n}
\Big|
=\Big|
\brho^\top\bX(\bb-\ba)
\Big|
\le
\|\hbA\|_{op}\|\bX\|_{op}^2 + n^{-1/2}\|\bX\|_{op}^2L_*
\label{eq:note-in-passing}
\end{equation}
so that $\E[\eqref{eq:note-in-passing}]\le \C(\mu,\gamma)$ by \Cref{lemma:Davidson} and \eqref{eq:hbA-op}.
\end{proof}
\begin{proof}[Proof of \Cref{thm:main_robust} under \Cref{assum:g}(ii)]
    If $\mu_\rho>0$, by Proposition~\ref{prop:jacobian-psd-htheta} and with
    the notation of \Cref{lemma:ninja-variant}, the matrix
    $\bV=\diag(\bpsi') - \diag(\bpsi')\bX\hbA\bX^\top\diag(\bpsi')$
    is symmetric psd. Let $(\bu_1,...,\bu_{q})$ be an orthonormal basis
    of $\ker(\diag(\bpsi')\bX\hbA\bX^\top\diag(\bpsi'))^\perp$ 
    and note that $q\ge n-p$ since $\bX$ and the $n\times n$ matrix inside $\ker(\cdot)$
    have rank at most $p$. Since $\bV$ is psd,
    $\trace[\bV]\ge \sum_{i=1}^q \bu_i^\top \bV \bu_i =
    \sum_{i=1}^q \bu_i^\top \diag(\bpsi') \bu_i$. 
    Since $\psi'\ge\mu_\rho$ by \Cref{assum:g}(ii), we obtain
    $\trace[\bV]\ge \mu_\rho q \ge \mu_\rho(n-p)\ge n\mu_\rho(1-\gamma)$ as desired.

    The argument and notation are the same as in the previous proof.
    We apply \eqref{eq:new_cor_first} with only two notable differences.
    First,
    by \Cref{prop:lipschtiz2-gamma-less-one},
    $L_*$ is now random and can be chosen (enlarging $L_*$ if necessary) as 
    \begin{equation}
        \label{Lstar_ii}
L_* = \C(\mu_\rho) \max(1, \|n^{-1/2}\bX\|_{op})\big/ \min(1,\phi_{\min}(\tfrac 1 n \bX^\top\bX)).
    \end{equation}
    The Jacobian of $\brho(\bX)$ has operator
    norm at most $2L_* n^{-1/2}$ by \eqref{eq:lipschitz-br-D-inverse} and the gradient of $D^{-1}$ has
    Euclidean norm at most $L_* n^{-1/2}$ by
    \eqref{eq:lipschitz-D-inverse}.
    Second, we use the operator norm bound 
    $\|\hbA\|_{op} \le \frac{1}{n\mu_\rho} \phi_{\min}(\frac 1 n\bX^\top\bX)^{-1}$
    by \eqref{eq:hbA-op}.
    Inequality \eqref{eq:RHS_2_18} [upper bound on the quantity RHS in \eqref{eq:new_cor_first}], inequality \eqref{eq:negligible-term-cross} [upper bound on the negligible terms
    in \eqref{eq:key_algebra-1}] and
    inequality \eqref{divergcence-square-to-bound} are still valid,
    and these three upper bounds are, in expectation, smaller
    than $\C(\mu_\rho,\gamma) \sqrt n$ since by \Cref{lemma:Davidson,lemma:Edelman} and the Cauchy-Schwarz inequality
    we have
    $\E[\max(1, \|n^{-1/2}\bX\|_{op})^k\big/ \min(1,\phi_{\min}(\tfrac 1 n \bX^\top\bX))^{k'}]\le \C(\gamma,k,k')$ for any absolute constants
    $k,k'$; for our purpose we may take $k=4$, $k'=2$.
\end{proof}

The proofs under \Cref{assum:g}(iii.a) and (iii.b) are
more technical as the right-hand side of \Cref{big_cor},
\eqref{eq:negligible-term-cross} and \eqref{divergcence-square-to-bound}
can only be controlled in a high-probability event $\Omega$.
To overcome this problem, we use the argument detailed in the next section.
The formal proofs of \Cref{thm:main_robust}
under \Cref{assum:g}(iii.a) and (iii.b)
are provided in \Cref{sec:appendix_lasso,sec:appendix_huber_lasso_new}.

\subsection{Kirszbraun's theorem:
controlling derivatives outside high-probability events
}
\label{sec:kirszbraun}

Under \Cref{assum:g}(i) and (ii), the proof of \Cref{thm:main_robust}
leverages that for a fixed noise vector $\bep$, the function
$\bX\mapsto \brho(\bX)$ defined in \eqref{eq:bfeta-brho} satisfies
\begin{equation}
\|\brho(\bX)-\brho(\bX')\|
\le n^{-1/2} 2 L_* \|\bX - \bX'\|_F
\label{eq:lispchitzLstar_Discussion}
\end{equation}
for some deterministic $L_*=\C(\mu)$ under \Cref{assum:g}(i)
and a random but integrable $L_*$ given by \eqref{Lstar_ii}
under \Cref{assum:g}(ii).
For the Lasso and Huber Lasso in \Cref{assum:g}(iii),
we are only able to derive 
inequality \eqref{eq:lispchitzLstar_Discussion} for $\bX,\bX'\in U_\bep$
for some open set $U_{\bep}\subset\R^{n\times p}$ such that the
event $\Omega = \{ \bX\in U_\bep\}$ has $\P(\Omega)\to 1$.
We use the following variant of \Cref{big_cor} to prove \Cref{thm:main_robust}
in such situations where the derivatives of $\brho(\bX)$ cannot be controlled
in a small probability event $\bX\not\in U_\bep$.

\begin{corollary}
    \label{kirsz_cor}
    Let $L>0$ and $U\subset\R^{n\times p}$ be open.
    Assume that $\bX$ has iid $N(0,1)$ entries,
    that $\brho:\R^{n\times p}\to\R^n$ is weakly differentiable
    and that $\|\brho(\bX)\|\le 1$ and
    $\|\brho(\bX)-\brho(\bX')\| \le Ln^{-1/2} \|\bX-\bX'\|_F$
    for any two $\bX\in U,\bX'\in U$.
    Then for
    $\RHS=
    L^2
+
(1+L)
\sqrt{p+n} 
$ we have
    \begin{align}
        \E\Bigl[I\{\bX\in U\}\Big|
\|\bX^\top\brho\|^2
- p \|\brho\|^2
-
\sum_{j=1}^p\Big(
\sum_{i=1}^n \frac{\partial \rho_i}{\partial x_{ij}} \Big)^2
- 2\sum_{i=1}^n\sum_{j=1}^p
\rho_i\be_j^\top\bX^\top\frac{\partial \brho}{\partial x_{ij}}
\Big|\Bigr]
&\le 
\C \RHS,
\label{Kirszbraun_first}
\\
\E\Bigl[I\{\bX\in U\}\Big|
\sum_{j=1}^p
\brho^\top\bX\be_j\sum_{i=1}^n
\frac{\partial \rho_i}{\partial x_{ij}}
-
\sum_{j=1}^p\Big(
\sum_{i=1}^n \frac{\partial \rho_i}{\partial x_{ij}} \Big)^2
- \sum_{i=1}^n\sum_{j=1}^p
\rho_i\be_j^\top\bX^\top\frac{\partial \brho}{\partial x_{ij}}
\Big|
\Bigr]
&\le \C \RHS.
\label{Kirszbraun_second}
    \end{align}
\end{corollary}
\begin{proof}
    By Kirszbraun's theorem, there exists $\brhobar:\R^{n\times p}\to\R^n$ such that $\brhobar(\bX)=\brho(\bX)$ for $\bX\in U$ and
    $\brhobar$ is $n^{-1/2}L$-Lipschitz on the whole $\R^{n\times p}$.
    Applying \Cref{big_cor} to $\brhobar$, the right-hand sides
    of \eqref{eq:new_cor_first} and \eqref{eq:new_cor_second} for $\brhobar$
    are bounded
    from above by an absolute constant times
    $L^2 + \sqrt{p+n} + \E[p+\|\bX\|_{op}^2]^{1/2}L$
    thanks to $\sum_{i=1}^n\sum_{j=1}^p \|\frac{\partial\brhobar}{\partial x_{ij}}\|^2
    \le L^2$ since the Frobenius norm of the Jacobian of a Lipschitz map
    $\R^{np}\to\R^n$ is bounded by $n$ times the square of the Lipschitz constant. The left-hand side of \eqref{eq:new_cor_first} for $\brhobar$
    is bounded from below
    by the left-hand side of \eqref{Kirszbraun_first} since $I\{\bX\in U\}\le 1$, hence \eqref{Kirszbraun_first} for $\brho$ follows from \eqref{eq:new_cor_first} for $\brhobar$.
    Similarly, \eqref{Kirszbraun_second} for $\brho$
    follows from \eqref{eq:new_cor_second} for $\brhobar$.
\end{proof}

Consequently, as long as \eqref{eq:lispchitzLstar_Discussion}
and $\P(\bX\in U_\bep)\to 1$ hold, and the remainder
terms \eqref{divergcence-square-to-bound} and \eqref{eq:negligible-term-cross}
are negligible for $\bX\in U_\bep$, the same algebra as in
\Cref{sec:proof-sketch}
can be used to derive a version of \Cref{thm:main_robust} that holds
in the event $\Omega=\{\bX\in U_\bep\}$.
This approach is used in \Cref{sec:appendix_lasso,sec:appendix_huber_lasso_new} for the formal proof of \Cref{thm:main_robust} under \Cref{assum:g}(iii.a)
and (iii.b).
}

\subsection{Approximation of the multiplicative factors in $\hat R$
    \label{sec:divergence}
for arbitrary loss and penalty $(\rho,g)$}
For general penalty function, however, no closed form solution is available.
Still, it is possible to approximate the multiplicative factors appearing in
$\hat R$ using the following Monte Carlo scheme.
Since
$\trace[(\partial/\partial \by)\hbpsi]$
and $\df=\trace[(\partial/\partial\by)\bX\hbbeta]$
are the divergence of the vector fields $\by\mapsto \hbpsi$
and $\by \mapsto \bX\hbbeta$
respectively, we can use the following Monte Carlo approximation
of the divergence of a vector field, which was
suggested at least as early as
\cite{metzler2016denoising},
and for which accuracy guarantees are proved in \cite{bellec_zhang2018second_stein}.

Let $\bF:\R^n\to\R^n$ be a vector field, and let $\bz_k, k=1,...,m$
be iid standard normal random vector in $\R^n$.
Then for some small scale parameter $a>0$, we approximate the divergence of $\bF$
at a point $\by\in\R^n$ by
$$
\hat{~\dv~} \bF( \by ) = \frac 1 m \sum_{k=1}^m a^{-1} \bz_k^\top
\big[
    \bF(\by + a \bz_k) - \bF(\by)
\big].
$$
Computing the quantities $\bF(\by + a \bz_k)$ at the perturbed
response vector $\by + a \bz_k$ for $\bF(\by) = \bX\hbbeta$
or $\bF(\by)=\psi(\by-\bX\hbbeta)$ require the computation
of the $M$-estimator $\hbbeta(\by+a\bz_k,\bX)$ at the perturbed response.
If $\hbbeta(\by,\bX)$ has already been computed as a solution
to \eqref{M-estimator-rho} by an iterative algorithm, one
can use $\hbbeta(\by,\bX)$ as a starting point of the iterative
algorithm to compute $\hbbeta(\by+a\bz_k,\bX)$ efficiently,
since for small $a>0$ and by continuity,
$\hbbeta(\by,\bX)$ should provide a good initialization.
We refer to \cite{bellec_zhang2018second_stein} for an analysis
of the accuracy of this approximation.

\begin{figure}[ht]
    \centering
\includegraphics[width=4in]{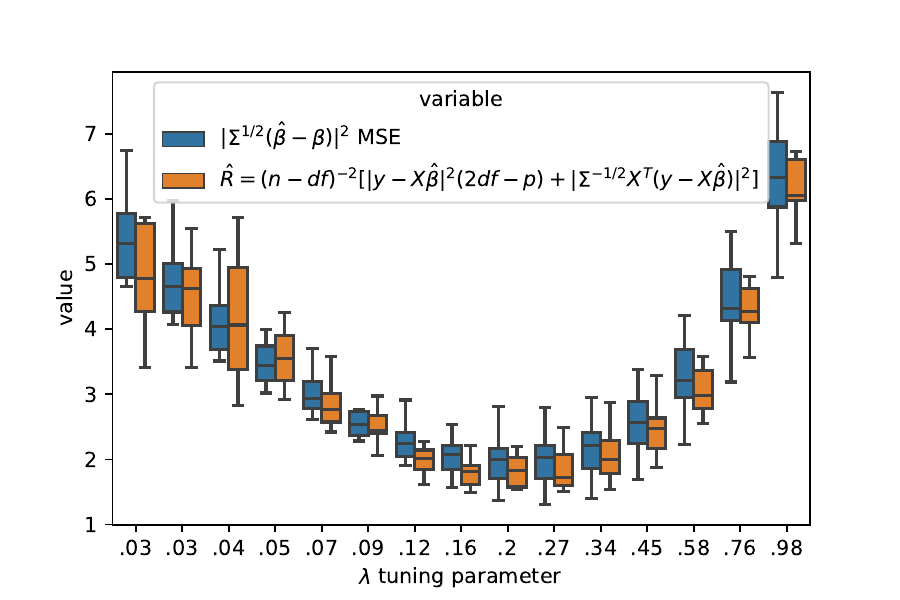}
\caption[Boxplots of $\hat R$ and its target $n=400,p=500$, nuclear normal penalty]{
    \small
    Boxplots of $\|\bSigma^{1/2}(\hbbeta-\bbeta)\|^2$
    and $\hat R$ over 10 repetitions, with $n=400, p=500$,
    the square loss and nuclear norm penalty
    $g(\bb)=\lambda \|\text{mat}(\bb)\|_{nuc}$
    for different values of the tuning parameter $\lambda$.
    The function $\text{mat}:\R^p \to \R^{20\times 25}$ maps
    $\R^p$ to matrices of size $20\times25$ so that the inverse
    map is the usual vectorization operator.
    For the true $\bbeta$, $\text{mat}(\bbeta)$ is rank 3.
    The Monte Carlo scheme of \Cref{sec:divergence} is used to
    compute $\df$ with $a=0.01$ and $m=100$.
    The full simulation setup is described in \Cref{sec:trace-norm-simu}.
}
\label{fig:nuclear-norm-hat-R}
\end{figure}

Hence, even in situations where no closed form expressions for the Jacobians 
$(\partial/\partial \by)\hbpsi$
and $(\partial/\partial\by)\bX\hbbeta$ are available,
the estimate $\hat R$ of the out-of-sample error of the $M$-estimator  $\hbbeta$
can be used by replacing the divergences 
$\trace[(\partial/\partial \by)\hbpsi]$
and $\df$
by their Monte Carlo approximations.
\Cref{fig:nuclear-norm-hat-R} illustrates the use of this Monte Carlo
scheme by showing boxplots $\hat R$
and its target $\|\bSigma^{1/2}(\hbbeta-\bbeta)\|^2$ over 10 repetitions
for the nuclear norm penalty over a range of tuning parameters.

Although this approximation scheme induces some computational
overhead as it requires computation of several $\hbbeta(\by+a\bz_k,\bX)$,
we stress that this approximation scheme is not needed
for the $\ell_1$ and Elastic-Net penalty since
explicit formulae are available (cf. Propositions~\ref{prop:specific-ENet} and~\eqref{prop:Huber-Lasso}). 
For these two commonly used penalty functions
the computational burden
of computing $\df$ and $\trace[(\partial/\partial)\hbpsi]$ 
is negligible.

\begin{figure}[!ht]
\includegraphics[width=5.8in]{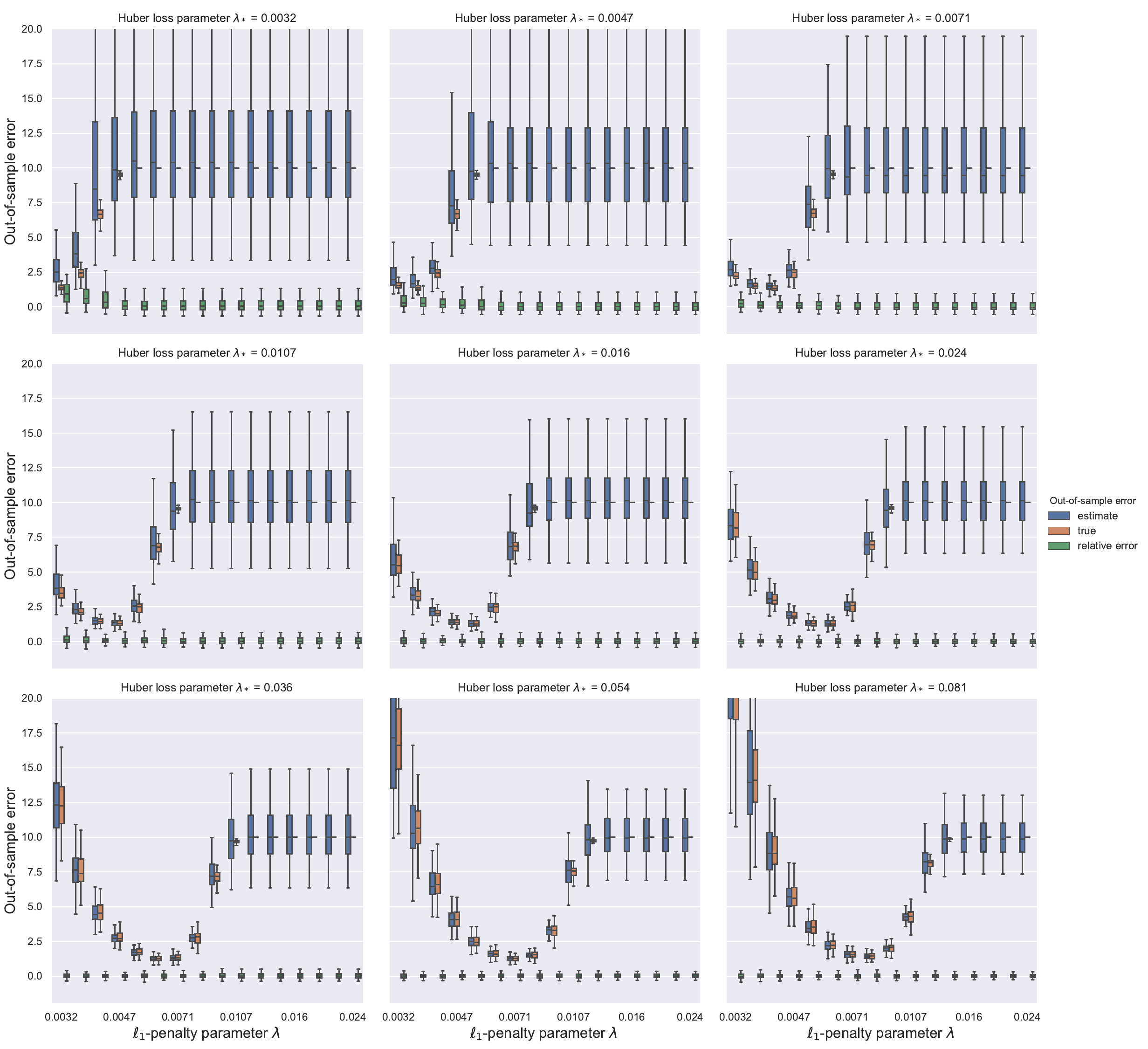}
\caption[Boxplots of $\hat R$ and its target $n=1001,p=1000$, Huber Lasso]{
    Boxplots over 100 repetitions of 
    the out-of-sample error $\|\bSigma^{1/2}(\hbbeta-\bbeta)\|^2$
    for the Huber Lasso  with parameters $(\lambda,\lambda_*)$,
    the estimate $\hat R$ in \eqref{eq:hat-R-Huber-Lasso}
    and the relative error
    $|1-\hat R \big/ \|\bSigma^{1/2}(\hbbeta-\bbeta)\|^2 |$.
    The heatmap below displays the average over the same
    100 repetitions of $\hat R$
    (Left)
    and $\|\bSigma^{1/2}(\hbbeta-\bbeta)\|^2$ (Right).
    The experiment is described in \Cref{sec:simu}.
}
\label{fig:boxplots}
\includegraphics[width=6.7in]{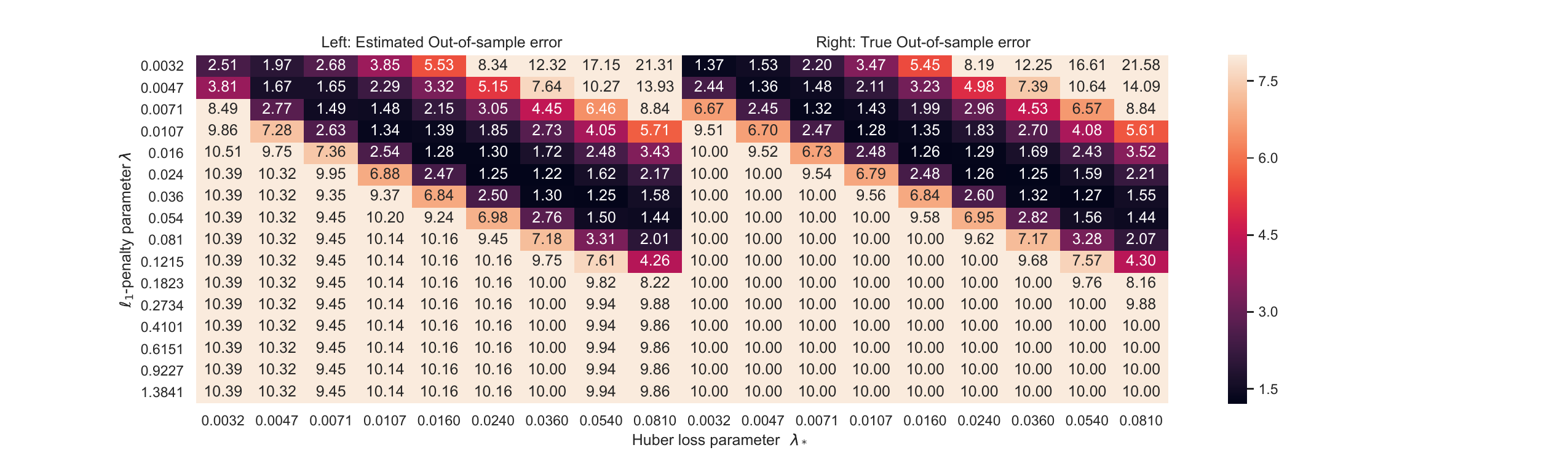}
\end{figure}

\subsection{Simulation study}
\label{sec:simu}

\subsubsection{Huber Lasso}
We illustrate the above result with a short simulation study.
For given tuning parameters $\lambda,\lambda_*>0$,
the $M$-estimator
$\hbbeta$ is the Huber Lasso estimator \eqref{M-estimator-rho}
with loss $\rho$ and penalty $g$ given in Proposition~\ref{prop:Huber-Lasso},
and the estimate $\hat R$ is given by \eqref{eq:hat-R-Huber-Lasso}.
We set $n=1001, p=1000$, $\bSigma=\bI_p$
and $\bbeta$ has 100 nonzero coefficients
all equal to $10 p^{-1/2}$.
The components of $\bep$ are iid with $t$-distribution 
with $2$ degrees-of-freedom (so that the variance of each component
does not exist).
Define the sets
$\Lambda = \{ 0.1 n^{-1/2}  (1.5)^k, k=0,...,15\}$ and
$\Lambda_* = \{ 0.1 n^{-1/2} (1.5)^k, k=0,...,8\}$.
For each $(\lambda,\lambda_*)$ in the discrete grid
$\Lambda \times \Lambda_*$,
the estimator 
$\hat R$, its target $\|\bSigma^{1/2}(\hbbeta - \bbeta)\|^2$
and the relative error $|1-\hat R/\|\bSigma^{1/2}(\hbbeta - \bbeta)\|^2|$
are reported, over 100 repetitions, in the boxplots in \Cref{fig:boxplots}.
\Cref{fig:boxplots} provides also a heatmap of the average of $\hat R$
and $\|\bSigma^{1/2}(\hbbeta - \bbeta)\|^2$ over the same 100 repetitions.

The plots show that the estimate $\hat R$
accurately estimates $\|\bSigma^{1/2}(\hbbeta - \bbeta)\|^2$ across the grid
$\Lambda \times \Lambda_*$,
at the exception of the lowest value of the Huber loss parameter
$\lambda_*$ coupled
with the two lowest values for the penalty parameter $\lambda$
as seen on the left of the top left boxplot in \Cref{fig:boxplots}.
These low values for $(\lambda,\lambda_*)$ lead to small
values for $(|\hat I|-|\hat S|)^2$ in the denominator of \eqref{eq:hat-R-Huber-Lasso}. This provides additional evidence that
$\hat R$ should not be trusted for low values 
of $\trace[(\partial/\partial\by)\hbpsi]^2/n^2$
(cf. \Cref{subsec:too_small}).
These inaccurate estimations for low values for $(\lambda,\lambda_*)$
do not contradict the theoretical results, as the proof
of \Cref{thm:main_robust} bounds from above
$(\trace[(\partial/\partial\by)\hbpsi]/n)^2
|\hat R - \|\bSigma^{1/2}\bh\|^2|
=(|\hat I|/n-|\hat S|/n)^2
|\hat R - \|\bSigma^{1/2}\bh\|^2|$ 
and upper bounds on
$|\hat R - \|\bSigma^{1/2}\bh\|^2|$
are not guaranteed by \Cref{thm:main_robust}
when $(\frac 1 n \trace \bV)^2=(\frac 1 n (|\hat I|-|\hat S|))^2$ is close to 0.
Furthermore, \Cref{fig:boxplots} suggests that the estimate $\hat R$
is accurate for $(\lambda,\lambda_*)$ smaller than 
the values $(\lambda,\lambda_*)$
required in \Cref{assum:g}(iii.b).
\pb{
This suggests
that the validity of $\hat R$ may hold for smaller tuning parameters
than those required 
by \Cref{assum:g}(iii.a) or (iii.b). The recent result from
\cite{celentano2020lasso} also confirms this: The theory for the Lasso
\cite{celentano2020lasso} holds for any constant tuning parameter $\lambda$, with no assumption
of the form $\lambda\ge\sigma\lambda_*$ as required
in \Cref{assum:g}(iii.a) for the proofs in the present paper.
}

\subsubsection{Square loss and nuclear norm penalty}
\label{sec:trace-norm-simu}

A second simulation study is provided with the square loss $\rho(u)=u^2/2$
and nuclear norm penalty.
With $n=400$, $p=500$,
a linear isomorphism $\text{mat}:\R^p\to\R^{20\times 25}$ is fixed
so that the inverse map is the usual vectorization operator.
The true $\bbeta$ is such that $\text{mat}(\bbeta)$ has iid $N(0,1)$
entries in the first three columns and zeros in the remaining columns
so that $\text{mat}(\bbeta)$ is rank 3.
The covariance matrix $\bSigma\in\R^p$ is defined as
$\bSigma = \bW/(5p)$ where $\bW$ is a Wishart matrix with identity covariance
and $5p$ degrees of freedom; $\bSigma$ is generated once and is the same
across the repetitions.
The noise $\bep$ has iid $N(0,2)$ entries.
For 10 repetitions, $(\by,\bX)$ are generated and M-estimators with
penalty proportional to the nuclear norm,
$g(\bb) = \lambda \|\text{mat}(\bb)\|_{nuc}$, are computed
for each value of $\lambda$
in $\{0.5 \cdot 1.3^k n^{-1/2}, k=0,1,2,...,14\}$.
For each $\lambda$, the estimate $\df$
is computed with the Monte Carlo scheme of \Cref{sec:divergence}
with $a=0.01$ and $m=100$ and used to construct the estimate $\hat R$
of the out-of-sample error.
The resulting boxplots, over 10 repetitions, are given in \Cref{fig:nuclear-norm-hat-R}.

\section{Square loss}
\label{sec:square_loss}

Throughout this section $\rho(u)=u^2/2$ in \eqref{M-estimator-rho}
so that $\hbbeta$ is the regularized least-squares
\begin{equation}
    \hbbeta(\by,\bX) = \argmin_{\bb\in\R^p}
    \Big(
        \|\bX \bb - \by\|^2/(2n)
        + g(\bb)
    \Big).
    \label{hbbeta-squared-loss}
\end{equation}
for some convex penalty $g:\R^p\to\R$.
Here $\psi(u)=\rho'(u)=u$, $\diag(\bpsi')=\bI_n$ so that
$\hbpsi=\psi(\by-\bX\hbbeta)$,
$\trace[(\partial/\partial\by)\hbpsi]$ are simply given by
$$\hbpsi = \by-\bX\hbbeta,
\qquad
\trace[\bV]=\trace[(\partial/\partial\by)\hbpsi] = n -\df,
$$
i.e., $\hbpsi$ is the vector of residuals.
In regression or sequence model with Gaussian noise,
the quantity $\df$ was introduced in
\cite{stein1981estimation} where Stein's Unbiased Estimate (SURE) was developed,
showing that
$\E[\|\bX(\hbbeta-\bbeta)\|^2] = \E[\SURE]$
 where 
$\SURE= \|\by-\bX\hbbeta\|^2 + 2\sigma^2\df - \sigma^2 n$
when $\bep\sim N({\bf 0},\sigma^2\bI_n)$
under mild differentiability and integrability assumptions.
Numerous works followed with the goal to characterize
the quantity $\df$ for estimators of interest, see for instance
\cite{zou2007,tibshirani2012,kato2009degrees,dossal2013degrees}
for the Lasso and the Elastic-Net,
\cite{vaiter2012degrees} for the Group-Lasso,
\cite{minami2020degrees} for Slope and submodular regularizers,
\cite{chen2019degrees} for projection estimators,
among others.
A surprise of the present paper
is that for general penalty functions, $\df$ is not only useful to estimate
the in-sample error in
$\E[\|\bX(\hbbeta-\bbeta)\|^2] = \E[\SURE]$, but also
the out-of-sample error
$\|\bSigma^{\frac12}(\hbbeta-\bbeta)\|_2^2$.
Furthermore, 
$\E[\|\bX(\hbbeta-\bbeta)\|^2] = \E[\SURE]$
 requires normality
of $\bep$ while the estimate $\hat R$ of the present paper rely
on the normality of $\bX$ but not that of $\bep$.

\subsection{Estimation of the noise level and generalization error}

The simple algebraic structure of the square loss allows us
to provide generic estimators of the noise level $\sigma^2$ and the generalization
error 
$\sigma^2 + \|\bSigma^{\frac12}\bh\|^2$,
assuming that the components of $\bep$ are iid
with mean zero and variance $\sigma^2$.
The quantity $\sigma^2 + \|\bSigma^{\frac12}\bh\|^2$ can be seen
as the generalization error, since 
$\sigma^2 + \|\bSigma^{\frac12}\bh\|^2 = \E[(\bx_{new}^\top\hbbeta - Y_{new})^2| (\by,\bX)]$
where $(\bx_{new}^\top,Y_{new})$ is independent of $(\bX,\by)$
with the same distribution as any row of $(\bX,\by)\in\R^{n\times (p+1)}$.

When the components of $\bep$ are assumed iid, mean-zero with variance $\sigma^2$, the convergence
$\|\bep\|^2/n\to \sigma^2$ holds almost surely
by the law of large numbers, and
$|\|\bep\|^2/n - \sigma^2| = O_\P(n^{-\frac12})$
by the central limit theorem if the fourth moment of the entries of $\bep$
is uniformly bounded as $n,p\to+\infty$.
We may thus consider the estimation
targets
$\sigma_*^2 \defas \|\bep\|^2/n$
and
$\sigma_*^2 + \|\bSigma^{\frac12}\bh\|^2 $
for the noise level and generalization
error, respectively.
Results for $\sigma^2$ and $\sigma^2+\|\bSigma^{\frac12}\bh\|^2$
can be deduced up to an extra additive error term of order
$|\|\bep\|^2/n - \sigma^2|$ which converges to 0 almost surely
and that satisfies
$|\|\bep\|^2/n - \sigma^2| = O_\P(n^{-\frac12})$
under the uniformly bounded fourth moment assumption
on the components of $\bep$. Define
\begin{equation}
\begin{aligned}
    \hat \tau^2 &= \big(n-\df\big)^{-2}  \|\by - \bX\hbbeta\|^2 n,
\\
    \hat R & =\big(n-\df\big)^{-2} \big\{ \|\by - \bX\hbbeta\|^2
        (2\df - p)
+ \|\bSigma^{-\frac12}\bX^\top(\by - \bX\hbbeta)\|^2\big\}
,
\\
        \hat\sigma^2 &=
\big(n - \df\big)^{-2}
\big\{
    \|\by - \bX\hbbeta\|^2(n-(2\df - p))
 - \|\bSigma^{-\frac12}\bX^\top(\by - \bX\hbbeta)\|^2
\big\}.
\end{aligned}
\label{tau_hat_R_hat_sigma_hat}
\end{equation}

\pb{
\begin{theorem}
    \label{thm:main-squared-loss}
    Let \Cref{assum:Sigma} be fulfilled. Set $\rho(u)=u^2/2$ (square loss) and assume that one of \Cref{assum:g}(i), (ii) or (iii.a) is
    fulfilled.
    Then almost surely
\begin{align}
    \big|
    \{\|\bSigma^{1/2}\bh\|^2 + \sigma_*^2\}
    -\hat\tau^2
    \big|
 &\le (1-\df/n)^{-2} 
\big(
        \|\by-\bX\hbbeta\|^2/n + \|\bSigma^{\frac12}\bh\|^2
    \big)
    \Rem_*
,
\label{eq:main-squared-loss}
\\
\big|
\|\bSigma^{1/2}\bh\|^2 - \hat R
\big|
&\le
(1-\df/n)^{-2} 
\big(
        \|\by-\bX\hbbeta\|^2/n + \|\bSigma^{\frac12}\bh\|^2
    \big)
\Rem_*
,
\nonumber
\\
\big|
\sigma_*^2 - \hat\sigma^2
\big|
&\le
(1-\df/n)^{-2} 
\big(
        \|\by-\bX\hbbeta\|^2/n + \|\bSigma^{\frac12}\bh\|^2
    \big)
\Rem_*
\nonumber
\end{align}
where $\Rem_*$ and $(1-\df/n)^{-1}$ in the right-hand side satisfy
\begin{enumerate}
    \item
$\E[\Rem_*]\le \C(\gamma,\mu)n^{-1/2}$
and
$(1-\df/n)^{-1} \le 
 1 + \frac{1}{n\mu}\|\bX\bSigma^{-1/2}\|_{op}^2
$ a.s.
under \Cref{assum:g}(i), so that
$\P[(1-\df/n)^{-1} \le 1 + (\sqrt \gamma + 1 + t)^2/\mu]
\ge 1 - \exp(-t^2/2)$.
\item
$\E[\Rem_*]\le \C(\gamma,\mu_\rho)n^{-1/2}$
and $(1-\df/n)^{-2}\le(1-\gamma)^{-2}$
a.s.
under \Cref{assum:g}(ii).
\item
    $\E[I\{\Omega\}\Rem_*]\le \C(\gamma,\varphi) n^{-1/2}$
    and
    $(1-\df/n)^{-2}
    \le \C(\gamma,\varphi)$
    in $\Omega$ for some event
    $\Omega$ of probability
    converging to one
    under \Cref{assum:g}(iii.a).
\end{enumerate}
\end{theorem}
\begin{proof}[Proof of \Cref{thm:main-squared-loss} under \Cref{assum:g}(i)]
    First, apply the change of variable \eqref{eq:variable-change}
    as in the proof  of \Cref{thm:main_robust}.
    We apply inequality \eqref{eq:new_cor_second} to $\brho$
    in \eqref{eq:br-bh-proof-overview}. The quantity RHS
    is bounded from above in \eqref{eq:RHS_2_18} while the three terms 
    in the left-hand side of \eqref{eq:new_cor_second}
    satisfy the approximations
    \eqref{eq:note-in-passing},
    \eqref{divergcence-square-to-bound} and
    \eqref{eq:negligible-term-cross}. This gives,
    by the triangle inequality,
    \begin{equation}
        \E\big|
        -\tfrac{\trace[\bV]}{D\sqrt n} \bh^\top\bX^\top\brho
        -\tfrac{\trace[\bV]^2}{nD^2}\|\bh\|^2 + \df \|\brho\|^2
        \big|
        \le
        \eqref{eq:RHS_2_18}
        +
        \E\bigl[
        \eqref{eq:negligible-term-cross}
        +
        \eqref{divergcence-square-to-bound}
        +
        \eqref{eq:note-in-passing}
        \bigr].
        \label{squareloss_RHS-sum}
    \end{equation}
    where each term on the right-hand side denotes the upper bound
    of the corresponding numbered equation in the previous section.
    For the square loss $\psi$ is the identity so that
    $\brho =(D\sqrt n)^{-1}(\bep - \bX\bh)$
    and $\trace[\bV]=n-\df$.
    Using these identities for the first term,
    the quantity inside the absolute value in the left-hand side
    of \eqref{squareloss_RHS-sum} equals
    $$
    n\|\brho\|^2
    -\tfrac{\trace[\bV]^2}{nD^2} \|\bh\|^2
    -  \tfrac{\trace[\bV]}{D\sqrt n}\bep^\top\brho
    =
    n\|\brho\|^2
    -\tfrac{\trace[\bV]^2}{nD^2} (\|\bh\|^2 + \tfrac1n\|\bep\|^2)
    + \Rem
    $$
    where $\Rem= 
    (D^2n)^{-1} \trace[\bV]
    (
    \trace[\bV] \tfrac1n \|\bep\|^2
    - \bep^\top(\bep - \bX\bh)
    )
    $.
    By the triangle inequality,
    \begin{equation}
        \E\big|
        n\|\brho\|^2
        -\tfrac{\trace[\bV]^2}{nD^2}(\|\bh\|^2 +\tfrac1n \|\bep\|^2)
        \big|
        \le
        \E|\Rem|
        +
        \eqref{eq:RHS_2_18}
        +
        \E\bigl[
        \eqref{eq:negligible-term-cross} 
        +
        \eqref{divergcence-square-to-bound}
        +
        \eqref{eq:note-in-passing}
        \bigr].
        \label{squareloss_RHS-sum_with_REM}
    \end{equation}
    \Cref{prop:eta_1} provides the bound $\E|\Rem|\le \C(\gamma) \sqrt n$
    while the other terms in the right-hand side have been shown
    to be smaller than $\C(\gamma,\mu)\sqrt n$ in the proof 
    of \Cref{thm:main_robust}.
    If we define $\Rem_*$ such that $n\Rem_*$ equals
    the random variable inside the expectation in the left-hand side of
    \eqref{squareloss_RHS-sum_with_REM},
    then the bound on the first line of \eqref{eq:main-squared-loss} is
    satisfied and
    $\E[\Rem_*]\le \C(\gamma,\mu)n^{-1/2}$.
    The desired bound on the second line in \eqref{eq:main-squared-loss}
    follows as a special case of   \Cref{thm:main_robust} 
    for the square loss, for a different $\Rem_*$ again satisfying
    $\E[\Rem_*]\le \C(\gamma,\mu)n^{-1/2}$.
    The bound on the third line
    in \eqref{eq:main-squared-loss} is obtained by taking the difference
    of the first two lines, where this third $\Rem_*$ is the sum of
    the $\Rem_*$ in the first line and the $\Rem_*$ in the second line.

    It remains to bound $(1-\df/n)^{-1}$.
    If $\bX$ is fixed and $\by,\tby$ are two response vectors
    with respective M-estimator $\hbbeta,\tbbeta$,
    multiplying by $(\hbbeta-\tbbeta)$ the KKT conditions
    $\bX^\top(\by-\bX\hbbeta) \in \partial g(\hbbeta)$
    and
    $\bX^\top(\tby-\bX\tbbeta) \in \partial g(\tbbeta)$
    and taking the difference, we find
    \begin{align}
        n (\partial g(\hbbeta) - \partial g(\tbbeta))^\top(\hbbeta-\tbbeta)
        + \|\bX(\hbbeta-\tbbeta)\|^2
        &\ni
        (\by-\tby)^\top\bX(\hbbeta-\tbbeta).
    \end{align}
    Since the infimum of
    $(\partial g(\hbbeta) - \partial g(\tbbeta))^\top(\hbbeta-\tbbeta)$
    is at least $\mu \|\bSigma^{\frac12}(\hbbeta-\tbbeta)\|^2$
    by strong convexity of $g$,
    this proves that
    $(n\mu\|\bX\bSigma^{-1/2}\|_{op}^{ \pb{ -2} } + 1) \|\bX(\hbbeta-\tbbeta)\|^2
    \le
    (\by-\tby)^\top\bX(\hbbeta-\tbbeta)$
    in $\Omega$.
    Thus $\by\mapsto \bX\hbbeta(\by,\bX)$
    is $L$-Lipschitz and the operator norm of $(\partial/\partial \pb{\by})\bX\hbbeta$ is bounded by $L$ for 
    $L=(n\mu\|\bX\bSigma^{-1/2}\|_{op}^{ \pb{ -2} } + 1)^{-1} < 1$.
    Thus $\df=\trace[(\partial/\partial \pb{ \by })\bX\hbbeta]\le nL$ and
    $(1-\df/n)^{-1} \le (1-L)^{-1}
    = 1 + \frac{1}{n\mu}\|\bX\bSigma^{-1/2}\|_{op}^2
    $. \Cref{lemma:Davidson} thus completes the proof for the tail bound
    on $(1-\df/n)^{-1}$.
\end{proof}

\begin{proof}[Proof of \Cref{thm:main-squared-loss} under \Cref{assum:g}(ii)]
    The algebra is the same, in particular
    \eqref{squareloss_RHS-sum}-\eqref{squareloss_RHS-sum_with_REM} are still
    valid. The bound $\E|\Rem|\le \C(\gamma)\sqrt n$ is valid by \Cref{prop:eta_1} (cf. \eqref{eq:bound-Rem-noise})
    while 
        \eqref{eq:RHS_2_18},
        \eqref{eq:note-in-passing},
        \eqref{divergcence-square-to-bound} and
        \eqref{eq:negligible-term-cross}
    are bounded from above by $\C(\mu_\rho,\gamma)\sqrt n$
    under \Cref{assum:g}(ii) as explained in the proof of 
    \Cref{thm:main_robust}.
    By \Cref{thm:main_robust} with $\mu_\rho=1$ for the square loss,
    $1-\df/n \ge 1-\gamma$ always holds.
\end{proof}

The proof under \Cref{assum:g}(iii.a) uses the argument
from \Cref{sec:kirszbraun}
and is provided in \Cref{sec:appendix_lasso}.}
The main message from \Cref{thm:main-squared-loss} 
is that $\hat R$ is consistent as an estimate of the out-of-sample
error $\|\bSigma^{\frac12}\bh\|^2$,
the estimate $\hat \tau^2$ is consistent for the generalization error
$\|\bSigma^{\frac12}\bh\|^2 + \sigma_*^2$,
and the estimate $\hat\sigma^2$
is consistent for the noise level $\sigma_*^2 =\|\bep\|^2/n$.

These estimates where known
for the unregularized Ordinary Least-Squares in \cite{leeb2008evaluation},
for the Lasso with square loss
in \cite{bayati2012lasso,bayati2013estimating,miolane2018distribution},
and for $\hbbeta=\mathbf{0}$ in \cite{dicker2014variance}.
Apart from these works and the specific Lasso penalty,
to our knowledge the above estimates for general convex penalty $g$
are new, so that \Cref{thm:main-squared-loss} considerably extends
the scope of applications of the estimates $\hat \tau^2, \hat R$ and $\hat\sigma^2$.
The estimate $\hat \tau^2$ of the generalization error is
of particular interest as it does not require the knowledge of $\bSigma$,
and can be used to choose the estimator with the smallest estimated
generalization error among a collection of convex regularized least-squares of
the form \eqref{hbbeta-squared-loss}. Since $\hat \tau^2$
estimates the risk for the actual sample size $n$, this provides a favorable
alternative to $K$-fold cross-validation which provides estimates
of the risk corresponding to the biased sample size $(1-1/K)n$
(cf. Figure 1 in \cite{rad2018scalable}).
In defense of cross-validation, which is known to successfully tune
parameters in practice for arbitrary data distribution,
the above estimates are valid 
when the rows of $\bX$ are iid $N(\mathbf{0},\bSigma)$
and it is unclear if the constancy of $\hat R,\hat\tau^2,\hat\sigma^2$
extends to non-Gaussian designs.
The following asymptotic
corollary of \Cref{thm:main-squared-loss} holds.

\begin{corollary}
    \label{corollary:square_loss}
    For some fixed value of $\gamma,\varphi>0,\mu\ge 0$,
    consider a sequence of regression problems and penalties
    with $n,p\to+\infty$ such  that for each $n,p$, the setting and assumptions of
    \Cref{thm:main-squared-loss} are fulfilled. Then
\begin{equation*}
\big|
 \big\{\|\bSigma^{\frac12}\bh\|^2 + \sigma_*^2\big\}-\hat \tau^2 \big|
 \quad
 +
 \quad
\big|\|\bSigma^{\frac12}\bh\|^2  -  \hat R \big|
 \quad
+
 \quad
 \big|\sigma_*^2-\hat\sigma^2\big| 
  \le  O_\P(n^{-\frac12}) \hat \tau^2.
\end{equation*}
Consequently, for the generalization error,
$\hat \tau^2/\{\|\bSigma^{\frac12}\bh\|^2 + \sigma_*^2\} \to^\P 1$
in probability.
\end{corollary}
\begin{proof}
This is a direct application of \Cref{thm:main-squared-loss}:
Since $(1-\df/n)^{-1} = O_\P(1)$
and $\Rem_*= O_\P(n^{-1/2})$ in \eqref{eq:main-squared-loss},
the first line of \eqref{eq:main-squared-loss} gives
$|\tau_*^2 - \hat \tau^2|
\le O_\P(n^{-1/2})(\hat\tau^2+\tau_*^2)$ where
$\tau_*^2 = \|\bSigma^{1/2}\bh\|^2 + \sigma_*^2$.
This implies $|\hat \tau^2/\tau_*^2 - 1| = O_\P(n^{-1/2})$
so that $\hat \tau^2/\tau_*^2\to^\P 1$.
Consequently the RHS of each line in \eqref{eq:main-squared-loss}
is $O_\P(n^{-1/2})\hat\tau^2$ which provides the claim.
\end{proof}

\pb{
One consequence of \Cref{thm:main-squared-loss}
and \Cref{corollary:square_loss} is the relationship
\begin{equation}
\Big|\bigl(1-\frac{\df}{n}\bigr) - \frac{\|\by-\bX\hbbeta\|/\sqrt n }{(\sigma_*^2 + \|\bSigma^{1/2}\bh\|^2)^{1/2}}\Big|
\le
\frac{1}{1-\df/n}
    \Big|\bigl(1-\frac{\df}{n}\bigr)^2 - \frac{\|\by-\bX\hbbeta\|^2/n}{\sigma_*^2 + \|\bSigma^{1/2}\bh\|^2}\Big| = O_\P(n^{-1/2})
    \label{eq:useful-rate}
\end{equation}
using $|a-b|\le \frac1a |a^2-b^2|$ for any $a,b>0$ for the inequality.
In particular, if there exist deterministic constants $\tau,\zeta>0$ such
the residual norm $\|\by-\bX\hbbeta\|/\sqrt n$ and error
$\|\bSigma^{1/2}\bh\|$ converge respectively to $\zeta\tau$ and $\tau$ at
a rate $O_\P(n^{-c})$ for $c>0$ in the sense
\begin{equation}
\|\by-\bX\hbbeta\|/\sqrt n= \zeta\tau +  O_\P(n^{-c}),
\qquad
\|\bSigma^{1/2}\bh\| = \tau + O_\P(n^{-c}),
\label{eq:montanari-celentano-wei}
\end{equation}
then $1-\df/n = \zeta +O_\P(\max\{n^{-c}, n^{-1/2})$ by \eqref{eq:useful-rate}.
Results of the form \eqref{eq:montanari-celentano-wei} and the
constants $\tau,\zeta$ are typically
characterized by the fixed-point equations discussed around \eqref{eq:regime-gamma-prime}, see the recent works \cite{celentano2020lasso,loureiro2021capturing} and references therein. For the Lasso, \cite[Theorems 5 and 7]{celentano2020lasso} proves \eqref{eq:montanari-celentano-wei} for $c=1/4$ up to logarithmic factors,
and if \Cref{assum:g}(iii) additionally holds then
\Cref{corollary:square_loss} and \eqref{eq:useful-rate} provides
$1-\df/n = \zeta + O_\P(n^{-1/4})$. This improves upon the rate
$1-\df/n-\zeta =  O_\P(n^{-1/6})$ obtained in Theorem 8 of the same work.
The argument used in \citep{miolane2018distribution,celentano2020lasso}
to connect $1-\df/n$ to the
fixed-point solutions $(\tau,\zeta)$ relies on relating $\df$  to
the law of the empirical distribution of the 
subgradient $\bX^\top(\by-\bX\hbbeta)$.
This relationship between $\df$ and the empirical distribution of 
the subgradient is specific to the $\ell_1$ penalty of the Lasso, and, as far as we
are aware, this technique does not extend to M-estimators \eqref{hbbeta-squared-loss}
other than $\ell_1$-penalized ones.
\Cref{corollary:square_loss} and \eqref{eq:useful-rate} show that
the connection between $1-\df/n$ and the ratio
$\|\by-\bX\hbbeta\|^2/(n(\sigma_*^2+\|\bSigma^{1/2}\bh\|^2))$
holds beyond $\ell_1$-penalized estimates.
}

We conclude
by supplementing the simulation setup in \Cref{sec:trace-norm-simu}
with the estimates $\hat\sigma^2$ and $\hat\tau^2$.
Boxplots of these estimates and their targets are provided
in \Cref{fig:nuclear-norm-2-sigma-hat}.
The quantity of approximation deteriorates for the smallest
tuning parameters, which can be explained by the multiplicative
factor $1-\df/n$ being close to 0.
An interesting phenomenon is visible regarding the empirical
variance of the estimate $\hat\sigma^2$: the smallest variances
are obtained for the tuning parameters with the smallest out-of-sample error.
Our theoretical results do not explain this observation;
further investigation of this phenomenon is left for future work.

\begin{figure}[tp]
    \centering
    \includegraphics[width=2.7in]{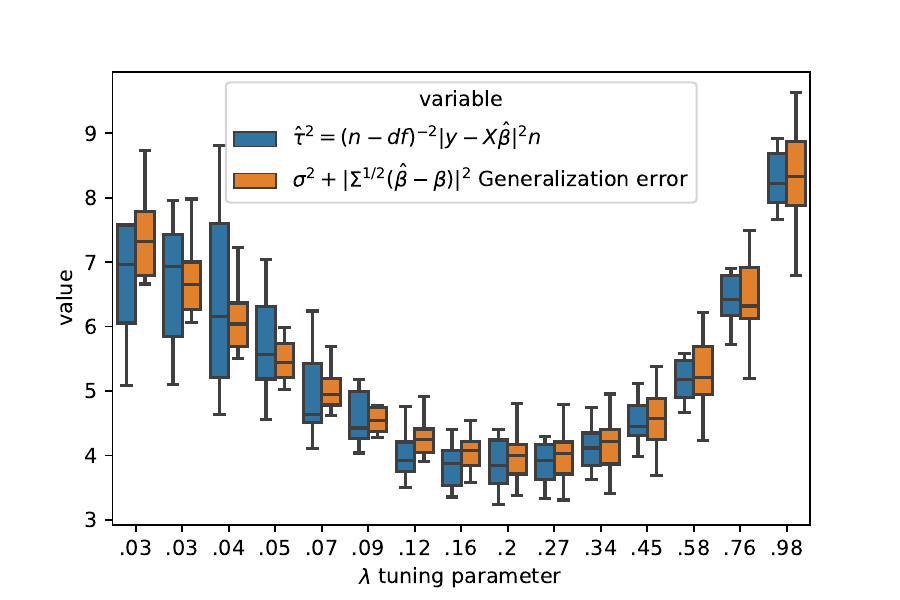}
    \includegraphics[width=2.7in]{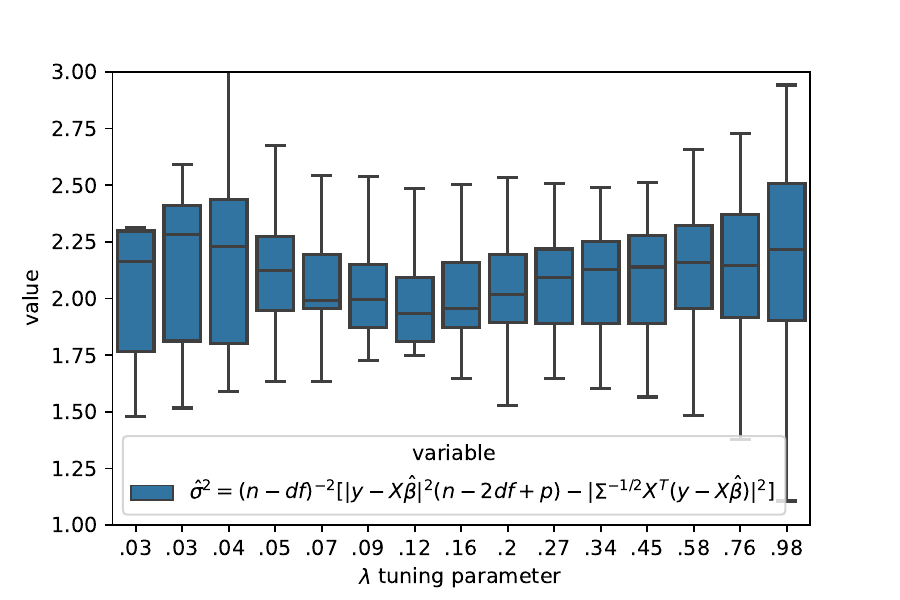}
\caption[Boxplots of $\hat \sigma^2,\hat\tau^2$ and their target $n=400,p=500$, nuclear normal penalty]{
        \label{fig:nuclear-norm-2-sigma-hat}
        Estimators $\hat\sigma^2,\hat\tau^2$
        and their targets for the nuclear norm penalty
        simulation described in \Cref{sec:trace-norm-simu}.
        The distribution of the entries of the noies $\bep$
        are iid $N(0,2)$.
    }
\end{figure}




\newcommand{\dX}{\dot{\bX}}
\newcommand{\dy}{\dot{\by}}
\section{Derivatives of $M$-estimators}
\label{sec:derivatives-6}

\subsection{Lipschitz properties}
\label{sec:lipschitz-properties}
\pb{Throughout the paper and the following propositions, 
the penalty $g:\R^p\to\R\cup\{+\infty\}$ and loss function
$\rho:\R\to\R$ are assumed convex.}

\begin{proposition}
    \label{prop:hbG-Lipschitz}
    Let $\rho$ be a loss function such that $\psi$ is $L$-Lipschitz,
    where $\psi=\rho'$. Then for any fixed design matrix $\bX\in\R^{n\times p}$,
    the mapping
    $\by\mapsto  \psi(\by-\bX\hbbeta)$
    is $L$-Lipschitz.
\end{proposition}
\begin{proof}
    let $\by,\tby\in\R^n$
    be two response vectors, $\hbbeta=\hbbeta(\by,\bX),\tbbeta=(\tby,\bX)$
    and $\bpsi=\psi(\by-\bX\hbbeta)$, $\tbpsi = \psi(\tby-\bX\tbbeta)$.
    The KKT conditions read
    $\bX^\top\bpsi \in n \partial g(\hbbeta)$ and
    $\bX^\top\tbpsi \in n\partial g(\tbbeta)$
    where $\partial g(\hbbeta)$ denotes the subdifferential
    of $g$ at $\hbbeta$.
    Multiplying by $\hbbeta-\tbbeta$ and taking the difference
    of the two KKT conditions above,
    we find
    \begin{align}
        \nonumber
       &n(\partial g(\hbbeta)- \partial g(\tbbeta))^\top(\hbbeta-\tbbeta) 
       +
       [\{\by - \bX\hbbeta\} -\{\tby - \bX\tbbeta\} ]^\top(\bpsi - \tbpsi)
       \\& \ni
       [\by - \tby]^\top(\bpsi -\tbpsi)
       =
       [\bX(\hbbeta-\tbbeta)]^\top(\bpsi-\tbpsi)
       + 
       [\{\by - \bX\hbbeta\} -\{\tby - \bX\tbbeta\} ]^\top(\bpsi - \tbpsi)
       .
       \label{eq:intermediate-KKT-prop5.1}
    \end{align}
    By the monotonicity of the subdifferential,
    $(\partial g(\hbbeta)- \partial g(\tbbeta))^\top(\hbbeta-\tbbeta)
    \subset [0,\infty)$.
    We now lower bound the second term in the first line for each term
    indexed by $i=1,...,n$.
    Since
    $\psi:\R\to\R$ is nondecreasing and $L$-Lipschitz,
    $\psi(u) - \psi(v) \le L (u-v)$ holds for any $u>v$,
    as well as 
    $(\psi(u) - \psi(v))^2 \le L (u-v) (\psi(u) - \psi(v))$
    since $\psi(u) - \psi(v) \ge 0$ by monotonicity.
    Applying this inequality for each $i$ to
    $u=y_i-\bx_i^\top\hbbeta$ and $v=\tilde y_i - \bx_i^\top\tbbeta$, we 
    obtain
    \begin{align*}
        L^{-1} \|\bpsi - \tbpsi\|^2
    &\le[\{\by - \bX\hbbeta\} -\{\tby - \bX\tbbeta\} ]^\top(\bpsi - \tbpsi)
       \le [\by - \tby]^\top(\bpsi -\tbpsi).
    \end{align*}
    The Cauchy-Schwarz inequality completes the proof.
\end{proof}

Proposition~\ref{prop:hbG-Lipschitz} generalizes the result of \cite{bellec2016bounds}
to general loss functions.
The following proposition uses a variant of \eqref{eq:intermediate-KKT-prop5.1}
to derive Lipschitz properties with respect to $(\by,\bX)$.

\begin{proposition}
    \label{prop:constant-bv-e_j}
    Assume that $\psi$ is $L$-Lipschitz. Let $(\by,\bX)\in\R^n \times \R^{n\times p}$ and $(\tby,\tbX)\in\R^n \times \R^{n\times p}$ be fixed. 
    Let $\hbbeta=\hbbeta(\by,\bX)$ be the estimator
    in \eqref{M-estimator-rho} with observed data $(\by,\bX)$
    and let
    $
    \tbbeta
    =\argmin_{\bb\in\R^p}
    \frac 1 n 
    \sum_{i=1}^n
    \rho(\tilde y_i - \be_i^\top\tbX\bb)
    + g(\bb)
    $,
i.e., the same $M$-estimator as \eqref{M-estimator-rho}
with the data $(\by,\bX)$ replaced by $(\tby,\tbX)$.
Set $\bpsi = \psi(\by-\bX\hbbeta)$ as well as
$\tbpsi = \psi(\tby - \tbX\tbbeta)$.
Then
\begin{equation}
    \label{lipschitz-fundamental-inequality}
    \begin{split}
    & n \mu \|\bSigma^{\frac12}(\hbbeta-\tbbeta)\|^2 +  \max\big(
    L^{-1} \|\tbpsi - \bpsi\|^2,
    ~
    \mu_\rho \|\by-\bX\hbbeta - \{\tby-\tbX\tbbeta\}\|^2
    \big)
    \\&\le
    (\tbbeta -\hbbeta)^\top(\tbX-\bX)^\top\bpsi
    +
    (\tby + (\bX-\tbX)\hbbeta -\by)^\top(\tbpsi - \bpsi).
    \end{split}
\end{equation}
Consequently  
\begin{enumerate}
    \item 
The map $(\by,\bX)\mapsto (\hbbeta(\by,\bX),\hbpsi(\by,\bX))$
is Lipschitz on every compact subset of $\R^n\times \R^{n\times p}$ if $\mu>0$
as in \Cref{assum:g}(i).
\item
The map $(\by,\bX)\mapsto (\hbbeta(\by,\bX),\hbpsi(\by,\bX))$
is Lipschitz on every compact subset of
$\{(\by,\bX)\in\R^n\times \R^{n\times p}: \phi_{\min}(\bSigma^{-\frac12}\bX^\top\bX\bSigma^{-\frac12})>0 \}$
if $\mu_\rho>0$ as in \Cref{assum:g}(ii).
    \item
        If $\tbX=\bX$ and $\tby=\by$, we must have $\bpsi=\tbpsi$.
        This means that if $\hbbeta$ and $\tbbeta$ are two distinct
        solutions of the optimization problem \eqref{M-estimator-rho},
        then $\psi(\by-\bX\hbbeta)=\psi(\by-\bX\tbbeta)$ must hold.
\end{enumerate}
\end{proposition}
\begin{proof}
    The KKT conditions for $\hbbeta$ and $\tbbeta$ read
$\tbX{}^\top\tbpsi \in n \partial g(\tbbeta)$ and
$\bX^\top\bpsi \in n\partial g(\hbbeta).$
If $D_g = (\hbbeta - \tbbeta)^\top(\partial g(\hbbeta) - \partial g(\tbbeta))$
then 
\begin{align}
    \nonumber
    &nD_g
      + (\tby - \tbX\tbbeta - \by + \bX\hbbeta)^\top
      (\tbpsi - \bpsi)
      \\&\ni (\tbbeta- \hbbeta)^\top(\tbX{}^\top \tbpsi - \bX^\top\bpsi)
      + (\tby - \tbX\tbbeta - \by + \bX\hbbeta)^\top
      (\tbpsi - \bpsi)
      \label{D_g-fundamental-ineq}
    \\&=
    (\tbbeta -\hbbeta)^\top(\tbX-\bX)^\top\bpsi
    +
    (\tby + (\bX-\tbX)\hbbeta -\by)^\top(\tbpsi - \bpsi).
    \nonumber
\end{align}
Since for any reals $u>s$ inequality
$(u-s)(\psi(u)-\psi(s)) \ge L^{-1} (\psi(u) - \psi(s))^2$ holds
when $\psi$ is $L$-Lipschitz and non-decreasing,
the first line is bounded from below by
$(n \inf D_g + L^{-1} \|\bpsi - \tbpsi\|^2)$.
If \Cref{assum:g}(ii) is satisfied for some $\mu_\rho>0$,
we also have $(u-s)(\psi(u)-\psi(s)) \ge \mu_\rho (u - s)^2$
so that the first line is bounded from below by
$\mu_\rho \|\by-\bX\hbbeta - \{\tby-\tbX\tbbeta\}\|^2$.
We also have $\inf D_g\ge \mu \|\bSigma^{\frac12}(\tbbeta-\hbbeta)\|^2$ by monotonicity
of the subdifferential and strong convexity of $g$
with respect to $\bSigma$.
This proves \eqref{lipschitz-fundamental-inequality}.

For (i), by bounding from above the right hand side of 
\eqref{lipschitz-fundamental-inequality} we find
\begin{align*}
    \min(\mu, L^{-1})(\|\bSigma^{\frac12}(\hbbeta-\tbbeta)\|^2+\|\bpsi-\tbpsi\|^2/n)^{\frac12}
\le
n^{-\frac12}\|\by-\tby\|
+
n^{-\frac12}\|(\bX-\tbX)\bSigma^{-\frac12}\|_{op}(n^{-\frac12}\|\bpsi\| + \|\bSigma^{1/2}\hbbeta\|).
\end{align*}
By taking a fixed $\bX$, e.g. $\bX=\mathbf{0}_{n\times p}$,
this implies that the supremum
$S(K) = \sup_{(\tby,\tbX)\in K}(\|\bSigma^{1/2}\smash{\tbbeta}\|+n^{-\frac12}\|\smash{\tbpsi}\|)$ is finite
for every compact $K$. If $(\by,\bX),(\smash{\tby},\smash{\tbX})\in K$
the right hand side is bounded from above
by 
$n^{-\frac12}\|\by-\smash{\tby}\|
+
n^{-\frac12}\|(\bX-\smash{\tbX})\bSigma^{-\frac12}\|_{op} S(K)$
which proves that the map is Lipschitz on $K$.

For (ii), we use that on the left hand side of \eqref{lipschitz-fundamental-inequality},
\begin{equation*}
\|\by-\bX\hbbeta - \{\tby-\tbX\tbbeta\}\|^2
=
\|\tbX(\hbbeta-\tbbeta)\|^2
+ 2[\tbX(-\hbbeta+\tbbeta)]^\top[\by-\tby+(\tbX-\bX)\hbbeta]
+
\|\by-\tby+(\tbX-\bX)\hbbeta\|^2.
\end{equation*}
Combined with \eqref{lipschitz-fundamental-inequality} this implies that
\begin{align*}
    &\min\big\{\mu_\rho\phi_{\min}(\tfrac 1 n \bSigma^{-1/2}\tbX{}^\top\tbX\bSigma^{-1/2}), L^{-1}\big\}
    (n\|\bSigma^{1/2}(\hbbeta-\tbbeta)\|^2 +  \|\bpsi-\tbpsi\|^2)
    \\&\le
    (\tbbeta -\hbbeta)^\top(\tbX-\bX)^\top\bpsi
    +
    (\tby + (\bX-\tbX)^\top\hbbeta -\by)^\top(\tbpsi - \bpsi).
     - 2[\tbX(-\hbbeta+\tbbeta)]^\top[\by-\tby+(\tbX-\bX)\hbbeta].
\end{align*}
The same argument as in (i) applies on every compact where the eigenvalues
of $n^{-1/2} \tbX\bSigma^{-1/2}$ are bounded away from $0$ and $+\infty$.
Finally,
(iii) directly follows from \eqref{lipschitz-fundamental-inequality}.
\end{proof}

\subsection{Lipschitz properties for a given, fixed \texorpdfstring{$\bep$}{epsilon}}

In this subsection, $\bep$ is fixed and we consider functions of $\bX\in\R^{n\times p}$
as defined in the following Lemma.

\begin{lemma}
    \label{lemma:general-C}
    Let $\bep\in\R^n$ be fixed and $\bX,\tbX$ be two design matrices.
    Define $\hbbeta=\hbbeta(\bX\bbeta + \bep,\bX)$ and
    $\tbbeta = \hbbeta(\tbX\bbeta + \bep,\tbX)$,
    $\bpsi = \psi(\bep + \bX\bbeta - \bX\hbbeta)$
    and
    $\tbpsi = \psi(\bep + \tbX\bbeta - \tbX\tbbeta)$
    as well as
    $\br=n^{-\frac12}\bpsi$ and $\tbr=n^{-\frac12}\tbpsi$,
    $\bh=\hbbeta-\bbeta$ and $\tbh=\tbbeta-\bbeta$.
    Let also  $D=(\|\br\|^2 + \|\bSigma^{\frac12}\bh\|^2)^{\frac12}$
    and $\tilde D = (\|\tbr\|^2 + \|\bSigma^{\frac12}\tbh\|^2)^{\frac12}$.
    If for some constant $L_*$
    and $\{\bX,\tbX\}\subset \R^{n\times p}$
    \begin{equation}
    \label{eq:liscphitz-mu-h-r-general-C}
        (\|\bSigma^{\frac12}(\bh-\tbh)\|^2 + \|\br-\tbr\|^2 )^{\frac12}
        \le n^{-\frac12} \|(\bX-\tbX)\bSigma^{-\frac12}\|_{op} L_* (\|\br\|^2 + \|\bSigma^{\frac12}\bh\|^2)^{\frac12}
    \end{equation}
    holds, then we also have the Lipschitz properties
    \begin{align}
        \label{eq:lipschitz-bh-D-inverse}
        \|\bSigma^{\frac12}(\bh D^{-1} - \tbh \tilde D^{-1})\|   & \le
        n^{-\frac12}\|( \bX-\tbX)\bSigma^{-\frac12}\|_{op} 2 L_*,
        \\
        \|\br D^{-1} - \tbr \tilde D^{-1}\|   & \le
        n^{-\frac12}\|(\bX-\tbX)\bSigma^{-\frac12}\|_{op} 2 L_*,
        \label{eq:lipschitz-br-D-inverse}
        \\
        n^{-\frac12}\|\bSigma^{-\frac12}(\tfrac{\bX^\top \br}{D} -  \tfrac{\tbX{}^\top \tbr}{\tilde D} )\|   & \le
        n^{-\frac12}\|(\bX-\tbX)\bSigma^{-\frac12}\|_{op}(1 +  2 L_* \|n^{-\frac12}\tbX\bSigma^{-\frac12}\|_{op}),
        \label{eq:lipschitz-XT-br-D-inverse}
        \\
        |D^{-1} - \tilde D^{-1}|
        &
        \le n^{-\frac12}\|(\bX-\tbX)\bSigma^{-\frac12}\|_{op}
        L_* \tilde D^{-1}.
        \label{eq:lipschitz-D-inverse}
    \end{align}
\end{lemma}

\begin{proof}[Proof of \Cref{lemma:general-C}]
    Assume that $\bSigma=\bI_p$ without loss of generality,
    by performing the variable change \eqref{eq:variable-change}
    if necessary.
    By the triangle inequality,
$\|\bh D^{-1} - \tbh \tilde D^{-1}\|
\le \|\tbh\| |D^{-1} - \tilde D^{-1}| + D^{-1} \|\bh-\tbh\|$.
Then $D^{-1} \|\bh-\tbh\| \le n^{-\frac12}\|\bX-\tbX\|_{op}L_*$
for the second term by \eqref{eq:liscphitz-mu-h-r-general-C}.
For the first term, $\|\tbh\| |D^{-1} - \tilde D^{-1}|\le D^{-1}|D - \tilde D|
\le D^{-1}(\|\br-\tbr\|^2 + \|\bh-\tbh\|^2)^{\frac12}$
by the triangle inequality, and another application of \eqref{eq:liscphitz-mu-h-r-general-C}
provides \eqref{eq:lipschitz-bh-D-inverse}.
The exact same argument provides \eqref{eq:lipschitz-br-D-inverse}
since the roles of $\bh$ and $\br$ are symmetric in \eqref{eq:liscphitz-mu-h-r-general-C}.
For \eqref{eq:lipschitz-XT-br-D-inverse}, we use
$$\|\bX^\top \br D^{-1} -  \tbX{}^\top \tbr \tilde D^{-1}\|
\le \|\bX-\tbX\|_{op} (\|\br\|D^{-1})
+ \|\tbX\|_{op} \|\br D^{-1} - \tbr \tilde D^{-1}\|.$$
Combined with \eqref{eq:lipschitz-br-D-inverse} and $\|\br\|D^{-1} \le 1$, this provides
\eqref{eq:lipschitz-XT-br-D-inverse}.
For the fourth inequality,
by the triangle inequality
$
    |D^{-1} - \tilde D^{-1}|
    \le
    (D\tilde D)^{-1}
    |D-\tilde D|
    \le n^{-\frac12}\|\bX-\tbX\|
    L_* \tilde D
    $ thanks to \eqref{eq:liscphitz-mu-h-r-general-C}.
\end{proof}

\begin{proposition}
    \label{prop:lipschtiz}
    Let \Cref{assum:rho} and \Cref{assum:g}(i) be fulfilled.
    Consider the notation of \Cref{lemma:general-C} for
    $\bX,\bpsi,\br,\bh$ and $\tbX,\tbpsi,\tbr,\tbh$.
    Then by  \eqref{lipschitz-fundamental-inequality} we have
    \begin{align}
        \mu\|\bSigma^{\frac12}(\bh-\tbh)\|^2 + \|\br-\tbr\|^2
        &=\mu\|\bSigma^{\frac12}(\hbbeta-\tbbeta)\|^2 + \|\bpsi-\tbpsi\|^2/n
        \nonumber
        \\&\le
        \big[
        (\tbh - \bh)^\top(\tbX-\bX)^\top\bpsi
        + 
        \bh^\top(\bX-\tbX)^\top(\tbpsi - \bpsi)
        \big]/n
        \label{first-inequality-to-take-derivatives}
        \\&\le
        n^{-\frac12}\|(\bX-\tbX)\bSigma^{-\frac12}\|_{op}
        (\|\bSigma^{\frac12}(\tbh-\bh)\|^2 +  \|\tbr-\br\|^2)^{\frac12}
        (\|\br\|^2 + \|\bSigma^{\frac12}\bh\|^2)^{\frac12}.
        \nonumber
    \end{align}
    Hence \eqref{eq:liscphitz-mu-h-r-general-C} holds for
    $L_* = \max(\mu^{-1},1)$ and
    all $\bX,\tbX\in \R^{n\times p}$.
\end{proposition}
\begin{proof}
    This follows  by \eqref{lipschitz-fundamental-inequality}
    with $\tby = \bep + \tbX\bbeta$ and $\by = \bep + \bX\bbeta$.
    The last inequality in \eqref{first-inequality-to-take-derivatives}
    is due to the Cauchy-Schwarz inequality.
    Inequality \eqref{eq:liscphitz-mu-h-r-general-C} with the given $L_*$
    is obtained by dividing by $(\|\br-\tbr\|^2 + \|\bSigma^{\frac12}(\bh-\tbh)\|^2)^{\frac12}$.
\end{proof}

\begin{proposition}
    \label{prop:lipschtiz2-gamma-less-one}
    Let \Cref{assum:rho} and \Cref{assum:g}(ii) be fulfilled.
    Consider the notation of \Cref{lemma:general-C} for
    $\bX,\bpsi,\br,\bh$ and $\tbX,\tbpsi,\tbr,\tbh$.
    Then 
    \begin{align*}
      & \min\big(1, \mu_\rho\phi_{\min}(\bSigma^{-\frac12}\tbX{}^\top\tbX\bSigma^{-\frac12}/n)\big)
      \max\big(\|\br-\tbr\|, \|\bSigma^{\frac12}(\bh-\tbh)\|\big)
        \\&\le
        n^{-\frac12}\|(\bX-\tbX)\bSigma^{-\frac12}\|_{op}
        \big[\|\br\| + \|\bSigma^{\frac12}\bh\|
            \big(
            1 +
            2\mu_\rho\|n^{-\frac12}\tbX\bSigma^{-\frac12}\|_{op}
            \big)
        \big].
        \nonumber
    \end{align*}
    Hence \eqref{eq:liscphitz-mu-h-r-general-C} holds
    with $L_*= \max(1, \mu_\rho^{-1}
    \phi_{\min}(\frac1n\bSigma^{-\frac12}\tbX{}^\top\tbX\bSigma^{-\frac12})^{-1}
    )
    6\max(1, 2\mu_\rho \|n^{-1/2}\tbX\bSigma^{-1/2}\|_{op})
    $
    for all $\bX,\tbX\in\R^{n\times p}$.
\end{proposition}
\begin{proof}
    By  \eqref{lipschitz-fundamental-inequality} we have
    \begin{align*}
        \nonumber
      \max(\|\br-\tbr\|^2, \mu_\rho\|\bX\bh -\tbX\tbh\|^2/n)
        \le
        \big[
        (\tbh - \bh)^\top(\tbX-\bX)^\top\bpsi
        -
        \bh^\top(\bX-\tbX)^\top(\tbpsi - \bpsi)
        \big]/n.
    \end{align*}
    We also have 
    $\|\bX\bh -\tbX\tbh\|^2
    = \|\tbX(\bh-\tbh)\|^2
    + 2[\tbX(\bh-\tbh)]^\top(\bX-\tbX)\bh
    + \|(\bX-\tbX)\bh\|^2$
    so that the previous display implies
    \begin{equation}
    \begin{aligned}
      & \max(\|\br-\tbr\|^2, \mu_\rho\|\tbX(\bh -\tbh)\|^2/n)
        \\&\le
        \big[
        (\tbh - \bh)^\top(\tbX-\bX)^\top\bpsi
        -
        \bh^\top(\bX-\tbX)^\top(\tbpsi - \bpsi)
        - 2\mu_\rho[\tbX(\bh-\tbh)]^\top(\bX-\tbX)\bh
        \big]/n
    \end{aligned}
    \label{eq:useful-lasso}
    \end{equation}
    and the conclusion holds by the Cauchy-Schwarz
    inequality and properties of the operator norm.
\end{proof}

\section{Gradient identities}
\label{sec:gradient-identities}

\begin{corollary}
    \label{corollary:constant-bv-e_j}
    Let the setting
    and assumptions of Proposition~\ref{prop:constant-bv-e_j} be fulfilled.
If $\tby = \by + \hbeta_j\bv$ and
$\tbX = \bX + \bv\be_j^\top$ for any direction $\bv\in\R^n$
with $\bv^\top\bpsi = 0$ and index $j\in[p]$, then $\tbpsi=\bpsi$
and the solution $\hbbeta$ of the optimization problem \eqref{M-estimator-rho}
is also solution of the same optimization problem with
$(\by,\bX)$ replaced by $(\tby,\tbX)$.
If additionally $\mu>0$,
then $\tbbeta=\hbbeta$ must hold.
\end{corollary}
\begin{proof}
The right hand side of \eqref{lipschitz-fundamental-inequality}
is 0 for the given $\tby-\by$ and $\tbX-\bX$.
This proves that $\tbpsi=\bpsi$. Furthermore, the KKT conditions
for $\hbbeta$ read
$\bX^\top\bpsi \in n\partial g(\hbbeta)$,
and we have $\bX^\top\bpsi  = \tbX{}^\top\bpsi$,
and $\bpsi = \psi(\by-\bX\hbbeta) = \psi(\tby - \tbX\hbbeta)$.
This implies that
$\tbX{}^\top\psi(\tby-\tbX\hbbeta) \in n \partial g(\hbbeta)$
so that $\hbbeta$ is solution to the optimization problem
with data $(\tby,\tbX)$, even if $\mu=0$.
The claim $\hbbeta=\tbbeta$
for $\mu>0$ follows
by unicity of the minimizer of strongly convex functions.
\end{proof}


\begin{proof}[Proof of \Cref{lemma:ninja-variant}]
    Existence of the partial derivatives of $\hbpsi$ and $\hbbeta$
    at $(\by^0,\bX^0)$ is granted by the assumption
    of Frechet differentiability. If $\psi$ is continuously differentiable,
    the chain rule
    \eqref{eq:chain-rule-hard}
    holds.
    We now compute the directional derivatives.
    To this end, let $(\dy,\dX)\in\R^{n\times(1+p)}$ representing a perturbation direction (e.g., take $(\dy,\dX)=(\be_l,\mathbf 0_{n\times p})$ for the partial derivative
    in \eqref{eq:derivatives-y} or 
    $(\dy,\dX)=(\mathbf 0_n,\be_i\be_j^\top)$
    for the partial
    derivative in \eqref{eq:derivatives-X}).
    For $t\in\R$, let
    $(\by(t),\bX(t)) = (\by^0 + t\dy, \bX^0 + t \dX )$ and
    $\bb(t) = \hbbeta(\by(t),\bX(t))$. Set also
    $\bpsi(t)=\psi(\by(t) - \bX(t)\bb(t))$
    where as usual $\psi=\rho'$ acts componentwise.
    The KKT conditions at $t$ and $0$ read
    $\bX(t)^\top \bpsi(t)
    \in n \partial g(\bb(t))$
    and 
    $\bX(0)^\top\bpsi(0) \in n \partial g(\bb(0))$.
    Multiplying the difference of these KKT conditions
    by $(\bb(t)-\bb(0))$ and using the strong convexity of $g$ with respect to $\bSigma$, we find
    $$
    n \mu \|\bSigma^{\frac12}(\bb(t)-\bb(0)\|^2 
    \le
    (\bb(t) - \bb(0))^\top
    \big[\bX(t)^\top\bpsi(t) - \bX(0)^\top\bpsi(0)
    \big].
    $$
    Denote by $\bX'(0)$, $\bb'(0)$ and $\bpsi'(0)$ the derivatives at $t=0$.
    By the product rule
    $\frac{d}{dt} \bX(t)^\top\bpsi(t)\big|_{t=0} = 
    \bX'(0)^\top\hbpsi{}^0 + \bX^0{}^\top\bpsi'(0)$ and by the chain rule
    $\bpsi'(0) = \bD^0(\by'(0) - \bX'(0)\bb(0) - \bX(0)\bb'(0))$
    where $\bD^0 = \diag(\bpsi^0{}')$.
    Dividing by $t^2$ and taking the limit $t\to 0$ in the previous display and moving $\bb'(0)^\top\bX(0)^\top \bD^0 \bX(0)\bb'(0)$ to the left-hand side gives
    $$
    n\mu \|\bSigma^{1/2}\bb'(0)\|^2
    + \|(\bD^0)^{1/2} \bX(0)\bb'(0)\|^2
    \le
    \bb'(0)^\top\bigl[
        \bX'(0)^\top\hbpsi^0 + \bX^0{}^\top\bD^0\bigl(\by'(0) - \bX'(0)\bb(0) - \bX(0)\bb'(0)\bigr)
    \bigr].
    $$
    By definition of Frechet differentiability,
    the mapping $\mathcal B: (\dy,\bX)\mapsto \bSigma^{1/2}\bb'(0)$
    appearing in the left-hand side
    is a linear map $\mathcal B: \R^n\times \R^{n\times p}\to \R^p$.
    The mapping
    $\mathcal L:(\dy,\dX)\mapsto
    \bSigma^{-1/2}[
    \dX{}^\top\hbpsi^0 + \bX^0{}^\top\bD^0\bigl(\dy - \dX{}^\top\hbbeta{}^0\bigr)
    ]$ appearing on the right-hand side is also a linear map.
    Since $\bSigma$ and the matrix inside $\phi_{\min}(\cdot)$
    in \eqref{eq:hbA-op} are positive definite,
    $\mathcal L(\dy,\dX) = \mathbf 0_p$ implies $\bb'(0)=\mathbf 0_{p}$.
    Two linear mappings $\mathcal L$ and $\mathcal B$ from $\R^{n\times p}$
    to $\R^p$ have $\ker \mathcal L\subset\ker \mathcal B$ if and only
    if there exists a matrix $\bM$ with $\ker(\bM)^\perp\subset\text{Image}(\mathcal L)$ such that
    $\mathcal B = \bM \mathcal L$. This proves the existence of
    $\hbA(\by^0,\bX^0)$ in \eqref{eq:derivatives-X}-\eqref{eq:derivatives-y}
    by taking $\hbA(\by^0,\bX^0) = \bSigma^{-1/2}\bM\bSigma^{-1/2}$.
    If $\bu\in\R^p$ has unit norm and is a right singular vector associated with
    the largest singular value of $\bM$ then
    $\|\bM\|_{op} = \|\bM \bu\|$ and $\bu\in\ker(\bM)^\perp$. Since
    $\ker(\bM)^\perp\subset\text{Image}(\mathcal L)$ we can find
    $(\by,\dX)$ such that $\bu=\mathcal L(\dy,\dX)$. Then previous display
    then yields
    $f_{\min}\|\bM\bu\|^2 \le \bu^\top \bM \bu$ where
    $f_{\min}^{-1}$ is the right-hand side of \eqref{eq:hbA-op}.
    This provides $\|\bM\|_{op} = \|\bM\bu\| \le f_{\min}^{-1}$
    and concludes the proof of \eqref{eq:hbA-op}
    since $\bM=\bSigma^{1/2}\hbA{}(\by^0,\bX^0)\bSigma^{1/2}$.

The second claim, where $\psi$ is only assumed to be 1-Lipschitz, requires
the chain rule \eqref{eq:chain-rule-hard} to hold
for almost $(\by^0,\bX^0)$ in $U$.
    The validity of the chain rule \eqref{eq:chain-rule-hard} boils
    down to the chain rule
    for 
    $$\psi\circ u_i\qquad\text{ where }
    \qquad u_i:U \to \R,
    \quad
    (\by,\bX)\mapsto u_i(\by,\bX)= y_i - \bx_i^\top\hbbeta(\by,\bX)
    $$
    for all $i\in[n]$. 
    Since $\psi:\R\to\R$ is Lipschitz and $u_i$ is Lipschitz in $U$,
    \cite[Theorem 2.1.11]{ziemer2012weakly} implies that
    $(\partial/\partial y_l)\hpsi_i(\by,\bX) = \psi'(u_i(\by,\bX))
    (\partial/\partial y_l)u_i(\by,\bX)$
    almost everywhere in $U$.
    (This version of the chain rule is straightforward 
    at points $(\by,\bX)$ where
    $\psi'(u_i(\by,\bX))$ and $(\partial/\partial y_l)u_i(\by,\bX)$
    both exist,
    as well as at points $(\by,\bX)$ where 
    $(\partial/\partial y_l)u_i(\by,\bX) = 0$ thanks
    to
    $t^{-1}|\hbpsi(\by+t\be_l,\bX) - \hbpsi(\by,\bX)|
    \le M t^{-1} |u_i(\by+t\be_l,\bX)  - u_i(\by,\bX)|$
    in which case
    $(\partial/\partial y_l)\hpsi_i(\by,\bX)=0$
    and $\psi'(u_i(\by,\bX))$ need not exist.
    The non-trivial part of the argument
    in
    \cite[Theorem 2.1.11]{ziemer2012weakly}
    is to prove
    that the set $\{(\by,\bX)\in U:
        \psi'(u_i(\by,\bX)) \text{ fails to exist and }
        (\partial/\partial y_l)u_i(\by,\bX) \ne 0
    \}$ has Lebesgue measure 0.)
\end{proof}

\section{Inequalities for functions of standard multivariate normals}
\label{sec:6-identitities-gaussian}

This section provides several useful tail bound and moment
inequalities for functions of a matrix with iid $N(0,1)$ entries,
including the proof of \Cref{prop:SOS-X-D-bh-br-main-text}.

\begin{lemma}
    \label{lemma:Davidson}
    Let $\gamma>0$.
    If $p/n\le \gamma$ and
    $\bG\in\R^{n\times p}$ has iid $N(0,1)$ entries then
    the tail bound
    $\P(\|\bG\|_{op}>\sqrt n + \sqrt p + t)\le \Phi(-t)\le e^{-t^2/2}$
    holds
    where $\Phi$ is the standard normal CDF. As a consequence
    $\E[\|n^{-1/2}\bG\|_{op}^k]\le \C(k,\gamma)$ for any integer $k\ge 1$.
\end{lemma}
The above tail bound is given in 
\cite[Theorem II.13]{DavidsonS01} and the moment bound
$\E[\|n^{-1/2}\bG\|_{op}^k]\le \C(k,\gamma)$ is obtained by integrating
the tail bound.
The next result is well known and follows from
\cite{edelman1988eigenvalues} as explained 
in \cite[Proposition A.1]{bellec_zhang2019second_poincare} among others.
\begin{lemma}[Negative moments]
    \label{lemma:Edelman}
    Let $\gamma\in(0,1)$.
    If $p/n\le \gamma$ and
    $\bG\in\R^{n\times p}$ has iid $N(0,1)$ entries then
    $\E[\phi_{\min}(\frac 1 n \bG^\top\bG)^{-k}]\le \C(k,\gamma)$ for any integer $k\ge 1$.
\end{lemma}

\begin{proposition}[
    Eq. (8.6) in \cite{stein1981estimation}
    or \cite{bellec_zhang2018second_stein}
    ]
    \label{prop:BZ18}
If $\by = \bmu + \bep$ with $\bep\sim N({\bf 0},\sigma^2\bI_n)$ 
and $\mathbf{f}:\R^n\to\R^n$ has weakly differentiable components
then
\begin{align}
    \nonumber
\E[(\bep^\top \mathbf{f}(\by) - \sigma^2\dv \mathbf{f}(\by))^2] 
    &=
    \sigma^2 \E[\|\mathbf{f}(\by)\|^2] + \sigma^4\E\bigl[\trace[\left(\nabla \mathbf{f}(\by)\right)^2]\bigr] 
    \\&\le 
    \sigma^2 \E[\|\mathbf{f}(\by)\|^2] + \sigma^4\E[\|\nabla \mathbf{f}(\by)\|_F^2]
    \label{SOS-upper-bound-nabla-Frob}
    \\&\le
    \sigma^2\|\E[\mathbf{f}(\by)]\|^2 + 2 \sigma^4\E[\|\nabla \mathbf{f}(\by)\|_F^2],
    \label{SOS-followed-by-Poincare}
\end{align}
provided that the last line is finite. If $\E[\mathbf{f}(\by)]={\bf 0}$ then
\begin{equation}
    \E[(\bep^\top \mathbf{f}(\by) - \sigma^2\dv \mathbf{f}(\by))^2] 
    \le 
    2 \sigma^4\E[\|\nabla \mathbf{f}(\by)\|_F^2].
    \label{SOS-upper-bound-expectation-zero}
\end{equation}
\end{proposition}
The first equality in \eqref{SOS-upper-bound-nabla-Frob}
is the identity studied in \cite{bellec_zhang2018second_stein} and
\eqref{SOS-upper-bound-nabla-Frob}
follows by the Cauchy-Schwarz inequality.
The second inequality is a consequence of the Gaussian Poincar\'e inequality
\cite[Theorem 3.20]{boucheron2013concentration} applied to 
each component of $\mathbf{f}$.

\onArxivOrInJournal{}{
The variant 
\eqref{SOS-concentration-bg} below may also be useful.

\begin{proposition}
    \label{prop:SOS-high-proba-g}
    Let $\mathbf{f}, \bep, \by$ be as in
    Proposition~\ref{prop:BZ18}.
    Then there exist random variables $Z,T,\tilde T$ with
    $Z\sim N(0,1)$
    and $\E[\tilde T^2]\vee \E[T^2]\le 1$ such that
    \begin{align}
        |\bep^\top \mathbf{f}(\by)
        - \sigma^2\dv \mathbf{f}(\by)
        |
        &\le
        \sigma |Z| ~ \|\E[\mathbf{f}(\by)]\|
        + \sigma^2 ~2|T|~\E[\nabla \mathbf{f}(\by)\|_F^2]^{\frac12}
        \\
        &\le
        \sigma |Z|~\|\mathbf{f}(\by)\|
        +
        \smash{
            \sigma^2(2|T|+|Z \tilde T|)
            \E[\nabla \mathbf{f}(\by)\|_F^2]^{\frac12}.
        }
        \label{SOS-concentration-bg}
    \end{align}
\end{proposition}

\begin{proof}
    Define
    $\mathbf{F}(\by) = \mathbf{f}(\by) - \E[\mathbf{f}(\by)]$, $Z = \sigma^{-1}\bep^\top\E[\mathbf{f}(\by)]/\|\E[\mathbf{f}(\by)]\|\sim N(0,1)$  and
    $$T^2=\big(\bep^\top \mathbf{f}(\by) - \dv \mathbf{f}(\by) - Z\sigma \|\E \mathbf{f}(\by)\|\big)^2
          \big/(2\sigma^4\E[\|\nabla f(\by)\|_F^2]).$$
    Since 
    $\E[\mathbf{F}(\by)]=0$, by \eqref{SOS-followed-by-Poincare} applied to $\mathbf{F}$
    we have
    \bes
    2\sigma^4\E[\|\nabla \mathbf{f}(\by)\|_F^2] \E[T^2]
      &=&\E[(\bep^\top \mathbf{F}(\by) - \dv \mathbf{F}(\by))^2]
    \\&\le& 2 \sigma^4 \E[\| \nabla \mathbf{F}(\by)\|_F^2]
    = 2 \sigma^4 \E[\| \nabla \mathbf{f}(\by)\|_F^2]
    \ees
    so that $\E[T^2]\le 1$.
    Next, let $\tilde T =[\sigma^2\E[\|\nabla \mathbf{f}(\by)\|_F^2]^{-\frac12}
    \big|\|\E[\mathbf{f}(\by)]\| - \|\mathbf{f}(\by)\|\big|$, which satisfies
    $\E[\tilde T^2]\le 1$ by the Gaussian Poincar\'e inequality.
    By construction of $T$ and $\tilde T$, we obtain \eqref{SOS-concentration-bg}.
\end{proof}
}

\propSOSmatrix

\begin{proof}[Proof of Proposition~\ref{prop:SOS-X-D-bh-br-main-text}]
Proposition~\ref{prop:SOS-X-D-bh-br-main-text} is 
obtained for $\bX$ with iid $N(0,1)$
entries by applying
\eqref{SOS-upper-bound-nabla-Frob}
to $\by=\bep=\text{vect}(\bX)$
and $\mathbf f(\bX)=\text{vect}(\brho(\bX)\bfeta(\bX)^\top)$ where
$\text{vect}(\cdot)$ is the 
vectorization operator.
\end{proof}

\section{$\chi^2$ type bounds under dependence}
\label{sec:proof_chi2}
To prove \Cref{thm:chi-square-type-main-text-rho}
and \Cref{big_cor}, we first derive a lemma
to control the correlation between
the two mean-zero random variables 
\begin{equation}
    (\bz_j^{\pb{\top}} f(\bz_k))^2-\|f (\bz_k)\|^2
\quad \text{ and } \quad
(\bz_k^{\pb{\top}} h(\bz_j))^2-\|h (\bz_j)\|^2
\label{eq:two-randomvariables-small-correlation}
\end{equation}
where $\bz_j,\bz_k$ are independent standard normal random vectors
and $f,h$ are functions $\R^n\to\R^n$.
If $f,h$ are constant, then
the correlation between these two random variables is 0 by independence.
If $f,h$ are non-constant, the following gives an exact
formula and an upper bound for the correlation
of the two random variables in 
\eqref{eq:two-randomvariables-small-correlation}.

\begin{lemma}
    \label{lemma:f-g-z_j-z_k}
    Let $\bz_j,\bz_k$ be independent $N({\bf 0},\bI_n)$ random vectors.
    Let $f,h:\R^n\to\R^n$ deterministic with weakly differentiable
    components and define the random matrices $\bA,\bB\in\R^{n\times n}$
    respectively by
    $\bA = \bigl(
            \bz_j^\top f(\bz_k) \bI_n
            + f(\bz_k)\bz_j^\top
        \bigr)
        \nabla f(\bz_k)^\top
    $ and
    $\bB =
         \bigl(
         \bz_k^\top h(\bz_j) \bI_n
        + h(\bz_j) \bz_k^\top
        \bigr)
    \nabla h(\bz_j)^\top
    $.
    Assume that
    \begin{equation}
        \label{assumption-A-B-sobolev} 
    \E[\|f(\bz_k)\|^4]
    +
    \E[\|h(\bz_j)\|^4]
    +
    \E[\|\bA\|_F^2]
    +
    \E[\|\bB\|_F^2]
    < +\infty.
    \end{equation}
    Then equality
\begin{equation}
    \label{eq:equality-correlation-lemma}
\E\big[
    \big\{(\bz_j^{\pb{\top}} f(\bz_k))^2-\|f (\bz_k)\|^2 \big\}
    \big\{(\bz_k^{\pb{\top}} h(\bz_j))^2-\|h (\bz_j)\|^2 \big\}
    \big]
    =
    \E[
    \trace
    \{
        \bA \bB
    \}
    ]
    \end{equation}
holds and
$\E|\trace[\bA\bB]| \le
        4
        \E\bigl[\|\nabla f(\bz_k)\|_F^2\|f(\bz_k)\|^2\bigr]^{\frac12}
        \E\bigl[\|\nabla h(\bz_j)\|_F^2\|h(\bz_j)\|^2\bigr]^{\frac12}
$.
\end{lemma}
\begin{proof}
    Define $\bz\in\R^{2n}$ with $\bz=(\bz_j^\top, \bz_k^\top)^\top$
    as well as $F,H:\R^{2n}\to \R^{2n}$ by
    $$
    F\Bigl(\begin{bmatrix}\bz_j \\\bz_k\end{bmatrix}\Bigr)
    =\Bigl( \begin{bmatrix}f(\bz_k)  f(\bz_k)^\top\bz_j \\ {\bf 0}_{\R^n} 
    \end{bmatrix}
        \Bigr),
    \qquad
    H\Bigl(\begin{bmatrix}\bz_j \\\bz_k \end{bmatrix}\Bigr)
    =\Bigl( \begin{bmatrix} {\bf 0}_{\R^n} \\  h(\bz_j) h(\bz_j)^\top\bz_k
    \end{bmatrix}\Bigr).$$
    The Jacobians of $F, H$ are the $2n\times 2n$ matrices 
    \[
        \nabla F(\bz)^\top =
    \begin{bmatrix}
        f(\bz_k) f(\bz_k)^\top & \bA \\
     {\bf 0}_{n\times n} & {\bf 0}_{n\times n}
    \end{bmatrix}
    ,
    \qquad
    \nabla H(\bz)^\top =
    \begin{bmatrix}
    {\bf 0}_{n\times n} & {\bf 0}_{n\times n} \\
    \bB & h(\bz_j) h(\bz_j)^\top
    \end{bmatrix}
    .
    \]
    Since $\dv F(\bz)= \trace[\nabla F(\bz)]
    = \|f(\bz_k)\|^2$ and similarly for $\dv H$,
    the left hand side in \eqref{eq:equality-correlation-lemma} equals
    $$
    \E[(\bz^\top F(\bz) - \dv F(\bz))(\bz^\top H(\bz) - \dv H(\bz))]
    $$
    with $F,H$ being weakly differentiable with
    $\E[\|F(\bz)\|^2 + \|H(\bz)\|^2 + \|\nabla H(\bz)\|_F^2 + \|\nabla F(\bz)\|_F^2]
    < +\infty$
    thanks to \eqref{assumption-A-B-sobolev}.
    The last display is equal to
    $\E[F(\bz)^\top H(\bz) + \trace \{\nabla F(\bz) H(\bz)\}]$
    by Section 2.2 in \cite{bellec_zhang2018second_stein}.
    Here $F(\bz)^\top H(\bz)=0$ always holds by construction of $F,H$
    and the matrix product by block gives
    $\trace \{ \nabla H(\bz) \nabla F(\bz)\} = \trace \{\bA \bB\}$.

    Next, by the Cauchy-Schwarz inequality we have
    $\E|\trace[\bA\bB]|
    \le \E[\|\bA\|_F\|\bB\|_F]
    \le \E[\|\bA\|_F^2]^{1/2}
        \E[\|\bB\|_F^2]^{1/2}
    $.
    By definition of $\bA$ and properties
    of the operator norm,
    \begin{align*}
    &\|\bA\|_F \le \|\nabla f(\bz_k)\|_F 
    |\bz_j^\top f(\bz_k)| 
    + \| f(\bz_k)\| \|\nabla f(\bz_k) \bz_j\|.
    \end{align*}
    By the triangle inequality and independence we find
    \begin{align*}
    \E[\|\bA\|_F^2]^{1/2}
    &\le
    \E[\|\nabla f(\bz_k)\|_F^2(\bz_j^\top f(\bz_k))^2]^{1/2}
    +
    \E[\|\nabla f(\bz_k)\bz_j\|^2\| f(\bz_k)\|^2]^{1/2}
  \\&=
    \E[\|\nabla f(\bz_k)\|_F^2\|f(\bz_k)\|^2]^{1/2}
    +
    \E[\|\nabla f(\bz_k)\|_F^2\| f(\bz_k)\|^2]^{1/2}
    \end{align*}
    thanks to $\E[\bz_j\bz_j^\top|\bz_k]=\bI_n$
    and $\E[\|\nabla f(\bz_k) \bz_j\|^2|\bz_k] =
    \trace(\nabla f(\bz_k)^\top \nabla f(\bz_k) \E[\bz_j\bz_j^\top|\bz_k])$.
    Similarly,
    $\E[\|\bB\|_F^2]^{1/2}
    \le
    2\E[\|\nabla h(\bz_j)\|_F^2 \|h(\bz_j)\|^2]$
    which completes the proof.
\end{proof}

\thmChiSquare
\begin{proof}[Proof of \Cref{thm:chi-square-type-main-text-rho}]
    Let $\bz_j = \bX\be_j$. Let also
    $\xi_j = \bz_j^\top(\brho-\E_j[\brho]) - d_j $
    where $\E_j[\cdot]$ is the conditional expectation
    $\E_j[\cdot] = \E[\cdot|\bX_{-j}]$
    and $d_j = \sum_{i=1}^n \frac{\partial \rho_i}{\partial x_{ij}}$
    so that $\E_j[\xi_j] = 0$ by Stein's formula.
    Writing $\bz_j^\top \brho - d_j = \xi_j + \bz_j^\top\E_j[\brho]$
    and expanding the square, we find
    \bes
    W & 
    =
      & \sum_{j=1}^p\Bigl(\bz_j^\top \brho - d_j\Bigr)^2
    - p\|\brho\|^2
    =
        \sum_{j=1}^p\Bigl(\xi_j + \bz_j^\top\E_j[\brho]\Bigr)^2
    - p\|\brho\|^2
    \\&=&
  \sum_{j=1}^p
  \Big[
      \underbrace{\xi_j^2}_{(i)} 
      + \underbrace{2\xi_j \bz_j^\top\E_j[\brho]}_{(ii)}
      +\underbrace{(\bz_j^\top\E_j[\brho])^2 -\|\E_j[\brho]\|^2}_{(iii)}
      + \underbrace{\|\E_j[\brho]\|^2 -  \|\brho\|^2}_{(iv)}
  \Big].
    \ees
    This decomposition gives rise to 4 terms that we will bound separately.

    (i) $\E \sum_{j=1}^p \xi_j^2  \le 2\E \sum_{j=1}^p \sum_{i=1}^n \|\frac{\partial \brho}{\partial x_{ij}}\|^2$ by \eqref{SOS-upper-bound-expectation-zero}
    applied with respect to $\bz_j$ conditionally on $\bX_{-j}$
    for each $j\in[p]$.

    (ii)
    Since 
    $\E[(\bz_j^\top\E_j[\brho])^2]=\E[\|\E_j\brho\|^2]\le1$ by Jensen's inequality,
    the Cauchy-Schwarz inequality $\E[\sum_{j=1}^p a_j b_j]
    \le
    \E[\sum_{j=1}^p a_j^2]^{1/2}
    \E[\sum_{j=1}^p b_j^2]^{1/2}
    $
    yields
    $$\E\Big[\sum_{j=1}^p|\xi_j \bz_j^\top\E_j[\brho]|\Big]
    \le
    \sqrt p
    ~
    \E\Big[\sum_{j=1}^p\xi_j^2\Big]^{\frac12}
    \le
    \sqrt p
    \Big(2 \E \sum_{j=1}^p \sum_{i=1}^n \|\frac{\partial \brho}{\partial x_{ij}}\|^2
    \Big)^{\frac12}
    $$
    where the second inequality follows from the inequality
    derived for term (i).

    (iii)
    For the third term, set 
    $\chi_j = (\bz_j^\top\E_j[\brho])^2 -\|\E_j[\brho]\|^2$
    and note that
    $\E [ ( \sum_{j=1}^p \chi_j)^2 ]
  =\E\sum_{j=1}^p
\sum_{k=1}^p
    \chi_j \chi_k
    $.
    For the diagonal terms,
    $\E\sum_{j=1}^p \chi_j^2 = 
    2 \E\sum_{j=1}^p \|\E_j[\brho]\|^4
    \le 2p
    $
    because $\bz_j\sim N(0,\bI_n)$ is independent of $\E_j[\brho]$ and
    $\E[(Z^2-s^2)^2]= 2 s^4$ if $Z\sim N(0,s^2)$.

    For the non-diagonal terms we compute $\E[\chi_j\chi_k]$
    using \Cref{lemma:f-g-z_j-z_k} with
    $f^{(j,k)}(\bz_k)=\E_j[\brho]$ and $h^{(j,k)}(\bz_j)=\E_k[\brho]$
    conditionally on $(\bz_l)_{l\notin \{j,k\}}$.
    Thanks to $\|\brho\|\le 1$ this gives
    \begin{align*}
        \E[
        \chi_j\chi_k]
    \le
    \Cr{C10} \E \sum_{i=1}^n
    \pb{\Bigl[}
    \|\frac{\partial \E_k[\brho]}{\partial x_{ij} } \|^2
    + \|\frac{\partial \E_j[\brho]}{\partial x_{ik} } \|^2
    \pb{\Bigr]}
    &\le 
    \Cr{C10} \E \sum_{i=1}^n
    \pb{\Bigl[}
    \E_k[\|\frac{\partial \brho}{\partial x_{ij} } \|^2]
    +\E_j[ \|\frac{\partial \brho }{\partial x_{ik} } \|^2]
    \pb{\Bigl]}
    ,
    \end{align*}
    where the second inequality follows by dominated convergence
    for the conditional expectation
    (i.e., $(\partial/\partial x_{ij})\E_k[\brho] = \E_k[(\partial/\partial x_{ij})\brho]$ almost surely)
    and Jensen's inequality.
    Finally, summing over all pairs $j\ne k$ we find
    $$
        \sum_{j=1}^p
        \Bigl(
        \sum_{k=1, k\ne j}^p
        \E[
        \chi_j\chi_k]
        \Bigr)
        \le \Cr{C10} 
        \E\Bigl[
        \Bigl(
        p
        \sum_{j=1}^p
        \sum_{i=1}^n
        \|\frac{\partial \brho }{\partial x_{i\pb{j}} } \|^2
        \Bigr)
        +
        \Bigl(
        p
        \sum_{k=1}^p
        \sum_{i=1}^n
        \|\frac{\partial \brho }{\partial x_{i\pb{k}} } \|^2
        \Bigr)
        \Bigr]
        =
        2p\Cl{C10} 
        \E
        \sum_{j=1}^p
        \sum_{i=1}^n
        \|\frac{\partial \brho }{\partial x_{i\pb{j}} } \|^2
        .
    $$
    
    (iv)
    For the last term, using
    $|\|\E_j[\brho]\|^2 - \|\brho\|^2| \le \|\E_j[\brho]-\brho\| ~ \|\E_j[\brho]+\brho\|$
    and the Cauchy-Schwarz inequality we find
    \begin{align*}
        \E\Big|\sum_{j=1}^p
        \|\E_j[\brho]\|^2 - \|\brho\|^2
        \Big|
    &\le
    \E\Big[ \sum_{j=1}^p \|\E_j[\brho]-\brho\|^2\Big]^{\frac12}
    \E\Big[ \sum_{j=1}^p \|\E_j[\brho]+\brho\|^2\Big]^{\frac12}
    \le
\E\Big[ \sum_{j=1}^p\sum_{i=1}^n \|\frac{\partial\brho}{\partial x_{ij}}\|^2 \Big]^{\frac12}
    \sqrt{4 p}
    \end{align*}
    where the second inequality follows from $\|\brho\|\le 1$ and
    the Gaussian Poincar\'e inequality
    \cite[Theorem 3.20]{boucheron2013concentration}
    with respect to $\bz_j=\bX\be_j$ conditionally
    on $(\bz_k)_{k\in[p]\setminus\{j\}}$, which gives
    $\E_j[(\rho_l - \E_j[\rho_l])^2] \le 
    \E_j\sum_{i=1}^n (\frac{\partial \rho_l}{\partial x_{ij}})^2
    $ for every $l=1,...,n$.
\end{proof}

\begin{proof}[Proof of \Cref{big_cor}]
    By the product rule,
    $\frac{\partial}{\partial x_{ij}}[\brho\bfeta^\top]
    =
    \brho\rho_i \be_j^\top
    +
    (
    \frac{\partial\brho}{\partial x_{ij}}
    ) \brho^\top\bX
    +
    \brho
    [
    \frac{\partial\brho}{\partial x_{ij}}
    ]^\top\bX^\top$
    for $\bfeta=\bX^\top\brho$, 
    so that
    $
    \|\frac{\partial}{\partial x_{ij}}
    [\brho\bfeta^\top]
    \|_F
    \le |\rho_i| \|\brho\|
    + 2\|\bX\|_{op}\|\brho\| 
    \|\frac{\partial\brho}{\partial x_{ij}}\|
    $.
    Using $(a+b)^2\le 2a^2+2b^2$ and summing over $i\in[n],j\in[p]$,
    we find that
    the right-hand side of \eqref{eq:to-evaluate-SOS} satisfies
    $$
    \|\brho\|^2\|\bX\bfeta\|^2
    +\sum_{i=1}^n\sum_{j=1}^p
    \|\frac{\partial(\brho\bfeta^\top)}{\partial x_{ij}}\|_F^2
    \le 
    \|\bX\|_{op}^2\|\brho\|^4
    +
    2 p \|\brho\|^4
    + 4 \|\bX\|_{op}^2 \|\brho\|^2 
    \sum_{i=1}^n\sum_{j=1}^p
    \|\frac{\partial\brho}{\partial x_{ij}}\|^2
    .
    $$
    By \Cref{lemma:Davidson},
    $\E[\|\bX\|_{op}^2]\le (\sqrt p + \sqrt n)^2 + 1$.
    Thus if $\|\brho\|\le 1$ a.s., the right-hand side of \eqref{eq:to-evaluate-SOS}
    is bounded from above by an absolute constant times
    $RHS^2$ where $RHS$ is defined in 
    \Cref{big_cor}. By the Cauchy-Schwarz inequality to lower bound
    the left-hand side of
    \eqref{eq:to-evaluate-SOS}, 
    \begin{equation*}
    \E\Big|\|\bX^\top\brho\|^2 - p \|\brho\|^2 
    -
    \sum_{i=1}^n\sum_{j=1}^p 
    \Bigl[
    \be_j^\top\bX^\top\brho \frac{\partial \rho_i}{\partial x_{ij}}
    +
    \rho_i\be_j^\top\bX^\top
    \frac{\partial \brho}{\partial x_{ij}}
    \Bigr]
    \Big|
    =
    \E\Big|\|\bX^\top\brho\|^2 -\sum_{i=1}^n\sum_{j=1}^p \frac{\partial[\rho_i \eta_j]}{\partial x_{ij}} \Big|
    \le \C RHS.
    \end{equation*}
    By definition of $RHS$ and expanding the square in
    \eqref{novel-inequality-chi2} we also have
    $$
    \E\Big| p \|\brho\|^2  - \|\bX^\top\brho\|^2
    +2
    \sum_{i=1}^n\sum_{j=1}^p 
    \be_j^\top\bX^\top\brho \frac{\partial \rho_i}{\partial x_{ij}}
    -
    \sum_{j=1}^p 
    \Bigl(
    \sum_{i=1}^n
    \frac{\partial \rho_i}{\partial x_{ij}}
    \Bigr)^2
    \Big|
    \le \C RHS.
    $$
    If the left-hand sides of the two previous displays are written as
    $\E|U|$ and $\E|V|$ then the left-hand side of
    \eqref{eq:new_cor_second} is exactly $\E|U+V|$
    (cancelling out
    $\|\bX^\top\brho\|^2 - p \|\brho\|^2$)
    and the left-hand side of
    \eqref{eq:new_cor_first} is exactly $\E|2U+V|$
    (cancelling out
    $2
    \sum_{i=1}^n\sum_{j=1}^p 
    \be_j^\top\bX^\top\brho \frac{\partial \rho_i}{\partial x_{ij}}$
    ).
    
\end{proof}




\newpage
\section*{SUPPLEMENT}

\section{Rotational invariance of
regularized least-squares}

The crux of the following proposition is the rotational
invariance of the design. As in the rest of the paper,
let $(\bX,\bep)$ be independent such that
$\bX$ has iid $N(\mathbf{0},\bSigma)$ rows.
Consider $\bR\in O(n)$ a random rotation
distributed according to the Haar measure
(i.e., such that $\bR\bu$ is uniformly distributed on the sphere
for any deterministic $\bu$ with $\|\bu\|=1$).
Then by rotational invariance of the Gaussian measure of $\bX$,
$(\tbX,\tbep) \defas (\bR\bX,\bR\bep)$ is such that
$\tbX$ has iid $N(\mathbf{0},\bSigma)$ rows
and $\tbep$ is independent of $\bX$ with the uniform distribution
on the sphere of radius $\|\bep\|$.
If $\hbbeta$
is a penalized $M$-estimator with square loss
as in \eqref{hbbeta-squared-loss} then 
$\hbbeta(\bep + \bX\bbeta,\bX)
=\hbbeta(\bR\bep + \bR\bX\bbeta, \bR\bX)$
because $\|\bR(\by-\bX\bb)\|^2 = \|\by-\bX\bb\|\pb{^2}$ for all $\bb\in\R^p$.
Thus the distribution of $\hbbeta$ is unchanged if the noise $\bep$
is replaced by $\tbep = \|\bep\|_2 \bv$ where $\bv$ is uniformly
distributed on the sphere of radius $1$ and
independent of $(\bX,\|\bep\|)$.

\begin{proposition}
    \label{prop:eta_1}
    Set $\rho(u)=u^2/2$.
    Let $\hbbeta$ in \eqref{hbbeta-squared-loss} and
    $\df$ in \eqref{eq:def-df}.
    Assume that $\by$ has continuous distribution with respect to
    the Lebesgue measure.
    Assume that $\bX$ has iid $N({\bf 0,\bSigma})$ rows and
    that $\bep$ is independent of $\bX$. Then
    $\E[\xi^2]\le \C / n$ for
    \begin{align}
        \xi
        =
        \frac{(\by-\bX\hbbeta)^\top\bep - \|\bep\|^2(1-\df/n)}{
            (
            \|\bep\|^4
            +
            \|\bep\|^2\|\by-\bX\hbbeta\|^2
            )^{1/2}
    }
        .
        \label{xi}
    \end{align}
\end{proposition}
Consequently,
the random variable $\Rem$
defined just before \eqref{squareloss_RHS-sum_with_REM} satisfies
$$
|\Rem| = (n-\df)
\frac{
    |(\by-\bX\hbbeta)^\top\bep - \|\bep\|^2(1-\df/n)|
}
{
    \|\by-\bX\hbbeta\|^2 + n \|\bSigma^{1/2}\bh\|^2
}
=
|\xi|
(n-\df)
\frac{
    (\|\bep\|^4 + \|\bep\|^2\|\by-\bX\hbbeta\|^2)^{1/2}
}
{
    \|\by-\bX\hbbeta\|^2 + n \|\bSigma^{1/2}\bh\|^2
}
$$
and using $\|\bep\|\le \|\by-\bX\hbbeta\| + \|\bX\bSigma^{-1/2}\|_{op}
\|\bSigma^{1/2}\bh\|$ to bound the rightmost numerator,
\begin{equation}
\E|\Rem|
\le \C n
\E[\xi^2]^{1/2}
(1+\E[\|n^{-1/2}\bX\bSigma^{-1/2}\|_{op}^4]^{1/2})
\le \C(\gamma) \sqrt n
\label{eq:bound-Rem-noise}
\end{equation}
if $p/n\le \gamma$ by \Cref{lemma:Davidson}.

\begin{proof}[Proof of \Cref{prop:eta_1}]
    As explained above the proposition, let $\bR\in O(n)$
    be a random rotation independent of $(\bep,\bX)$, so that
    $(\bR\bX,\|\bep\|)$ is independent of $\bz\defas \bR\bep /\|\bep\|$.
    Conditionally on $(\bR\bX,\|\bep\|)$, define the function
    $\bff:\R^n\to\R^n$ by
    $\bff(\bz) \defas \bz-\|\bep\|^{-1}\bR\bX[\hbbeta(\|\bep\|\bR^\top\bz + \bX\bbeta,\bX)-\bbeta]$
    so that $\bff(\bz)=\|\bep\|^{-1}\bR(\by-\bX\hbbeta)$
    and $\|\bff(\bz)\| = \|\bep\|^{-1}\|\by-\bX\hbbeta\|$.
    Then dividing the numerator and denominator by $\|\bep\|^2$ in
    \eqref{xi},
    $$
    \xi = \frac{\bff(\bz)^\top\bz  - (1-\df/n)}{
        \sqrt{ 1+ \|\bff(\bz)\|^2}
    }.
    $$
    With $\bH = (\partial/\partial\by)(\bX\hbbeta(\by,\bX))$ holding $\bX$ fixed, viewing $\by = \|\bep\|\bR^\top\bz + \bX\bbeta$
    as a function of $\bz$ we have
    by the chain rule
    $\nabla \bff(\bz)^\top = \bI_n - \bR\bH\bR^\top$.
    Since $\bH$ is almost surely positive semi-definite
    with eigenvalues in $[0,1]$ by \Cref{prop:jacobian-psd-htheta},
    this proves that $\bff$ is 1-Lipschitz.
    By the chain rule,
    $\tbf(\bz) \defas \bff(\bz)/\sqrt{1+\|\bff(\bz)\|^2}$ 
    has Jacobian
    \begin{equation}
    \nabla \tbf(\bz)^\top
    = \frac{\nabla \bff(\bz)^\top}{\sqrt{1+\|\bff(\bz)\|^2}}
    - \frac{\bff(\bz)\bff(\bz)^\top}{(1+\|\bff(\bz)\|^2)^{3/2}}
    \label{tbf-jac}
    \end{equation}
    which has operator norm bounded by 2, hence $\tbf(\bz)$ is 2-Lipschitz.
    Since $\bz$ is uniformly distributed on the sphere,
    with $\bP_{\bz} = (\bI_n - \bz\bz^\top)$, \cite[Lemma B.1]{bellec2021asymptotic} shows
    \begin{align*}
    \E[(n \tbf(\bz)^\top\bz - \trace[\nabla \tbf(\bz)^\top\bP_{\bz}])^2]
    &\le n \E[\|\tbf(\bz)\|^2] + (1-\tfrac{2}{n})^{-1}\E[\|\nabla \tbf(\bz)^\top\bP_{\bz}\|_F^2]
  \\&\le \C n.
    \end{align*}
    Since 
    $\trace[\nabla f(\bz)^\top]=n-\df$ and
    $|\trace[\nabla \tbf(\bz)^\top\bP_{\bz}]-(1+\|\bff(\bz)\|^2)^{1/2}(n-\df)|
    \le
    \C
    $ a.s. by \eqref{tbf-jac},
    \begin{align*}
        \E[(n\xi)^2]
        &=
    \E[(
    n\tbf(\bz)^\top\bz - (1+\|\bff(\bz)\|^2)^{-1/2}\trace[\nabla \bff(\bz)^\top])^2]
      \\&\le
    2\E[(n \tbf(\bz)^\top\bz - \trace[\nabla \tbf(\bz)^\top\bP_{\bz}])^2]
    + 2
    \E[
    |\trace[\nabla \tbf(\bz)^\top\bP_{\bz}]-(1+\|\bff(\bz)\|^2)^{-1/2}(n-\df)|^2
    ]
      \\&\le \C n + \C.
    \end{align*}
\end{proof}

\section{Proofs of auxiliary results}

\subsection{Proof of some properties of the Jacobian
    \texorpdfstring{$\bV=(\partial/\partial\by)\hbpsi$}{of psi}
}
\label{sec:proof-jacobian-psd}

\PropPsiPsd*

The following lemma is useful
to prove \Cref{prop:jacobian-psd-htheta}.
\begin{lemma}
    \label{lemma:jacobian-psd-htheta-FULL}
    Assume that $\rho$ is convex differentiable  and that $\psi=\rho'$ is 1-Lipschitz.
    Then $\rho(u)=\min_{v\in\R}\{(u-v)^2/2 + h(v) \}$  for some convex
    function $h$. Consider 
    \begin{equation}
        \label{eq:hbb-htheta}
    (\hbb,\hbtheta) \in \argmin_{\bb\in\R^p,\btheta\in\R^n}
    \|\by-\bX\bb - \btheta\|^2/(2n) + g(\bb) + \sum_{i=1}^n h(\theta_i)/n.
    \end{equation}
    Then for every solution $\hbbeta$ to the optimization problem
    \eqref{M-estimator-rho},
    there exists a solution
    $(\hbb,\hbtheta)$ to the optimization problem \eqref{eq:hbb-htheta}
    such that
    $\hbbeta=\hbb$ and $\psi(\by-\bX\hbbeta) = \by-\bX\hbb-\hbtheta$.
\end{lemma}

\begin{proof}[Proof of Proposition~\ref{lemma:jacobian-psd-htheta-FULL}]
    If $\psi=\rho'$ is 1-Lipschitz 
    then $f(u) =  u^2/2 - \rho$ is convex
    and 1-smooth (in the sense that $f'$ is again 1-Lipschitz),
    so that its Fenchel conjugate $f^*(v) = \max_{u\in\R} \{ uv - f(u) \}$
    is 1-strongly convex (in the sense that $v\mapsto f^*(v) - v^2/2$ is convex).
    Let $h(v) = f^*(v) - v^2/2$.
    For this choice of $h$, we have thanks to $f^{**} = f$
    $$\min_{v\in\R} \big\{ \frac{(u-v)^2}{2} + h(v) \big\}
    = \frac{u^2}{2} - \max_{v\in\R} \{uv - f^*(v)\} = \frac{u^2}{2}
    - \big(\frac{u^2}{2} - \rho(u)\big) = \rho(u).
    $$
    If $\rho$ is the Huber loss \eqref{eq:Huber-loss} this construction
    was already well studied and in this case  $h(v)=|v|$,
    see for instance \cite[Section 6]{donoho2016high} or
    \cite{dalalyan2019outlier} and the references therein.

    Next consider the $M$-estimator with square loss
    and design matrix $[\bX | \bI_n]\in\R^{n\times (p+n)}$ defined by
    \eqref{eq:hbb-htheta}.
    The KKT conditions are given by
    \begin{equation}
        \label{eq:KKT-A1}
    \bX^\top(\by - \bX\hbb - \hbtheta)\in n \partial g(\hbb),
    \qquad
    y_i-\bx_i^\top\hbb - \htheta_i \in \partial h(\htheta_i), \qquad i\in[n]
    \end{equation}
    where $\partial g$ and $\partial h$ denote the subdifferentials of $g$ and $h$. That is,
    $(\hbb,\hbtheta)$  is solution to \eqref{eq:hbb-htheta}
    if and only if \eqref{eq:KKT-A1} holds.
    We claim that one solution of the optimization problem
    \eqref{eq:hbb-htheta} is given by
    $(\hbb,\hbtheta) = (\hbbeta,\by-\bX\hbbeta - \psi(\by-\bX\hbbeta))$
    where $\hbbeta$ is any solution in \eqref{M-estimator-rho}.
    Indeed, the first part in \eqref{eq:KKT-A1}
    holds by the optimality conditions $\bX^\top\psi(\by-\bX\hbbeta)\in n \partial g(\hbbeta)$ of $\hbbeta$ as a solution to the optimization problem
    \eqref{M-estimator-rho}; it remains to check 
    that $y_i-\bx_i^\top\hbb - \htheta_i\in\partial h(\htheta_i)$
    holds for all $i\in[n]$, or equivalently
    that 
    \begin{equation}
        \label{eq:previous-display-psi-h-i}
    \psi(y_i-\bx_i^\top\hbbeta)\in \partial h(
        y_i-\bx_i^\top\hbbeta - \psi(y_i-\bx_i^\top\hbbeta)
    )
    \end{equation}
    by definition of $\htheta_i$.
    By additivity of the subdifferential,
    $v + \partial h(v)=\partial f^*(v)$.
    Furthermore $u\in\partial f^*(v)$ if and only 
    if $f^{**}(u) + f^*(v) = uv$ by property of the Fenchel conjugate,
    where here we have
    $f^{**}(u)=f(u)=u^2/2 - \rho(u)$ since here $f$ is convex and finite valued.
    We also have $v\in \partial f(u)$
    iff $f(u) + f^*(v) = uv$,
    and here $\partial f(u) = \{u - \psi(u)\}$ is a singleton.
    Combining these pieces together, for any $u,v\in\R$ we find
    \begin{align*}
        v= u - \psi(u)
     &\text{ iff }
        v\in \partial f(u) 
     \\&\text{ iff }
        f(u) + f^*(v) = uv
     \\&\text{ iff }
     f^{**}(u) + f^*(v) = uv
     \\&\text{ iff }
        u \in \partial f^*(v)
     \\&\text{ iff }
        u - v \in\partial h(v).
    \end{align*}
    Hence taking $u=y_i - \bx_i^\top\hbbeta$
    and $v=u-\psi(u)$, the previous sentence implies that
    $\psi(u)\in \partial h(u-\psi(u))$
    and the previous display \eqref{eq:previous-display-psi-h-i} must hold for all $i\in [n]$.
    This proves that the given $(\hbb,\hbtheta)$ is solution
    to \eqref{eq:hbb-htheta}.
\end{proof}
\begin{proof}[Proof of \Cref{prop:jacobian-psd-htheta}]
    By \cite[Proposition J.1]{bellec_zhang2019second_poincare}
    applied to $(\hbb,\hbtheta)$ with design matrix
    $[\bX|\bI_n]$, the map
    $\by\mapsto \by-\bX\hbb - \hbtheta$ is 1-Lipschitz on $\R^n$,
    and for almost every $\by\in\R^n$ this map has
    symmetric positive semi-definite Jacobian.
    Since $\by-\bX\hbb - \hbtheta = \psi(\by-\bX\hbbeta)$,
    this proves the first bullet point
    of Proposition~\ref{prop:jacobian-psd-htheta}.

    For the second bullet point, under \Cref{assum:g}(iii.a)
    the claim is proved in \Cref{prop:Huber-Lasso}.
    Under \Cref{assum:g}(i) or (ii),
    by \Cref{lemma:ninja-variant}
    and \eqref{eq:chain-rule-hard} we have
    $\bV = (\partial/\partial\by)\hbpsi=
    \bD(\bI_n - (\partial/\partial\by)\bX\hbbeta)
    =
    \bD(\bI_n-\bX\hbA\bX^\top\bD)$
    where $\bD=\diag(\bpsi')$.
    If $\bH = (\partial/\partial\by)(\bX\hbbeta) = \bX\hbA\bX^\top\bD$,
    we bound $\df=\trace\bH$ as follows:
    if $\bD^\dagger$ denotes the pseudo-inverse of $\bD$,
    using the commutation property of the trace and $\bD=\bD\bD^\dagger\bD$,
    $$|\hat I| - \df=
    \trace[\bD^\dagger \bD - \bH]
    =
    \trace[\bD^\dagger \bD - \bX\hbA\bX^\top\bD\bD^\dagger\bD]
    = \trace[
    (\bD^\dagger)^{1/2} \bV (\bD^\dagger)^{1/2}
    ]\ge 0
    $$
    where the last inequality is thanks to $\bV$ being symmetric psd.
    This proves $\df\le |\hat I|$.
\end{proof}

\subsection{Elastic-Net penalty and Huber Lasso}
\label{proof:ENet-Huber-Lasso}

\begin{proof}[Proof of Proposition~\ref{prop:specific-ENet}]
    The KKT conditions read
    $\bX^\top\hbpsi - n\mu\hbbeta \in n\lambda \partial \|\hbbeta\|_1$ 
    where $\partial \|\bb\|_1$ denotes the sub-differential of the $\ell_1$ norm at $\bb\in\R^p$.
    We first prove that the KKT conditions hold strictly
    with probability one, in the sense that
    $$\P(\forall j\in[p], ~ j\notin\hat S \text{ implies }\be_j^\top\bX^\top\hbpsi
    \in (-n\lambda,n\lambda) )=1.$$
    Let $j_0$ be fixed and let $\hbalpha$ be the solution
    to the same optimization problem as $\hbbeta$,
    with the additional constraint 
    that the $j_0$-th coordinate is always set to 0.
    Then $\{j_0\not \in \hat S\} = \{\hbalpha = \hbbeta \}$
    as the solution of each optimization problem is unique
    thanks to $\mu>0$.
    Let $\bX_{-j_0}$ be $\bX$ with $j_0$-th column removed.
    The conditional distribution of
    $\bX\be_{j_0}$ given $(\bX_{-j_0},\by)$ is continuous
    because $(\bX,\by)$ has continuous distribution.
    Hence
    $\be_{j_0}^\top \bX^\top\psi(\by - \bX\hbalpha)$ also has continuous
    distribution conditionally on $(\bX_{-j_0},\by)$ when
    $\psi(\by - \bX\hbalpha)\ne \mathbf{0}$, so that
    $\P(\be_{j_0}^\top \bX^\top\psi(\by - \bX\hbalpha) \in \{-\lambda n, \lambda n\} | \bX_{-j_0},\by) = 0$
    because a continuous distribution has no atom.
    The unconditional probability is also 0 by the tower property.
    This shows that
    $\P(j_0\not\in \hat S \text{ and }
    \be_{j_0}^\top\bX^\top\hbpsi
    \in \{-n\lambda,n\lambda\} )=0$ for all $j_0$. The union
    bound over all $j_0\in[p]$ proves that the KKT conditions
    hold strictly with probability one, as desired.
    \footnote{
    Similar arguments to prove that the KKT conditions hold
    strictly are used in \cite[Proposition 3.9]{bellec_zhang2018second_stein}
    for the Lasso
    or \cite[Lemma L.1]{bellec_zhang2019second_poincare} for the Group-Lasso.
    The above argument is provided for completeness.}

    The maps $(\by,\bX)\mapsto\hbbeta$ and $(\by,\bX)\mapsto \hbpsi$
    are Lipschitz continuous on every compact
    by Proposition~\ref{prop:constant-bv-e_j}(i) as $\bSigma$ is invertible.
    At a point $(\by_0,\bX_0)$ where the KKT conditions hold
    strictly, the KKT conditions stay strict and $\hat S$
    stay the same in a neighborhood of $(\by_0,\bX_0)$
    because the continuity of 
    $(\by,\bX) \mapsto \be_j^\top\bX^\top \hbpsi - n \mu\hbeta_j$
    ensure that $\be_j^\top\bX^\top \hbpsi - n \mu\hbeta_j$
    stay bounded away from $\{-n\lambda, n\lambda\}$
    for every $j\in[p]$ not in the active set at $(\by_0,\bX_0)$.
    Furthermore,
    by \eqref{lipschitz-fundamental-inequality} there exists an open
    set $U\subset \R^n\times \R^{n\times p}$
    with $U\ni(\by_0,\bX_0)$ such that the maps $(\by,\bX)\mapsto\hbbeta$
    and $(\by,\bX)\mapsto \hbpsi$ are Lipschitz in $U$, and
    the chain rule
    \eqref{eq:chain-rule-hard}
    yields
    $(\partial/\partial\by)\hbpsi
    =
    \diag(\bpsi')(\bI_n - \bX(\partial/\partial\by)\hbbeta)
    $ for almost every $(\by,\bX) \in U$.
    In a neighborhood of a point $(\by_0,\bX_0)$ where the KKT conditions
    hold strictly and where the aforementioned chain rule holds,
    since $\hat S$ is locally constant we have
    $(\partial/\partial\by) \hbbeta_{\hat S^c} = \mathbf{0}_{\hat S^c\times [n]}$
    as well as
    $$\bX_{\hat S}^\top\diag(\bpsi')[\bI_n -\bX(\partial/\partial \by)\hbbeta]
    - n \mu (\partial/\partial \by)\hbbeta
    = \mathbf{0}_{\hat S\times [n]}.$$
    By simple algebra, this implies
    $(\partial/\partial \by)\hbbeta_{\hat S}
    = (\bX_{\hat S}^\top\diag(\bpsi') \bX_{\hat S} + \mu n\bI_{|\hat S|})^{-1}
    \bX_{\hat S}^\top \diag(\bpsi')
    $ and the desired expressions for
    $(\partial/\partial\by)\bX\hbbeta$
    and
    $(\partial/\partial\by)\hbpsi$.

\end{proof}

\begin{proof}[ Proof of Proposition~\ref{prop:Huber-Lasso} ]
    For the Huber loss with $\ell_1$-penalty, the M-estimator $\hbbeta$
    satisfies
    $$(\hbbeta,\hbtheta)=\argmin_{(\bb,\btheta)\in\R^p\times \R^n}
    \|\bX\bb + \kappa \btheta - \by\|^2/(2n) + \lambda(\|\bb\|_1 + \|\btheta\|_1)
    $$
    where $\kappa>0$ is some constant,
    see e.g. \cite{dalalyan2019outlier} and the references therein
    or Proposition~\ref{lemma:jacobian-psd-htheta-FULL}
    with $h(\cdot)$ proportional to $|\cdot|$ for the Huber loss.
    Let $\bbetabar=(\hbbeta,\hbtheta)$. Then $\bbetabar$
    is a Lasso solution with data $(\by,\bXbar)$
    where the design matrix is $\bXbar = [\bX|\kappa\bI_n] \in \R^{n\times (n+p)}$.

    In this paragraph, we show that if $\bX$ has continuous distribution
    then
    $\bXbar$ satisfies Assumption 3.1 of \cite{bellec_zhang2018second_stein}
    with probability one.
    That assumption requires that for any $(\delta_j)_{j\in[p+n]}
    \{-1,+1\}^{p+n}$ and any columns $\bc_{j_1},...,\bc_{j_{n+1}}$
    of $\bXbar$ with $j_1<...<j_{n+1}$, the matrix
    \begin{equation}
        \label{eq:matrix-intertible-probability-one}
    \begin{pmatrix}
        \bc_{j_1} & \dots & \bc_{j_{n+1}} \\
        \delta_{j_1} & \dots & \delta_{j_{n+1}}
    \end{pmatrix}
    \in\R^{(n+1) \times (n+1)}
    \end{equation}
    has rank $n+1$.
    We reorder the columns so that any column of the form
    $\bc_{p+i},i\in[n]$ is the $i$-th column after reordering,
    and note that $\bc_{p+i} = \kappa \be_i$.
    Then there exists a value of $\bX\in\R^{n\times p}$ such that the above matrix,
    after reordering the columns,
    is equal to
    $$
    \begin{pmatrix}
        \kappa \bI_n  &|& \mathbf{0}_{n\times 1}  \\
        \delta_{k_1} ~ \dots ~ \delta_{k_{n}} &| & \delta_{k_{n+1}}
    \end{pmatrix}
    $$
    for some permutation $(k_1,...,k_{n+1})$ of $(j_1,...,j_{n+1})$.
    Since the previous display has nonzero determinant
    $\kappa^n\delta_{k_{n+1}}$,
    the determinant of matrix
    \eqref{eq:matrix-intertible-probability-one},
    viewed as a polynomial of the coefficients of $\bX$,
    is a non-zero polynomial. Since non-zero polynomials
    have a zero-set of Lebesgue measure 0 \cite{math_stackexchange920302},
    this proves that \eqref{eq:matrix-intertible-probability-one}
    is rank $n+1$ with probability one.

    Hence with probability one,
    by Proposition 3.9 in \cite{bellec_zhang2018second_stein},
    the solution $\bbetabar\in\R^{n+p}$ is unique,
    $\|\bbetabar\|_0\le n$
    and the KKT conditions of the optimization problem of $\bbetabar$
    hold strictly almost everywhere in $(\by,\bX)$
    (see \cite{tibshirani2013lasso} for related results).
    This shows that the sets $\hat S$ and $\hat I$,
    viewed as a function of $\by$ while $\bX$ is fixed,
    are constant in a neighborhood of $\by$ for almost every $(\by,\bX)$.
    Now the set of $\{i\in[n]:\htheta_i\ne 0\}$ exactly correspond 
    to the outliers $\{i\in[n]:\psi'(y_i-\bx_i^\top\hbbeta)=0\}=[n]\setminus \hat I$ 
    and 
    $\|\bbetabar\|_0\le n$ holds if and only if $|\hat S| + (n-|\hat I|) \le n$.
    This proves that  $|\hat S|\le |\hat I|$ almost surely.
    Furthermore, almost surely in $(\by,\bX)$,
    the derivative of $\by\mapsto \bXbar\bbetabar = \bX\hbbeta + \kappa\hbtheta$ exists
    and is equal to the orthogonal projection onto the linear span
    of $\{\be_i, i\in[n]\setminus \hat I\} \cup \{ \bX\be_j, j\in \hat S \}$.
    We construct an orthonormal basis of this linear span as follows:
    First by considering the vectors $\{\be_i,i\in [n]\setminus \hat I\}$
    and then completing by a basis $(\bu_k)_{k\in \hat S}$
    of the orthogonal complement of $\{\be_i,i\in [n]\setminus \hat I\}$.
    Note that this orthogonal complement is exactly the column span
    of $\pb{\diag(\bpsi')}\bX_{\hat S}$.
    The orthogonal projection onto the linear span of
    $\{\be_i, i\in[n]\setminus \hat I\} \cup \{ \bX\be_j, j\in \hat S \}$
    is thus
    $(\partial/\partial\by)\bXbar\bbetabar = \sum_{i\in[n]\setminus \hat I} \be_i\be_i^\top + \sum_{k\in\hat S} \bu_k\bu_k^\top$.
    Since $\diag(\bpsi')$ is constant in a neighborhood of $\by$ and
    $\diag(\bpsi')$ zeros out all rows corresponding to outliers,
    \begin{align*}
    \diag(\bpsi')(\partial/\partial\by)\bX\hbbeta
    &=
    (\partial/\partial\by) \diag(\bpsi')\bX\hbbeta
  \\&=
    (\partial/\partial\by) \diag(\bpsi')\bXbar\bbetabar
  \\&=
    \diag(\bpsi')(\partial/\partial\by)\bXbar\bbetabar
    =
    \textstyle
    \sum_{k\in \hat S}\bu_k \bu_k^\top 
    \end{align*}
    which is exactly the orthogonal projection $\hbQ$ defined in the proposition,
    as desired.
    The almost sure identity
    $(\partial / \partial \by )\hbpsi = \diag(\bpsi') - \hbQ$
    is obtained by the chain rule: Here
    $\psi$ is differentiable at $y_i-\bx_i^\top\hbbeta$
    for all $i\in[n]$ with probability one
    since the fact that the KKT conditions of $\bbetabar$
    hold strictly imply that no $y_i-\bx_i^\top\hbbeta$
    is a kink of $\psi$.
\end{proof}

\section{Lasso: Lipschitz conditions}
\label{sec:appendix_lasso}
\begin{lemma}[deterministic argument]
    \label{lemma:Xbar}
Let $n,\bar p\ge 1$ be integers and $m\in [\bar p]$.
Let $\bA \in\R^{n\times \bar p}$,
and define the Lasso
$\hbb =\argmin_{\bb\in\R^{\bar p}} \|\bA\bb-\bybar\|^2/(2n) + \frac{\lambda}{\sqrt n}
\|\bb\|_1$
where $\bybar = \bA\bb^* + \bz$ for some $\bb^*\in\R^{\bar p}$.
Then
\begin{equation}
\|\hbb\|_0
\le \phi_{\max}(\bA^\top\bA)
\max\Bigl\{ \frac{2 \|\bz\|^2}{\lambda^2 n}
    , \frac{4\|\bb^*\|_0}{n \kappa^2}
\Bigr\}
\quad
\text{ where }
\quad
\kappa^2 = \inf_{\bb\in\R^{\bar p}:
    \|\bb\|_1 < \|\bb^*\|_1
    }
    \Bigl[\frac{\|\bA(\bb-\bb^*)\|^2}{n\|\bb-\bb^*\|^2}\Bigr].
    \label{eq:upper-bound-sparsity-hbb}
\end{equation}
Let $\tbA\in\R^{n\times \bar p}$
and
$\tbb =\argmin_{\bb\in\R^{\bar p}} \|\tbA\bb-(\tbA\bb^* + \bz)\|^2/(2n) + \frac{\lambda}{\sqrt n}
\|\bb\|_1$. Then for any psd $\bSigmabar\in\R^{\bar p\times \bar p}$,
\begin{multline}
    \label{eq:multiline-lasso-uberlasso-lipschitz}
\min\Bigl\{1,
    \tfrac{\|\tbA(\hbb-\tbb)\|^2}{
        \|\bSigmabar{}^{1/2}(\hbb-\tbb)\|^2
    }
\Bigr\}
\max\Bigl\{
    \|\bSigmabar^{1/2}(\hbb-\tbb)\|,
    \|\bA\hbb-\tbA\tbb\|
\Bigr\}
\\\le
\|(\tbA-\bA)\bSigmabar{}^{-1/2}\|_{op}
\bigl[
\|\bybar-\bA\hbb\|
+\|\bSigmabar{}^{1/2}(\hbb-\bb^*)\|
(1+2\|\tbA\bSigmabar{}^{-1/2}\|_{op})
\bigr]
\end{multline}
\end{lemma}
\begin{proof}
    The KKT conditions read $\bA^\top(\bybar-\bX\hbb) = \lambda\sqrt n
    \partial\|\hbb\|_1$. Multiplying the KKT conditions by
    $\hbb-\bb^*$ we obtain
    $$
    \|\bA(\hbb-\bb^*)\|^2
    +
    \|\bA\hbb - \bybar\|^2
    \le
    \|\bz\|^2
    + 2 \sqrt n \lambda(\|\bb^*\|_1 - \|\hbb\|_1).
    $$
    We distinguish two cases, based on which of the two terms in the
    right-hand side is greater. If $\|\bz\|^2\ge 2 \sqrt n \lambda(\|\bb^*\|_1 - \|\hbb\|_1)$ then we find
    $\|\bA\hbb - \bybar\|^2\le 2 \|\bz\|^2$, and using the KKT conditions
    gives
    $\|\hbb\|_0\le \frac{1}{\lambda^2n}\|\bA^\top(\bybar-\bA\hbb)\|^2
    \le \phi_{\max}(\bA^\top\bA) 2 \|\bz\|^2 / (\lambda^2n)$.
    Otherwise, we have
    $\|\bz\|^2< 2 \sqrt n \lambda(\|\bb^*\|_1 - \|\hbb\|_1)$
    and
    $$
    \|\bA(\hbb-\bb^*)\|^2
    +
    \|\bA\hbb - \bybar\|^2
    < 4\sqrt n \lambda(\|\bb^*\|_1 - \|\hbb\|_1)
    \le 4 \sqrt{n} \lambda\|\bb^*\|_0^{1/2}\|\bb^*-\hbb\|
    \le 4 \lambda \|\bb^*\|_0^{1/2}\kappa^{-1} \|\bA(\hbb-\bb^*)\|.
    $$
    Using $4uv\le 4u^2+v^2$ for $v=\|\bA(\hbb-\bb^*)\|$, the term $v^2$
    cancel out and
    $\|\bybar-\bA\hbb\|^2\le 4\lambda^2\|\bb^*\|_0 \kappa^{-2}$.
    Using again
    $\lambda^2n\|\hbb\|_0\le\|\bA^\top(\bybar-\bA\hbb)\|^2
    \le \phi_{\max}(\bA^\top\bA)\|\bybar-\bA\hbb\|^2$
    completes the proof of \eqref{eq:upper-bound-sparsity-hbb}.
    By the same argument as the proof of \Cref{prop:lipschtiz2-gamma-less-one} with $\mu_\rho=1$
    (square loss) we 
    obtain \eqref{eq:multiline-lasso-uberlasso-lipschitz}.
\end{proof}

\begin{lemma}
    Let $d_*\in(0,1),\varphi>1,\gamma>1,a_*>0$ be arbitrary constants.
    Assume $p/n\le \gamma$ and
    let $\bX$ have iid entries with distribution $N(\mathbf 0,\bSigma)$
    with $\Sigma_{jj}=1$ for all $j\in[p]$ and
    $\|\bSigma\|_{op}\|\bSigma^{-1}\|_{op}\le \varphi$.
    Then there exist constants $t_1,t_2,k_2>0$ depending on $(d_*,\gamma,\varphi)$
    only and constants
    $t_3,t_4,k_4>0$ depending on $(d_*,\gamma,\varphi,a_*)$ only such that
    as $n,p\to+\infty$ while $(d_*,\gamma,\varphi,a_*)$ remain fixed we have
    \begin{align}
    \label{RIP_lasso}
    \P(\forall \bh\in\R^p:
    \|\bSigma^{1/2}\bh\|=1,
    \|\bh\|_0\le d_*n 
    &\Rightarrow \|\tfrac{1}{\sqrt n}\bX\bh\|> t_1
    ) \to 1,
    \\
    \label{RE_lasso}
    \P(\forall \bh\in\R^p:
    \|\bSigma^{1/2}\bh\|=1,
    \|\bh\|_1\le 2 \sqrt{k_2n}\|\bh\|_2
    &\Rightarrow \|\tfrac{1}{\sqrt n}\bX\bh\| > t_2
    ) \to 1,
    \\\P(\forall \bh,\btheta\in\R^{p+n}:
    \|\bSigma^{1/2}\bh\|^2 + \|\btheta\|^2=1,
    \|\bh\|_0+\|\btheta\|_0\le d_*n 
    &\Rightarrow \|\tfrac{1}{\sqrt n}\bX\bh + a_*\btheta\| > t_3
    ) \to 1,
    \label{RIP_huber}
    \\\P(\forall \bh,\btheta\in\R^{p+n}:
    \|\bSigma^{1/2}\bh\|^2 + \|\btheta\|^2=1,
    \tfrac{ \|\bh\|_1+\|\btheta\|_1}{(\|\btheta\|^2+\|\bh\|^2)^{1/2}}\le \sqrt{k_4n} 
    &\Rightarrow \|\tfrac{1}{\sqrt n}\bX\bh + a_* \btheta\| > t_4
    ) \to 1.
    \label{RE_huber}
    \end{align}
\end{lemma}
In the following proof, for a random matrix $\bZ\in\R^{n\times p}$
and two subspaces $V_L\subset \R^n$ and $V_R\subset\R^p$
of dimension $d_L$ and $d_R$ respectively,
we call the restriction
of $\bZ$ to $V_L$ and $V_R$ the random matrix $\bG=\bQ_L^\top\bZ\bQ_R\in\R^{d_L,d_R}$
where $\bQ_L\in\R^{n,d_L},\bQ_R\in\R^{p,d_R}$
have orthonormal columns such that $\bQ_L\bQ_L^\top$ is the orthogonal projection
onto $V_L$ and $\bQ_R\bQ_R^\top$ is the orthogonal projection onto $V_R$.
If $\bZ$ has iid $N(0,1)$ entries then $\bG$ also has iid $N(0,1)$ entries
by rotational invariance.
\begin{proof}
    The proof of \eqref{RIP_lasso} is a minor variant of the union bound argument
    in \cite[Proposition 2.10]{blanchard2011compressed}. In short, thanks to
    the explicit formula for the smallest density of a Wishart matrix
    with identity covariance from \cite{edelman1988eigenvalues},
    the argument in \cite[Proof of Lemma 4.1]{chen2005condition} gives
    $\P(\phi_{\min}(\bG^\top\bG/n)
    \le t^2 )
    \le (\frac{etn}{n-d+1})^{n-d+1} / \sqrt{2\pi(n-d+1)}$
    if $\bG\in\R^{n\times d}$ has iid $N(0,1)$ entries.
    With $d=\lfloor d_* n\rfloor$,
    we apply this inequality to all $\binom{p}{d}$
    Gaussian matrices obtained as the restriction of
    $\bX\bSigma^{-1/2}$ to a $d$-dimensional subspace
    generated as the span of $d$ columns of
    $\bSigma^{1/2}$.
    Taking the union bound, the probability of the union
    is bounded from above by
    $\binom{p}{d}(\frac{etn}{n-d+1})^{n-p+1} / \sqrt{2\pi(n-d+1)}$
    which converges to 0 if $t$ is a small enough constant
    thanks to $\binom{p}{d}=\binom{p}{\lfloor d_*n\rfloor}\le e^{n d_*\log(e/d_*)}$.
    By a direct application of \cite[Lemma 2.7]{lecue2014sparse},
    \eqref{RE_lasso} is then obtained by choosing the constant $k_2=k_2(d_*,\gamma,\varphi)>0$
    small enough.

    Next we focus on \eqref{RIP_huber}. We do not attempt to optimize
    the constants.
    The event $\|\bX\bSigma^{-1/2}\|_{op}\le 2\sqrt n + \sqrt p$
    has probability approaching one
    \cite[Theorem II.13]{DavidsonS01}.
    In this event,
    simultaneously for all $(\bh,\btheta)$ such that
    $\tfrac{a_*}{2}\|\btheta\|\ge 
    \max\{\frac{a_*}{2}, (2+\sqrt \gamma) \}\|\bSigma^{1/2}\bh\|$, 
    by the triangle inequality
    \begin{align*}
    \|\tfrac{1}{\sqrt n}\bX\bh + a_*\btheta\|
    \ge 
    a_*\|\btheta\| - (2+\sqrt \gamma)\|\bSigma^{1/2}\bh\|
    &\ge 
    \tfrac{a_*}{2}\|\btheta\|\qquad\qquad \text{ thanks to }
    \tfrac{a_*}{2}\|\btheta\|\ge 
     (2+\sqrt \gamma) \|\bSigma^{1/2}\bh\|
    \\
    &\ge 
    \tfrac{a_*}{4} (\|\bSigma^{1/2}\bh\|+\|\btheta\|)
    \quad\text{ thanks to }\|\bSigma^{1/2}\bh\|^2\le \|\btheta\|^2
    .
    \end{align*}
    We now consider $(\bh,\btheta)$ such that the reverse inequality
    $\tfrac{a_*}{2}\|\btheta\| < \max\{\frac{a_*}{2}, (2+\sqrt \gamma) \}\|\bSigma^{1/2}\bh\|$ holds.
    Let $O\subset [n]$ and $S\subset[p]$ be such that $|O|+|S| = \lfloor d_*n \rfloor$
    and let $\bG\in \R^{(n-|O|)\times |S|}$ be the Gaussian matrix
    obtained by
    restriction of the Gaussian matrix $\bX\bSigma^{-1/2}$
    restricted on the left to the rows indexed in $[n]\setminus O$,
    and restricted on the right to the subspace
    given by the linear span of $\{\bSigma^{1/2}\be_j, j\in S\}$.
    For any $\btheta$ supported in $O$ and $\bh$ supported in $S$,
    since the orthogonal projection $\bP_O^\perp=\sum_{i\in[n]\setminus O}\be_i\be_i^\top$
    decreases the norm,
    $$
    \|\tfrac{1}{\sqrt n}\bX\bh + a_*\sqrt n\btheta\|
    \ge 
    \|\tfrac{1}{\sqrt n}\bP_O^\perp\bX\bh\|
    =\|\tfrac{1}{\sqrt n}\bG\bSigma^{1/2}\bh\|
    \ge 
    \phi_{\min}^{1/2}(\bG^\top\bG/n)
    \|\bSigma^{1/2}\bh\|
    .$$
    We again resort to the union bound argument in \cite[Proposition 2.10]{blanchard2011compressed}
    to control
    $\phi_{\min}^{1/2}(\bG^\top\bG/n)$.
    As in the proof of \eqref{RIP_lasso} for a Gaussian matrix
    with $n-|O|$ rows and $|S|$ columns we have
    \cite[Proof of Lemma 4.1]{chen2005condition}
    $$\P\Bigl(
    \phi_{\min}\Bigl(\frac{\bG^\top\bG}{n-|O|}\Bigr)
    \le t^2 
    \Bigr)
    \le \frac{(\frac{et(n-|O|)}{n-|O|-|S|+1})^{n-|O|-|S|+1} }{\sqrt{2\pi(n-|O|-|S|+1)}}
    \le \frac{ (\frac{et}{(1-d_*)+1})^{n(1-d_*)+1} }{\sqrt{2\pi(n(1-d_*)+1)}}
    $$
    thanks to $n - |O|-|S|\ge n(1-d*)$
    and the fact that $u\mapsto(\frac1u)^u$ is decreasing on $[1/e,+\infty)$.
    There are $\binom{n+p}{\lfloor d_*n\rfloor}$ possible pairs $(O,S)$
    with $|O|+|S|=\lfloor d_*n\rfloor$.
    Using $\binom{N}{d}\le \exp[d\log(eN/d)]$ for integers $d\le N$,
    a union bound leads to an extra multiplicative
    factor at most
    $\exp(n d_*\log(e\frac{1+\gamma}{d_*})$
    in the previous display.
    Choosing $t>0$ a small enough constant depending on
    $(\gamma,d_*)$ only, the probability of the union
    over all pairs $(S,O)$ with $|S|+|O| =\lfloor d_*n \rfloor$
    of such events over converge to 0. This completes the proof
    of \eqref{RIP_huber} to obtain $t_3>0$ depending only
    on $(d_*,\gamma,a_*)$.
    Finally, \eqref{RE_huber} is again obtained from
    \eqref{RIP_huber} and \cite[Lemma 2.7]{lecue2014sparse}
    by choosing the constant $k_4\in (0,1)$ small enough and depending
    only on $(t_3,\varphi,d_*,\gamma,a_*)$.
\end{proof}

\begin{proposition}
    \label{prop:Omega_L}
    Let $d_*\in(0,1),\varphi>1,\gamma>1$ be arbitrary constants.
    Assume $p/n\le \gamma$ and
    let $\bX$ have iid rows with distribution $N(\mathbf 0,\bSigma)$
    with $\Sigma_{jj}=1$ for all $j\in[p]$ and
    $\|\bSigma\|_{op}\|\bSigma^{-1}\|_{op}\le \varphi$.
    Assume that the noise $\bep$ has iid $N(0,1)$ entries.
    Then there exist constants $s_*,\lambda_*$ depending only
    on $(d_*,\varphi,\gamma)$
    such that
    if $\|\bbeta\|_0\le s_*n$ and $\hbbeta=\argmin_{\bb\in\R^p}
    \|\bX\bb-\by\|^2/(2n) + \lambda n^{-1/2}\|\bb\|_1$
    with $\lambda\ge \sigma \lambda_*$, there exists
    an open set $\Omega_L\subset \R^n\times \R^{n\times p}$ 
    such that
    $(\bep,\bX)\in\Omega_L\Rightarrow \|\hbbeta\|_0\le d_*n/2$
    and $\P((\bep,\bX)\in\Omega_L)\to 1$ as
    $n,p\to+\infty$ while $(d_*,\varphi,\gamma)$ remain fixed.
    Furthermore, 
    with $\bep,\tbep\in\R^n$ with $\bep=\tbep$,
    $\tbX,\bX\in\R^{n\times p}$,
    and $\hbbeta,\tbbeta,\bpsi,\tbpsi, \br, \tbr$
    the corresponding quantities as in \Cref{lemma:general-C},
    then 
    $\{(\bep,\bX),
    (\bep,\tbX)\}
    \subset\Omega_L$ implies
    \begin{equation}
    (\tfrac1n\|\bpsi - \tbpsi\|^2 +
    \|\bSigma^{\frac12}(\hbbeta-\tbbeta)\|^2)^{1/2}
    \le
    n^{-1/2}
    L_*
    \|(\bX-\tbX)\bSigma^{-\frac12}\|_{op}
    (\tfrac1n\|\bpsi\|^2 + \|\bSigma^{\frac12}(\hbbeta-\bbeta)\|^2)^{1/2}
    \label{eq:lipschitz-condition-lasso-new}
    \end{equation}
    for a constant $L_* = L_*(d_*,\gamma,\varphi)>0$ depending
    only on $(d_*,\gamma,\varphi)$.
\end{proposition}
\begin{proof}
    Let $t_1,t_2,k_2$ be the constants in \eqref{RE_lasso}
    and note that $t_1,t_2,k_2$ depend only on $(d_*,\gamma,\varphi)$.
    Define
    $s_*=s_*(d_*,\gamma,\varphi)$,
    $\lambda_*=\lambda_*(d_*,\gamma,\varphi)>0$
    and $\Omega_L\subset \R^n\times \R^{n\times p}$ by
    \begin{align}
    s_*
    \defas
    \min\Bigl\{
        k_2,
        \frac{d_* t_2^2}{
            8\varphi^2(2+\sqrt\gamma)^2
        }
    \Bigr\},
    \qquad
    \lambda_*^2
    \defas 
    \frac{\varphi(2+\sqrt\gamma)^2 4.04}{d_*},
    \label{s_*_lambda_*_Lasso}
    \\
    \nonumber
    \Omega_L=
    \{(\bep,\bX):
    \|\bep\| < \sigma\sqrt{1.01n};
    \bX\text{ satisfies the events in }\eqref{RIP_lasso},\eqref{RE_lasso};
\|\bX\bSigma^{-1/2}\|_{op} < \sqrt n (2+\sqrt \gamma)\}.
    \end{align}
    Thanks to \eqref{eq:upper-bound-sparsity-hbb} with $\bA=\bX$, $\bz=\bep$ and $p=\bar p$
    we have in $\Omega_L$
    $$
    \|\bbeta\|_0\le s_*n
    ~ \Rightarrow ~
    \|\bbeta\|_0\le k_2n
    ~ \Rightarrow ~
    \Bigl(
        \frac{1}{\kappa^2} \le \frac{\varphi}{t_2^2}
        \quad\text{ and }\quad
    \|\hbbeta\|_0\le \varphi n(2+\sqrt\gamma)^2
    \max\Bigl\{ \frac{2.02\sigma^2}{\lambda^2}, \frac{4 \|\bbeta\|_0\varphi}{nt_2^2}
    \Bigr\}
    \Bigr).
    $$
    where $\kappa^2$ is defined in
    \eqref{eq:upper-bound-sparsity-hbb}. Above, the bound
    $\frac{1}{\kappa^2}\le \frac{\varphi}{t_2^2}$ follows from
    the definition of $\kappa,t_2$ and
    $$
    \Bigl(\|\bbeta\|_1 - \|\hbbeta\|_1\Bigr)
    +
    \|\hbbeta-\bbeta\|_1
    \le 2 \sum_{j\in S}|\hbeta_j - \beta_j| \le 2 \|\bbeta\|_0^{1/2} \bigl(\sum_{j\in S}(\hbeta_j-\beta_j)^2\bigr)^{1/2}
    \le 2 \|\bbeta\|_0^{1/2}\|\hbbeta-\bbeta\|_2.
    $$
    By construction of $s_*,\lambda_*$ in \eqref{s_*_lambda_*_Lasso},
    $\|\bbeta\|_0\le s_*n$ and $\lambda \ge \sigma \lambda_*$
    implies that in the event $\Omega_L$,
    the upper bound on $\|\hbbeta\|_0$
    is smaller than $d_*n/2$ so that
    $\|\hbbeta\|_0\le d_* n/2$ holds.

    The fact that $\P((\bep,\bX\in\Omega_L))\to 1$ follows from
    a standard bound on the deviation of the  $\chi^2_n$ random variable
    $\|\bep\|^2/\sigma^2$, \eqref{RIP_lasso}-\eqref{RE_lasso}
    and \Cref{lemma:Davidson}.

    To prove \eqref{eq:lipschitz-condition-lasso-new},
    if $(\bep,\bX),(\bep,\tbX)\in \Omega_L$ then
    $\|\hbbeta-\tbbeta\|_0\le d_* n$
    so that $\|\bX(\hbbeta-\tbbeta)\|>\sqrt n t_1 \|\bSigma^{1/2}(\tbbeta-\hbbeta)\|$ by \eqref{RIP_lasso}.
    Applying the last part of \Cref{lemma:Xbar}
    to $\bA=\bX$,$\tbA=\tbX$, $\bz=\bep$, $\bSigmabar = n\bSigma$
    and $\tbb=\tbbeta$, $\hbb=\hbbeta$
    we obtain
    $\|\bA(\tbb-\hbb )\|
    \ge t_1 \sqrt n \|\bSigma^{1/2}(\tbb-\hbb)\|
    = t_1 \|\bSigmabar{}^{1/2}(\tbb-\hbb)\|$ to bound from below
    the minimum in the left-hand side of \eqref{eq:multiline-lasso-uberlasso-lipschitz} which gives \eqref{eq:lipschitz-condition-lasso-new}.
\end{proof}

\begin{proof}[Proof of \Cref{thm:main_robust} and \Cref{thm:main-squared-loss}
    for the Lasso, under \Cref{assum:g}(iii.a)]
Let $d_*\in (0,1)$ be any absolute constant in $(0,1)$, e.g., $d_*=0.99$.
We make explicit the change of variable to create a new isotopic
design matrix:
Let 
\begin{equation}
  \bG\defas \bX\bSigma^{-1/2}  ,
  \qquad
  \bw=\bSigma^{1/2}(\hbbeta-\bbeta)
  \label{defw}
\end{equation}
so that $\bG$ has iid $N(0,1)$ entries.
Let $L_*$ and $\Omega_L$ be given by \Cref{prop:Omega_L}.
For $U=\{\bG\in\R^{n\times p}: (\bep,\bG\bSigma^{-1/2})\in\Omega_L\}$
and $\bpsi,\brho:\R^{n\times p}\to \R^n,D:\R^{n\times p}\to \R$ the
functions defined by
\begin{equation}
    \label{eq_bpsi-D-brho-Lasso}
\bpsi(\bG) = \psi(\by-\bX\hbbeta) = \psi(\bep - \bG\bw),
\quad
D(\bG) = (\|\bpsi\|^2/n + \|\bw\|^2)^{1/2},
\quad
\brho(\bG) = n^{-1/2}\bpsi(\bG) / D(\bG)
\end{equation}
we have $\|\brho(\bG)-\brho(\bG')\|
\le 2L_* n^{-1/2}\|\bG-\bG'\|_F$ by
\eqref{eq:lipschitz-br-D-inverse} if $\{\bG,\bG'\}\subset U$.
Applying \Cref{kirsz_cor}
conditionally on $\bep$ to $\brho(\bG)$, to the random matrix $\bG$ and to $L=2L_*$, 
\eqref{Kirszbraun_first} gives that
\begin{equation}
\Rem_L \defas
\|\bG^\top\brho\|^2-p\|\brho\|^2
-\sum_{k=1}^p(\sum_{i=1}^n\frac{\partial \rho_i}{\partial g_{ij}})^2
- 2 \sum_{i=1}^n\sum_{k=1}^P \rho_i \be_k^\top\bG^\top \frac{\partial \brho}{\partial g_{ij}}
\label{Rem_L}
\end{equation}
has $\E|I\{\Omega\}\Rem_L|\le \C(\gamma,d_*,\varphi) \sqrt n$
where $\Omega$ is the event $\Omega=\{(\bep,\bG\bSigma^{1/2})\in\Omega_L\}$.
The derivatives of $\by-\bX\hbbeta$ with respect to $\bX$ and 
for a fixed $\bep$ are given by
\cite[Proposition 4.1]{bellec_zhang2019second_poincare}:
for almost every $(\bep,\bX)$, Frechet differentiability holds at $\bX$
and the derivatives are given (holding $\bep$ fixed) by
$$
\tfrac{\partial}{\partial x_{ij}}(\by-\bX\hbbeta)
= - \bX\hbA\be_j(y_i-\bx_i^\top\hbbeta) - \bV\be_i (\hbeta_j-\beta_j).
$$
where $\hat S = \{j\in[p]: \hbeta_j\ne 0\}$,
$\bX_{\hat S}\in\R^{n\times |\hat S|}$
is the submatrix of $\bX$ made of the columns of $\bX$ indexed in $\hat S$,
and
\begin{equation}
\hbA_{\hat S,\hat S} = (\bX_{\hat S}^\top\bX_{\hat S})^{-1}
\text{ and }
\hbA_{jk} = 0 \text{ if }j\notin \hat S \text{ or } k\notin \hat S,
\qquad
\bV = \bI_n - \bX\hbA\bX^\top
.
\end{equation}
Consequently, by the chain rule using $\bX=\bG\bSigma^{1/2}$, the derivatives of $\bpsi(\bG)$ are given by
\begin{equation}
    \tfrac{\partial}{\partial g_{ik}}\bpsi(\bG) 
    =
    - \bG \bA \be_k \psi_i - \bV \be_i w_k
    \qquad \text{ where } \qquad
    \bA = \bSigma^{1/2}\hbA \bSigma^{1/2}\in\R^{p\times p}
\end{equation}
and where $\bw=(w_k)_{k=1,...,p}$ is the vector in \eqref{defw}.
At this point, the argument and algebra are the same as those of
\eqref{eq:negligible-term-cross}
and \eqref{divergcence-square-to-bound}; using the same argument as
in the discussion surrounding \eqref{eq:negligible-term-cross}-\eqref{divergcence-square-to-bound} we find
with 
$\df=\trace[\bG\bA\bG^\top]=\trace[\bX\hbA\bX^\top]=|\hat S|$ that almost surely
\begin{align}
\Big|2 \sum_{i=1}^n\sum_{k=1}^P \rho_i \be_k^\top\bG^\top \frac{\partial \brho}{\partial g_{ij}}
+ 2 \df\|\brho\|^2
\Big|
&\le 2 \|\bG n^{-1/2}\|_{op} + 2L_* \|\bG n^{-1/2}\|_{op},
\nonumber
\\
\Big|
\sum_{k=1}^p(\sum_{i=1}^n\frac{\partial \rho_i}{\partial g_{ij}})^2
- \trace[\bV]^2 \frac{\|\bw\|^2}{nD^2}
\Big|
&\le
\bigl(\|\bA\|_{op}\|\bG\|_{op} + \frac{L_*}{\sqrt n}\bigr)^2
+ (\|\bA\|_{op}\|\bG\|_{op} + \frac{L_*}{\sqrt n}) 2 \sqrt n.
\label{eq:two-lines_lasso-bounds}
\end{align}
In $\Omega$, thanks to \eqref{RIP_lasso} and $\|\hbbeta\|_0\le d_*n/2$
to bound from above $\|\hbA\|_{op}$ and thanks to
$\|\bX\bSigma^{-1/2}\|_{op}\le\sqrt n (2+\sqrt \gamma)$
in the event $\Omega_L$ in \eqref{s_*_lambda_*_Lasso},
the right-hand sides of the two displayed equations above
are bounded from above by $\C(\gamma,\varphi,d_*)$.
Multiplying by $I\{\Omega\}$ and taking expectation gives
\begin{align}
\E[I\{\Omega\}
|\|\bG^\top\brho\|^2 - (p-2\df)\|\brho\|^2 - \trace[\bV]^2\|\bw\|^2/(nD^2)
|
&\le
\C(\gamma,\varphi,d_*) +
\E[I\{\Omega\}\Rem_L]
\nonumber
\\&\le \C(\gamma,\varphi,d_*) \sqrt n.
\label{eq:final_lasso}
\end{align}
Since $\|\hbbeta\|_0\le d_*n/2$ in $\Omega$, the proof of
\Cref{thm:main_robust} under \Cref{assum:g}(iii.a) for the Lasso
is complete.
The proof of \Cref{thm:main-squared-loss} follows a similar adaptation
of the proof given in \Cref{sec:square_loss}, using
\eqref{Kirszbraun_second} applied to $\brho(\bG),\bG$ and $L=2L_*$
on the one hand
and Proposition~\ref{prop:eta_1} on the other.
\end{proof}

\section{Huber Lasso: Lipschitz conditions}
\label{sec:appendix_huber_lasso_new}

This section provides the necessary lemmas to prove
the main result under \Cref{assum:g}(iii.a) and (iii.b).
Assumption (iii.b) corresponds to
the $\ell_1$ penalty combined with a scaled Huber loss
is used as the loss function:
for tuning parameters $\lambda_H,\lambda$,
\begin{equation}
    \label{eq:huber-lasso-appendix}
\hbbeta=\argmin_{\bb\in\R^p}
\Big(
   \frac 1 n \sum_{i=1}^n \lambda_H^2 \rho_H\Big( \lambda_H^{-1}(y_i-\bx_i^\top\bb)\Big)
   + n^{-1/2} \lambda  \|\bb\|_1
\Big)
\end{equation}
where $\rho_H$ is the Huber loss \eqref{eq:Huber-loss}.
We let $\hat O =\{i\in[n]: \psi'(y_i-\bx_i^\top\hbbeta)=0 \}$ be the set of outliers.

To control the sparsity and number of outliers of the $M$-estimator
with Huber loss and $\ell_1$ penalty \eqref{eq:huber-lasso-appendix},
the following equivalent definition of the estimator will be useful.
    The $M$-estimator $\hbbeta\in\R^p$ is equal to the first $p$ components
    of the solution $(\hbbeta,\hbtheta)\in\R^{p+n}$ of the optimization
    problem
    \begin{equation}
        \label{eq:Huber-lasso-n-p}
    (\hbbeta,\hbtheta) = \argmin_{(\bb,\btheta)\in\R^{p+n}}
    \|\bX\bb + \sqrt n \pb{(\lambda/\lambda_H)} \btheta - \by\|^2/(2n) + n^{-1/2}\lambda (\|\btheta\|_1 
    + \|\bbeta\|_1)
    .
    \end{equation}
    This representation of the Huber Lasso $\hbbeta$ is well known
    in the study of $M$-estimators based on the Huber loss,
    cf. \cite[Section 6]{donoho2016high}
    or \cite{dalalyan2019outlier} and the references therein.
    Since
    \eqref{eq:Huber-lasso-n-p} reduces to a Lasso
    optimization problem in $\R^{p+n}$ with design matrix 
    $[\bX | \sqrt n \pb{(\lambda/\lambda_H)} \bI_n] \in\R^{n\times (p+n)}$ and response $\by$, any Lasso solver can be used to compute the robust
    penalized estimator $\hbbeta$
    and we can use \Cref{lemma:Xbar} to control $\|\hbbeta\|_0+\|\hbbeta\|_0$.
    Note that $a_*=\frac{\lambda}{\lambda_H}$ under \Cref{assum:g}(iii.b)
    so that the design matrix is \eqref{eq:Huber-lasso-n-p}
    is $[\bX|\sqrt n a_* \bI_n]\in\R^{n\times (p+n)}$.

    \Cref{assum:g}(iii.b) requires that
    $(\eps_i)_{i\in[n]\setminus O}$
    are iid $N(0,\sigma^2)$.
    As in \cite{dalalyan2019outlier},
    rewrite $\by$ as
    \begin{equation}
        \label{LM-with-btheta-star-z}
        \by = \bX\bbeta + \sqrt n \pb{a_*}\btheta^* + \bz
        \text{ where }
        z_i=0\text{ for }i\in O\text{ and }
        (z_i)_{i\in [n]\setminus O}=
        (\eps_i)_{i\in [n]\setminus O}\sim^{iid}N(0,\sigma^2)
    \end{equation}
    where $\btheta^*$ is supported on $O\subset[n]$
    with $\|\btheta^*\|_0\le \lfloor s_* n \rfloor- \|\bbeta\|_0$.
    The non-zero components of the unknown vector $\btheta^*$
    represent the contaminated responses and $\btheta^*$ is not independent of $\bz$.
    The sparsity of the unknown
    regression vector in the above linear model with design matrix
    $\bXbar = [\bX|\sqrt n\pb{a_*}\bI_n]$ is
    $\|\bbeta\|_0 + \|\btheta^*\|_0 \le s_*n$,
    and the support of $\hbtheta$ is exactly
    the set of outliers $\hat O=\{i\in [n]: \psi'(y_i-\bx_i^\top\hbbeta)=0\}$.
    \Cref{lemma:Xbar} shows that
    $\|\hbbeta\|_0+\|\hbtheta\|_0$
    can be controlled with high
    probability when $s_*\in (0,1)$
    is a small enough constant and the tuning parameter is large enough,
    and \eqref{RIP_huber}-\eqref{RE_huber} are used to control
    the $\kappa$ constant in \eqref{eq:upper-bound-sparsity-hbb}.

\begin{proposition}
    \label{prop:Omega_H}
    Let $d_*\in(0,1),\varphi>1,\gamma>1,a_*>0$ be arbitrary constants.
    Assume $p/n\le \gamma$ and
    let $\bX$ have iid entries with distribution $N(\mathbf 0,\bSigma)$
    with $\Sigma_{jj}=1$ for all $j\in[p]$ and
    $\|\bSigma\|_{op}\|\bSigma^{-1}\|_{op}\le \varphi$.
    Assume that the noise $\bep$ has iid $N(0,1)$ entries.
    Then there exist constants $s_*,\lambda_*$ depending only
    on $(d_*,\varphi,\gamma,a_*)$
    such that
    if $\|\bbeta\|_0+\|\btheta^*\|_0\le s_*n$ and 
    $(\hbbeta,\hbtheta)$
    is the minimizer of \eqref{eq:Huber-lasso-n-p}
    with $\lambda/\lambda_H=a_*$
    and  $\lambda\ge \pb{\sigma}\lambda_*$, there exists
    an open set $\Omega_{H}\subset \R^n\times \R^{n\times p}$ 
    such that
    $(\bep,\bX)\in\Omega_{H}\Rightarrow \|\hbbeta\|_0+\|\hbtheta\|_0\le d_*n/2$
    and $\P((\bep,\bX)\in\Omega_{H})\to 1$ as
    $n,p\to+\infty$ while $(d_*,\varphi,\gamma,a_*)$ remain fixed.
    Furthermore, 
    with $\bep,\tbep\in\R^n$ with $\bep=\tbep$,
    $\tbX,\bX\in\R^{n\times p}$,
    and $\hbbeta,\tbbeta,\bpsi,\tbpsi, \br, \tbr$
    the corresponding quantities as in \Cref{lemma:general-C},
    then 
    $\{(\bep,\bX),
    (\bep,\tbX)\}
    \subset\Omega_{H}$ implies
    \eqref{eq:lipschitz-condition-lasso-new}
    for a constant $L_*=L_*(d_*,\gamma,\varphi,a_*)>0$ depending
    only on $(d_*,\gamma,\varphi,a_*)$.
\end{proposition}
\begin{proof}
    We need to specify constants $s_*\in(0,1)$ and $\lambda_*>0$. 
    Let $O\subset [n]$ be as in \eqref{LM-with-btheta-star-z}.
    Define
    $$
    \Omega_{H}=
    \{(\bep,\bX):
        \|\bep_{[n]\setminus O}\| < \sqrt{1.01}\sqrt n;
    \bX\text{ satisfies the events in }\eqref{RIP_huber},\eqref{RE_huber};
\|\bX\bSigma^{-1/2}\|_{op} < \sqrt n (2+\sqrt \gamma)\}.
    $$
    Let $t_3,t_4,k_4$ be the constants in \eqref{RIP_huber}-\eqref{RE_huber}. 
    Thanks to \eqref{eq:upper-bound-sparsity-hbb} with $\bA=[\bX\mid\sqrt n a_* \bI_n]\in\R^{n\times (p+n)}$, $\bz=\bep$ and $\bar p=p+n$
    we have in $\Omega_{H}$
    $$
    \|\bbeta\|_0 + \|\btheta^*\|_0\le k_4n
    \quad\Rightarrow\quad
    \|\hbbeta\|_0 + \|\hbtheta\|_0\le \varphi^2n(a_*+2+\sqrt\gamma)^2
    \max\{ \tfrac{2.02\sigma^2}{\lambda^2}, \tfrac{4(\|\bbeta\|_0+\|\btheta^*\|_0)\varphi}{nt_4^2}
    \}.
    $$
    As in the proof of Proposition~\ref{prop:Omega_L},
    we can thus choose $s_*=s_*(\varphi,t_2,\gamma,d_*,a_*)$ small enough
    and $\lambda_*=\lambda_*(\varphi,t_2,\gamma,d_*,a_*)$ large enough
    such that $\|\bbeta\|_0+\|\btheta^*\|_0\le s_*n$ and $\lambda \ge\sigma \lambda_*$
    implies that in the event $\Omega_{H}$,
    the right-hand side of the previous display
    is smaller than $d_*n/2$, i.e., we have
    $\|\hbbeta\|_0+\|\hbtheta\|_0\le d_* n/2$.
    If $(\bep,\bX),(\bep,\tbX)\in \Omega_{H}$ then
    $\|\hbbeta-\tbbeta\|_0 + \|\hbtheta-\tbtheta\|_0\le d_* n$. Applying \Cref{lemma:Xbar}
    to $\bA=[\bX\mid \sqrt n a_* \bI_n]$,$\tbA=[\tbX\mid \sqrt n a_* \bI_n]$, $\bz$ defined in \eqref{LM-with-btheta-star-z}, $\bSigmabar\in\R^{(p+n)\times (p+n)}$ diagonal by block
    with the two blocks $(\bSigma,\bI_n)$, 
    and $\tbb{}^\top=[\tbbeta{}^\top\mid \tbtheta{}^\top]$, $\hbb{}^\top=[\hbbeta{}^\top \mid\hbtheta{}^\top]$
    we obtain
    $\|\bA(\tbb-\hbb )\|
    \ge t_3 \sqrt n (\|\bSigma^{1/2}(\tbb-\hbb)\|^2 + \|\hbtheta-\tbtheta\|^2)^{1/2}
    = t_3 \|\bSigmabar{}^{1/2}(\tbb-\hbb)\|$ to bound from below
    the minimum in the left-hand side of \eqref{eq:multiline-lasso-uberlasso-lipschitz} which gives \eqref{eq:lipschitz-condition-lasso-new}.
\end{proof}

\begin{proof}[Proof of \Cref{thm:main_robust}
    for the Huber Lasso, under \Cref{assum:g}(iii.b)]
Let $d_*\in (0,1)$ be any absolute constant in $(0,1)$, e.g., $d_*=0.99$.
Define $\bG\in\R^{n\times p}$ and $\bw$ by
\eqref{defw}.
Let $L_*$ and $\Omega_H$ be given by \Cref{prop:Omega_H}.
For $U=\{\bG\in\R^{n\times p}: (\bep,\bG\bSigma^{-1/2})\in\Omega_H\}$
and $\bpsi,\brho:\R^{n\times p}\to \R^n,D:\R^{n\times p}\to \R$ the
functions defined by \eqref{eq_bpsi-D-brho-Lasso},
we have $\|\brho(\bG)-\brho(\bG')\|
\le 2L_* n^{-1/2}\|\bG-\bG'\|_F$ by
\eqref{eq:lipschitz-br-D-inverse} if $\{\bG,\bG'\}\subset U$.
Applying \Cref{kirsz_cor}
conditionally on $\bep$ to $\brho(\bG)$, to the random matrix $\bG$ and to $L=2L_*$, 
\eqref{Kirszbraun_first} gives that \eqref{Rem_L}
has $\E|I\{\Omega\}\Rem_L|\le \C(\gamma,d_*,\varphi) \sqrt n$
where $\Omega$ is the event $\Omega=\{(\bep,\bG\bSigma^{1/2})\in\Omega_H\}$.

Using the argument discussed after
\eqref{eq:matrix-intertible-probability-one} that the KKT conditions
of the Huber Lasso hold strictly and some algebra (we omit the details),
the derivatives of $\psi(\by-\bX\hbbeta)$ with respect to $\bX$ and 
for a fixed $\bep$ are given by
$$
\tfrac{\partial}{\partial x_{ij}}\psi(\by-\bX\hbbeta)
= - \bD\bX\hbA\be_j\psi(y_i-\bx_i^\top\hbbeta) - \bV\be_i (\hbeta_j-\beta_j).
$$
where $\bD=\diag(\bpsi')\in\R^{n\times n}$,
$\hat S = \{j\in[p]: \hbeta_j\ne 0\}$,
$\bX_{\hat S}\in\R^{n\times |\hat S|}$
is the submatrix of $\bX$ made of the columns of $\bX$ indexed in $\hat S$,
and
\begin{equation}
\hbA_{\hat S,\hat S} = (\bX_{\hat S}^\top \bD \bX_{\hat S})^{-1}
\text{ and }
\hbA_{jk} = 0 \text{ if }j\notin \hat S \text{ or } k\notin \hat S,
\qquad
\bV = \bD - \bD\bX\hbA\bX^\top\bD
.
\end{equation}
Consequently, by the chain rule using $\bX=\bG\bSigma^{1/2}$, the derivatives of $\bpsi(\bG)$ in \eqref{eq_bpsi-D-brho-Lasso}
are given almost surely by
\begin{equation}
    \tfrac{\partial}{\partial g_{ik}}\bpsi(\bG) 
    =
    - \bD\bG \bA \be_k \psi_i - \bV \be_i w_k
    \qquad \text{ where } \qquad
    \bA = \bSigma^{1/2}\hbA \bSigma^{1/2}\in\R^{p\times p}
\end{equation}
and where $\bw=(w_k)_{k=1,...,p}$ is the vector in \eqref{defw}.
At this point, the argument and algebra are the same as those of
\eqref{eq:negligible-term-cross}
and \eqref{divergcence-square-to-bound}; using the same argument as
in the discussion surrounding \eqref{eq:negligible-term-cross}-\eqref{divergcence-square-to-bound} we find
with 
$\df=\trace[\bG\bA\bG^\top\bD]=\trace[\bX\hbA\bX^\top\bD]=|\hat S|$
(see \Cref{prop:Huber-Lasso}) that almost surely
\eqref{eq:two-lines_lasso-bounds} hold.
In $\Omega$, thanks to \eqref{RIP_huber} and $\|\hbbeta\|_0+\|\hbtheta\|_0\le d_*n/2$
to bound from above $\|\hbA\|_{op}$ and thanks to
$\|\bX\bSigma^{-1/2}\|_{op}\le\sqrt n (2+\sqrt \gamma)$
in the event $\Omega_H$,
the right-hand sides of \eqref{eq:two-lines_lasso-bounds}
are bounded from above by $\C(\gamma,\varphi,d_*)$.
Multiplying by $I\{\Omega\}$ and taking expectation gives
again \eqref{eq:final_lasso} and the proof
of \Cref{thm:main_robust} under \Cref{assum:g}(iii.b) is complete.
\end{proof}

\newpage
\onArxivOrInJournal{}{
\begin{table}[ht]
\small
\quad\qquad\qquad\qquad
\begin{tabular}{@{}|l|c|c|@{}}
\toprule
$u\in$& $[0,1]$ & $[1,\infty)$  \\ \midrule
$\psi_H'(u)$ & $1$  & $0$  \\ \midrule
$\psi_H(u)$ & $u$  & $1$ \\ \midrule
$\rho_H(u)$  & $\frac{u^2}{2}$ & $u-\frac{1}{2}$  \\ \bottomrule
\end{tabular}
\qquad
\begin{tabular}{@{}|l|c|c|c|@{}}
\toprule
$u\in$ & $[0,1]$ & $[1,2]$ & $[2,+\infty)$ \\ \midrule
$\psi_0'(u)$ & $1$  & $2-u$  & $0$ \\ \midrule
$\psi_0(u)$ & $u$  & $ -\frac{1}{2}  + 2u - \frac{u^2}{2}$ & $\frac{3}{2}$  \\ \midrule
$\rho_0(u)$  & $\frac{u^2}{2}$ & $\frac{1}{6} -\frac{u}{2} + u^2 - \frac{u^3}{6}$  &  $\frac{-7}{6} + \frac{3u}{2}$  \\ \bottomrule
\end{tabular}
\newline
\vspace*{0.2cm}
\newline
\begin{tikzpicture}
    \begin{axis}[
    height=1.6in,
    width=2.2in,
        ]
\addplot+[no marks,domain=0:3,samples=301] {  ( x<1 ? 1 : 0 ) };
\addplot+[no marks,domain=0:3,samples=301] {  ( x<1 ? x : 1 ) };
\addplot+[no marks,domain=0:3,samples=301] {  ( x<1 ? x^2/2 : x-1/2 ) };
\end{axis}
\end{tikzpicture}
$~~~~~~~~~~~~~$
\begin{tikzpicture}
    \begin{axis}[
    height=1.6in,
    width=2.2in,
        ]
\addplot+[no marks,domain=0:3,samples=301] {  ( x<1 ? 1 : ( x<2 ? 2-x  : 0) ) };
\addplot+[no marks,domain=0:3,samples=301] {  ( x<1 ? x : ( x<2 ? -1/2 + 2*x -x^2/2 : 3/2) ) };
\addplot+[no marks,domain=0:3,samples=301] {  ( x<1 ? x^2/2 : ( x<2 ? 1/6 - x/2 +x^2 - x^3/6  : -7/6 + 3*x/2) ) };
\end{axis}
\end{tikzpicture}

\caption[Huber loss and its derivatives, as well as its smoothed version $\rho_0$]{Huber loss $\rho_H(u)$ and its derivatives, as well as its smoothed version
    $\rho_0(u)$ and its derivatives.
    In the plots, the loss $\rho$
    is shown in brown,
    $\psi=\rho'$ in red
    and $\psi'$ in blue.
}
\label{table:rho_0}
\end{table}
\begin{table}[ht]
    \centering
\small
\begin{tabular}{@{}|l|c|c|c|@{}}
\toprule
$u\in$ & $[0,1]$ & $[1,2]$ & $[2,+\infty)$ \\ \midrule
$\psi_1'(u)$ & $1$  & $2x^3 - 9x^2 + 12x - 4$  & $0$ \\ \midrule
$\psi_1(u)$ & $u$  & $\frac 3 2  + \frac{(x-2)^3x}{2}$ & $\frac{3}{2}$  \\ \midrule
$\rho_1(u)$  & $\frac{u^2}{2}$ & $\frac{x^5}{10} -\frac{3x^4}{4} + 2x^3-2x^2 + \frac{3x}{2}-\frac{7}{20}$  &  $\frac{37}{20} + \frac{3(u-2)}{2}$  \\ \bottomrule
\end{tabular}
\begin{tikzpicture}
    \begin{axis}[
    height=1.5in,
    width=2.6in,
        ]
        \addplot+[no marks,domain=0:3,samples=301] {  ( x<1 ? 1 : ( x<2 ? 2*x^3 - 9*x^2 + 12*x - 4  : 0) ) };
\addplot+[no marks,domain=0:3,samples=301] {  ( x<1 ? x : ( x<2 ? 3/2 + (x-2)^3*x/2 : 3/2) ) };
\addplot+[no marks,domain=0:3,samples=301] {  ( x<1 ? x^2/2 : ( x<2 ? x^5/10 - 3*x^4/4 + 2*x^3 - 2*x^2 + 3*x/2 - 7/20  : 37/20 + 3*(x-2)/2) ) };
\end{axis}
\end{tikzpicture}

\caption[Smoothed version $\rho_1$ of the Huber loss]{Smooth robust loss $\rho_1(u)$ and its derivatives for $u\ge 0$.
}
\label{table:rho_1}
\end{table}
}




\textbf{Acknowledgments.}
Research partially supported by the NSF Grants DMS-1811976 and 
DMS-1945428.

\bibliographystyle{plainnat}
\bibliography{../../bibliography/db}

\end{document}